\documentclass[11pt,reqno]{amsart}
\usepackage{times}
\usepackage{amsmath,amsfonts,amstext,amssymb,amsbsy,amsopn,amsthm,eucal}
\usepackage{txfonts}
\usepackage{dsfont}
\usepackage{graphicx}   
\usepackage{hyperref}
\usepackage{wrapfig}
\usepackage{tikz}
\usetikzlibrary{decorations.fractals}

\numberwithin{equation}{section}

\hypersetup{
  pdftitle={The Singular Structure and Regularity of Stationary Varifolds},
  pdfauthor={Aaron Naber and Daniele Valtorta},
  pdfsubject={},
  pdfkeywords={},
  pdfpagelayout=SinglePage,
  pdfpagemode=UseOutlines,
  colorlinks,
  bookmarksopen,
  linkcolor=[rgb]{0,0,0.7},
  urlcolor=[rgb]{0,0,0.4},
  citecolor=[rgb]{0.4,0.1,0}
}

\newcommand{\Vol}{\text{Vol}}

\newcommand{\inj}{\text{inj}}
\newcommand{\Sing}{\text{Sing}}

\newcommand{\dN}{\mathds{N}}

\newcommand{\dQ}{\mathds{Q}}
\newcommand{\dR}{\mathds{R}}

\newcommand{\er}{\mathfrak{r}}

\newcommand{\norm}[1]{\left\|#1\right\|}
\newcommand{\ps}[2]{\left\langle#1,#2\right\rangle}
\newcommand{\ton}[1]{\left(#1\right)}
\newcommand{\qua}[1]{\left[#1\right]}
\newcommand{\cur}[1]{\left\{#1\right\}}
\newcommand{\abs}[1]{\left|#1\right|}

\newcommand{\B}[2]{B_{#1}\ton{#2}}

\newcommand{\supp}[1]{\operatorname{supp}\ton{#1}}

\newcommand{\N}{\mathbb{N}}

\newcommand{\R}{\mathbb{R}}

\renewcommand{\paragraph}[1]{\ \newline \ \textbf{#1\ }}

\newcommand{\hol}{H\"older }

\newtheorem{theorem}{Theorem}[section]

\newtheorem{proposition}[theorem]{Proposition}
\newtheorem{lemma}[theorem]{Lemma}

\theoremstyle{definition}
\newtheorem{definition}[theorem]{Definition}
\theoremstyle{remark}
\newtheorem{remark}{Remark}[section]
\theoremstyle{remark}
\newtheorem{example}{Example}[section]
\theoremstyle{remark}

\theoremstyle{remark}

\theoremstyle{remark}

\begin{document}

\title[The Singular Structure and Regularity of Stationary Varifolds]{The Singular Structure and Regularity \\ of Stationary Varifolds}

\author{Aaron Naber and Daniele Valtorta}\thanks{The first author has been supported by NSF grant DMS-1406259, the second author has been supported by SNSF grants 149539 and PZ00P2\_168006}

\begin{abstract}
If one considers an integral varifold $I^m\subseteq M$ with bounded mean curvature, and if $S^k(I)\equiv\{x\in M: \text{ no tangent cone at $x$ is }k+1\text{-symmetric}\}$ is the standard stratification of the singular set, then it is well known that $\dim S^k\leq k$.  In complete generality nothing else is known about the singular sets $S^k(I)$.  In this paper we prove for a general integral varifold with bounded mean curvature, in particular a stationary varifold, that every stratum $S^k(I)$ is $k$-rectifiable.  In fact, we prove for $k$-a.e. point $x\in S^k$ that there exists a unique $k$-plane $V^k$ such that {\it every} tangent cone at $x$ is of the form $V\times C$ for some cone $C$. 

In the case of minimizing hypersurfaces $I^{n-1}\subseteq M^n$ we can go further.  Indeed, we can show that the singular set $S(I)$, which is known to satisfy $\dim S(I)\leq n-8$, is in fact $n-8$ rectifiable with uniformly {\it finite} $n-8$ measure.  An effective version of this allows us to prove that the second fundamental form $A$ has {\it apriori} estimates in $L^7_{weak}$ on $I$, an estimate which is sharp as $|A|$ is not in $L^7$ for the Simons cone.  In fact, we prove the much stronger estimate that the regularity scale $r_I$ has $L^7_{weak}$-estimates.

The above results are in fact just applications of a new class of estimates we prove on the {\it quantitative} stratifications $S^k_{\epsilon,r}$ and $S^k_{\epsilon}\equiv S^k_{\epsilon,0}$.  Roughly, $x\in S^k_{\epsilon}\subseteq I$ if no ball $B_r(x)$ is $\epsilon$-close to being $k+1$-symmetric.  We show that $S^k_\epsilon$ is $k$-rectifiable and satisfies the Minkowski estimate $\Vol(B_r\,S_\epsilon^k)\leq C_\epsilon r^{n-k}$.  The proof requires a new $L^2$-subspace approximation theorem for integral varifolds with bounded mean curvature, and a $W^{1,p}$-Reifenberg type theorem proved by the authors in \cite{NaVa+}.
\end{abstract}

\date{\today}
\maketitle

\tableofcontents

\section{Introduction}

In this paper we will study integral $m$-varifolds $I^m$ with bounded mean curvature, and in particular stationary varifolds and area minimizing currents, see Section \ref{s:Prelim} for an introduction to the basics.  In the case of a varifold $I$ with bounded mean curvature we can consider the stratification of $I$ given by
\begin{align}\label{e:stratification}
S^0(I)\subseteq\cdots\subseteq S^k(I)\subseteq \cdots S^{m}(I)\subseteq I\, , 
\end{align}
where the singular sets are defined as
\begin{align}\label{e:sing_set}
S^k(I)\equiv\{x\in M: \text{ no tangent cone at $x$ is }k+1\text{-symmetric}\}\, ,
\end{align}
see Definition \ref{d:stratification} for a precise definition and more detailed discussion.  An important result of Almgren (see \cite[sections 2.25, 2.26]{almgren_big}) tells us that in the context of area minimizing currents we have the Hausdorff dimension estimate
\begin{align}\label{e:sing_set_haus_est}
\dim S^k(I)\leq k\, .
\end{align}
With some work, this estimate can be carried over to the context of integral varifolds. Unfortunately, essentially nothing else is understood about the structure of the singular sets in any generality.  In the case where $I$ is a codimension $1$ area minimizing current more is understood, at least for the top stratum of the singular set.  
Namely, if $I^{n-1}$ is an area minimizing hypersurface it has been shown in \cite{Simon_mult_1_min} that $S(I)=S^{n-8}(I)$ is rectifiable. This result relies on Simons' work \cite{Simons_MinVar}, where the author proves that all tangent cones to hypersurfaces in $\R^n$ are hyperplanes if $n\leq 7$.

Moreover, for a minimizing hypersurface, i.e., a minimizing current of codimension $1$, one can use the $\epsilon$-regularity theorem of \cite{Fed},\cite{ChNa2} (see theorem \ref{t:eps_reg}) to show that the whole singular set $S(I)$ coincides with $S^{n-8}(I)$. As it is well-known, this is not the case for minimizing currents of codimension higher than one, as singular points in this context may also arise as branching points on the top stratum $S^{m}(I)\setminus S^{m-1}(I)$. \\

The first goal of this paper is to give improved regularity results for stratification and associated quantitative stratification for varifolds under only a bounded mean curvature assumption.  Indeed, we will show for such a varifold that the strata $S^k(I)$ are $k$-rectifiable for every $k$.  In fact, we will show that for $k$-a.e. $x\in S^k(I)$ there exists a unique $k$-dimensional subspace $V\subseteq T_xM$ such that {\it every} tangent cone at $x$ is of the form $V^k\times C$ for some cone $C$.  Note that we are not claiming that the cone factor $C$ is necessarily unique, just that the Euclidean factor $V^k$ of the cone is.  It is a good question as to whether tangent cones need to be unique in general under only the assumption of bounded mean curvature.\\  

For a varifold $I^m=I^{n-1}$ which is an area minimizing hypersurface these results can be improved.  To begin with, we have that the top stratum of the singular set $S(I)=S^{n-8}(I)$ is $n-8$ rectifiable, and that there are {\it a priori} bounds on its $n-8$ Hausdorff measure.  That is,  if the mass $\mu_I(B_2)\leq \Lambda$ is bounded, then we have the estimate $H^{n-8}(S(I)\cap B_1)\leq C(\Lambda,B_2)$.  In fact, we have the stronger Minkowski estimate
\begin{align}
\Vol\ton{B_r\Big(S(I)\Big)\cap B_1},\, r\mu_I\Big(B_r\Big(S(I)\Big)\cap B_1\Big)\leq C r^8\, ,
\end{align}
where $B_r\Big(S(I)\Big)$ is the tube of radius $r$ around the singular set.  Indeed, we can prove much more effective versions of these estimates.  That is, in theorem \ref{t:main_min_weak_L7} we show that the second fundamental form $A$ of $I$, and in fact the regularity scale $r_I$, have {\it a priori} bounds in $L^7_{weak}$.  More precisely, we have that
\begin{align}
\mu_I\ton{\cur{|A|>r^{-1}}\cap B_1}\leq \mu_I\ton{B_r\cur{|A|>r^{-1}}\cap B_1} \leq C(\Lambda,B_2) r^7\, .
\end{align}

Let us observe that these estimates are sharp, in that the Simons' cone satisfies $|A|\not\in L^7_{loc}(I)$.  Let us also point out that this sharpens estimates of \cite{ChNa2}, where it was proven that for minimizing hypersurfaces that $|A|\in L^p$ for all $p<7$.  We refer the reader to Section \ref{ss:minimizing_results} for the precise and most general statements.\\

Now the techniques of this paper all center around the notion of the {\it quantitative} stratification.  In fact, it is for the quantitative stratification that the most important results of the paper hold, everything else can be seen to be corollaries of these statements.  The quantitative stratification was first introduced in \cite{ChNa1}, and later used in \cite{ChNa2} with the goal of giving effective and $L^p$ estimates on stationary and minimizing currents.  It has since been used in \cite{ChNaHa1}, \cite{ChNaHa2}, \cite{ChNaVa}, \cite{FoMaSpa}, \cite{brelamm} to prove similar results in the areas of mean curvature flow, critical sets of elliptic equations, harmonic map flow.  More recently, in \cite{NaVa+} the authors have used ideas similar to those in this paper to prove structural theorems for the singular sets of stationary harmonic maps.

Before describing the results in this paper on the quantitative stratification, let us give more precise definitions of everything.  To begin with, to describe the stratification and quantitative stratification we need to discuss the notion of symmetry associated to an integral varifold with bounded mean curvature.  Specifically:

\begin{definition}
For $y\in \R^n$ and $\lambda>0$, let $\eta_{y,\lambda}:\R^n\to \R^n$ and $\tau_y:\R^n\to \R^n$ be the functions
\begin{gather}
 \eta_{y,\lambda}(x)= y+\frac{x-y}{\lambda}\, , \quad \tau_{y}(x)=x+y\, ,
\end{gather}
and let $\eta_{y,\lambda \ \#}$ and $\tau_{y \#}$ be their pushforwards. We define the following:
\begin{enumerate}
\item An integral varifold $I^m\subseteq \dR^n$ is called $k$-symmetric if $\eta_{0,\lambda\ \#}I=I$ $\forall$ $\lambda>0$, and if there exists a $k$-plane $V^k\subseteq \dR^n$ such that for each $y\in V^k$ we have that $\tau_{y\#}I=I$.
\item Given an integral varifold $I^m\subseteq M$ and $\epsilon>0$, we say a ball $B_r(x)\subseteq M$ with $r<$inj$(x)$ and $x\in I^m$ is $(k,\epsilon)$-symmetric if there exists a $k$-symmetric integral varifold $\tilde I^m\subseteq T_xM$ such that \newline $d(\eta_{0,r \ \#}I^m\cap B_1(0),\tilde I^m\cap B_1(0))<\epsilon r$, where we have used the exponential map to identify $I$ as a varifold on $T_xM$.
\end{enumerate}
\end{definition}
\begin{remark}
The distance $d$ may be taken to be the weak distance induced by the Frechet structure of varifold convergence. In the case of minimizing currents, it is equivalent to take $d$ to be the flat distance.
\end{remark}

Thus, an integral varifold $I^m$ is $k$-symmetric if $I=V^k\times C$ for some cone $C$.  A varifold is $(k,\epsilon)$-symmetric on a ball $B_r(x)$ if it is weakly close to a $k$-symmetric varifold on this ball.\\

With the notion of symmetry in hand, we can define precisely the quantitative stratification associated to a solution.  The idea is to group points together based on the amount of symmetry that balls centered at those points contain.  In fact, there are several variants which will play a role for us.  Let us introduce them all and briefly discuss them:

\begin{definition}\label{d:stratification}
For an integral varifold $I^m$ with bounded mean curvature and finite density we make the following definitions:
\begin{enumerate}
\item For $\epsilon,r>0$ we define the $k^{th}$ $(\epsilon,r)$-stratification $S^k_{\epsilon,r}(I)$ by
\begin{align}
S^k_{\epsilon,r}(I)\equiv\{x\in I\cap B_1:\text{ for no }r\leq s<1\text{ is $B_{s}(x)$ a }(k+1,\epsilon)\text{-symmetric ball}\}.
\end{align}
\item For $\epsilon>0$ we define the $k^{th}$ $\epsilon$-stratification $S^k_{\epsilon}(I)$ by
\begin{align}
S^k_{\epsilon}(I)=\bigcap_{r>0} S^k_{\epsilon,r}(I)\equiv\{x\in I\cap B_1:\text{ for no }0< r<1\text{ is $B_{r}(x)$ a }(k+1,\epsilon)\text{-symmetric ball}\}.
\end{align}
\item We define the $k^{th}$-stratification $S^k(I)$ by
\begin{align}
S^k(I)=\bigcup_{\epsilon>0} S^k_{\epsilon}(I)=\{x\in I\cap B_1:\text{ no tangent cone at $x$ is $k+1$-symmetric}\}.
\end{align}
\end{enumerate}
\end{definition}
\begin{remark}
It is a small but important exercise to check that the standard stratification $S^k(I)$ as defined in \eqref{e:stratification} agrees with the set $\bigcup_{\epsilon>0} S^k_{\epsilon} (I)$.
\end{remark}

Let us discuss in words the meaning of the quantitative stratification, and how it relates to the standard stratification.  As discussed at the beginning of the section, the stratification $S^k(I)$ of $I$ is built by separating points of $I$ based on the infinitesimal symmetries of $I$ at those points.  The quantitative stratifications $S^k_{\epsilon}(I)$ and $S^k_{\epsilon,r}(I)$ are, on the other hand, instead built by separating points of $I$ based on how many symmetries exist on balls of definite size around the points.  In practice, the quantitative stratification has two advantages over the standard stratification.  First, on minimizing hypersurfaces the quantitative stratification allows to prove effective estimates.  In particular, in \cite{ChNa2} the $L^p$ estimates 
\begin{align}\label{e:sing_set_Lp_est}
\fint_{B_1\cap I}|A|^{7-\delta}\, <C_\delta \text{  }\forall\, \delta>0\, ,
\end{align}
on minimizing hypersurfaces were proved by exploiting this fact.  The second advantage is that the estimates on the quantitative stratification are much stronger than those on the standard stratification.  Namely, in \cite{ChNa2} the Hausdorff dimension estimate \eqref{e:sing_set_haus_est} on $S^k(I)$ was improved to the Minkowski content estimate 
\begin{align}\label{e:sing_set_mink_delta_est}
\Vol\ton{B_r\ton{ S^k_{\epsilon,r}}}\leq C_{\epsilon,\delta} r^{n-k-\delta}\text{  }\forall \delta>0\, .
\end{align}
One of the key technical estimates of this paper is that in theorem \ref{t:main_quant_strat_stationary} we drop the $\delta$ from this estimate and obtain an estimate of the form
\begin{align}\label{e:quant_strat_mink_r}
 \Vol\ton{B_r\ton{S^k_{\epsilon,r}}}\leq C_\epsilon r^{n-k}\, .
\end{align}
From this we are able to conclude in theorem \ref{t:main_eps_stationary} an estimate on $S^k_\epsilon$ of the form
\begin{align}\label{e:quant_strat_mink}
\Vol\ton{B_r\ton{S^k_{\epsilon}}}\leq C_\epsilon r^{n-k}\, .
\end{align}
In particular, this estimate allows us to conclude that $S^k_{\epsilon}$ has uniformly finite $k$-dimensional measure.  In fact, the techniques will prove much more for us.  They will show us that $S^k_\epsilon$ is $k$-rectifiable, and that for $k$-a.e. point $x\in S^k_\epsilon$ there is a unique $k$-plane $V^k\subseteq T_xM$ such that {\it every} tangent cone at $x$ is $k$-symmetric with respect to $V$.  By observing that $S^k(I)=\bigcup S^k_\epsilon(I)$, this is what allows us to prove in theorem \ref{t:main_stationary} our main results on the classical stratification.  This decomposition of $S^k$ into the pieces $S^k_\epsilon$ is crucial for the proof.

On the other hand, \eqref{e:quant_strat_mink}, combined with the $\epsilon$-regularity theorems of \cite{Fed},\cite{ChNa2}, allow us to conclude in the minimizing hypersurface case both the weak $L^7$ estimate on $|A|$, and the $m-7$-finiteness of the singular set of $I$.  Thus we will see that theorems \ref{t:main_min_finite_measure} and \ref{t:main_min_weak_L7} are fairly quick consequences of \eqref{e:quant_strat_mink}.\\

Thus we have seen that \eqref{e:quant_strat_mink_r} and \eqref{e:quant_strat_mink}, and more generally theorem \ref{t:main_quant_strat_stationary} and theorem \ref{t:main_eps_stationary}, are the main challenges of the paper.  We will give a more complete outline of the proof in Section \ref{ss:outline_proof}, however let us mention for the moment that two of the new ingredients to the proof are a new $L^2$-subspace approximation theorem for integral varifolds with bounded mean curvature, proved in Section \ref{s:best_approx}, and a $W^{1,p}$-Reifenberg theorem described in Section \ref{s:biLipschitz_reifenberg}.  
The $L^2$-approximation result roughly tells us that the $L^2$-distance of a measure from being contained in a $k$-dimensional subspace may be estimated by integrating the volume drop of the integral varifold over the measure.  To exploit the estimate we prove a new $W^{1,p}$-Reifenberg type theorem.  The classical Reifenberg theorem states that if we have a set $S$ which is $L^\infty$-approximated by an affine $k$-dimensional subspace at every point and scale, then $S$ is bi-\hol to a $k$-dimensional manifold, see theorem \ref{t:classic_reifenberg} for a precise statement.  It is important for us to improve on this bi-\hol estimate, at least enough that we are able to control the $k^{th}$-dimensional measure of the set and prove rectifiability.  In particular, we want to improve the $C^\alpha$-maps to $W^{1,p}$-maps for $p>k$, and we will want to do it using a condition which is integral in nature.  More precisely, we will only require a form of summable $L^2$-closeness of the subset $S$ to the approximating subspaces.  We will see in theorem \ref{t:best_approximation} that by using the $L^2$-subspace approximation theorem that the conditions of this new $W^{1,p}$-Reifenberg are in fact controllable for the quantitative stratifications $S^k_\epsilon$.

\subsection{Results for Varifolds with Bounded Mean Curvature}\label{ss:stationary_results}

We now turn our attention to giving precise statements of the main results of this paper.  In this subsection we focus on those concerning the singular structure of integral varifolds with bounded mean curvature.  Precisely, let $(M^n,g,p)$ be a Riemannian manifold satisfying
\begin{align}\label{e:manifold_bounds}
&|\sec_{B_2(p)}|\leq K,\,\, \inj(B_2(p))\geq K^{-1}\, ,
\end{align}
and let $I^m$ be an integral varifold on $M$ with mean curvature bounded by $H$ on $B_2(p)$.  That is, for every smooth vector field $X$ on $M$ with compact support in $B_2(p)$ we have the estimate
\begin{align}\label{e:mean_curvature_bound}
\big|\delta \mu_I(X)\big|\leq H\int_{I} |X|\,dI\, ,
\end{align}
see Section \ref{ss:int_varifolds} for more on this.  Let us begin by discussing our main theorem for the quantitative stratifications $S^k_{\epsilon,r}(I)$:\\

\begin{theorem}[$(\epsilon,r)$-Stratification of Integral Varifolds with BMC]\label{t:main_quant_strat_stationary}
Let $I^m$ be an integral varifold on $M^n$ satisfying the curvature bound \eqref{e:manifold_bounds}, the mean curvature bound \eqref{e:mean_curvature_bound}, and the mass bound $\mu_I(B_2(p))\leq \Lambda$.  Then for each $\epsilon>0$ there exists $C_\epsilon(n,K,H,\Lambda,\epsilon)$ such that
\begin{align}
\Vol\Big(B_r\Big(S^k_{\epsilon,r}(I)\Big) \cap \B 1 p\Big)\leq C_\epsilon r^{n-k}\, .
\end{align}
\end{theorem}
\vspace{1cm}

When we study the quantitative stratification $S^k_{\epsilon}(I)$ we can refine the above to prove structure theorems on the set itself: \\

\begin{theorem}[$\epsilon$-Stratification of Integral Varifolds with BMC]\label{t:main_eps_stationary}
Let $I^m$ be an integral varifold on $M^n$ satisfying the curvature bound \eqref{e:manifold_bounds}, the mean curvature bound \eqref{e:mean_curvature_bound}, and the mass bound $\mu_I(B_2(p))\leq \Lambda$.   Then for each $\epsilon>0$ there exists $C_\epsilon(n,K,H,\Lambda,\epsilon)$ such that
\begin{align}
\Vol\Big(B_r\big(S^k_{\epsilon}(I)\big)\cap \B 1 p \Big)\leq C_\epsilon r^{n-k}\, .
\end{align}
In particular, we have the $k^{th}$-dimensional Hausdorff measure estimate $\lambda^{k}(S^k_\epsilon(I))\leq C_\epsilon$.  Further, $S^k_\epsilon(I)$ is $k$-rectifiable, and for $k$-a.e. $x\in S^k_\epsilon$ there exists a {\it unique} $k$-plane $V^k\subseteq T_xM$ such that {\it every} tangent cone of $x$ is $k$-symmetric with respect to $V^k$.
\end{theorem}
\vspace{.5cm}

Finally, we end this subsection by stating our main results when it comes to the classical stratification $S^k(I)$.  The following may be proved from the previous theorem in only a few lines given the formula $S^k(I)=\bigcup S^k_\epsilon(I)$:\\

\begin{theorem}[Stratification of Integral Varifolds with BMC]\label{t:main_stationary}
Let $I^m$ be an integral varifold on $M^n$ satisfying the curvature bound \eqref{e:manifold_bounds}, the mean curvature bound \eqref{e:mean_curvature_bound}, and the mass bound $\mu_I(B_2(p))<\infty$.  Then for each $k$ we have that $S^k(I)$ is countably $k$-rectifiable.  Further, for $k$-a.e. $x\in S^k(I)$ there exists a {\it unique} $k$-plane $V^k\subseteq T_xM$ such that {\it every} tangent cone of $x$ is $k$-symmetric with respect to $V^k$.
\end{theorem}
\vspace{1cm}

\subsection{Results for Minimizing Hypersurfaces}\label{ss:minimizing_results}

In this section we restrict ourselves to studying codimension one integral currents $I^{n-1}$ which minimize the area functional on $B_2(p)$ with respect to compact variations.  That is, if $I'$ is another integral current with $\partial I=\partial I'$ in $B_2(p)$ and $\supp{I-I'}\Subset \B 2 p$ , then $\mu_I(B_2)\leq \mu_I'(B_2)$.  One could easily restrict to local minimizers to obtain similar results.  Most of the results of this section follow quickly by combining the quantitative stratification results of Section \ref{ss:stationary_results} with the $\epsilon$-regularity of \cite{Fed},\cite{ChNa2}, see Section \ref{ss:eps_reg} for a review of these points.  \\

Our first estimate is on the singular set $\Sing(I)$ of a minimizing hypersurface.  Recall that $\Sing(I)$ is the set of points where $I$ is not smooth.  \\

\begin{theorem}[Structure of Singular Set]\label{t:main_min_finite_measure}
Let $I^{m}=I^{n-1}$ be a minimizing integral current on $M^n$ with $\partial I \cap \B 2 0 = 0$  satisfying the curvature bound \eqref{e:manifold_bounds} and the mass bound $\mu_I(B_2(p))\leq \Lambda$.  Then $\Sing(I)$ is $m-7$-rectifiable and there exists $C(n,K,\Lambda)$ such that
\begin{align}
\operatorname{Vol}\Big(B_r\big(\Sing(I)\big)\cap B_1(p)\Big)\leq Cr^8\, ,\notag\\
\mu_I\Big(B_r\big(\Sing(I)\big)\cap B_1(p)\Big)\leq Cr^7\, .\label{e:main_min_finite_measure}
\end{align}
In particular, $\lambda^{m-7}(\Sing(I))\leq C$.
\end{theorem}
\vspace{.5 cm}

The above can be extended to effective regularity estimates on $I$.  To state the results in full strength let us recall the notion of the regularity scale associated to a function.  Namely:

\begin{definition}\label{d:regularity_scale}
Let $I^m$ be an integral varifold.  For $x\in I\cap B_1(p)$ we define the regularity scale $r_I(x)$ by
\begin{align}
r_I(x)\equiv \max\cur{0\leq r\leq 1: \sup_{B_r(x)}|A|\leq r^{-1} }\, .
\end{align}
By definition, $r_I(x)\equiv 0$ if $I$ is not $C^2$ in a neighborhood of $x$.
\end{definition}
\begin{remark}
The regularity scale is {\it scale invariant}.  That is, if $r\equiv r_I(x)$ and we rescale $B_r(x)\to B_1(x)$, then on the rescaled ball we will have $|A|\leq 1$ on $B_1(x)$.
\end{remark}
\begin{remark}
We have the easy estimate $|A|(x)\leq r_I(x)^{-1}$.  However, a lower bound on $r_I(x)$ at a point is in principle {\it much} stronger than an upper bound on $|A|(x)$.
\end{remark}
\begin{remark}
Notice that the regularity scale is a Lipschitz function with $|\nabla r_I|\leq 1$.
\end{remark}
\begin{remark}
If $I$ satisfies an elliptic equation, e.g. is a stationary varifold, then we have estimates of the form
\begin{align}
\sup_{B_{r_I/2}(x)} |\nabla^k A|\leq C_k \,r_I(x)^{-(k+1)}\, .
\end{align}
In particular, control on $r_I$ gives control on all higher order derivatives.
\end{remark}
\vspace{.5 cm}

Now let us state our main estimates for minimizing hypersurfaces:\\

\begin{theorem}[Estimates on Minimizing Hypersurfaces]\label{t:main_min_weak_L7}
Let $I^m=I^{n-1}$ be a minimizing integral current on $M^n$ with $\partial I \cap \B 2 0 = 0$ satisfying the curvature bound \eqref{e:manifold_bounds} and the mass bound $\mu_I(B_2(p))\leq \Lambda$.  Then there exists $C(n,K,\Lambda)$ such that
\begin{align}
\Vol\Big(\{x\in B_1(p): |A|> r^{-1}\}\cap \B 1 p \Big)&\leq Cr^8\, , \notag\\
\Vol\Big(\{x\in B_1(p): r_I(x)< r\}\cap \B 1 p  \Big)&\leq Cr^8\, ,\notag\\
\mu_I\Big(\{x\in B_1(p): r_I(x)< r\}\cap \B 1 p  \Big)&\leq Cr^7\, \label{e:main_min_weak_L7:1} .
\end{align}
In particular, both $|A|$ and $r_I^{-1}$ have bounds in $L^7_{weak}\Big(I\cap B_1(p)\Big)$, the space of weakly $L^7$ functions on $I$.
\end{theorem}
\vspace{1cm}



\subsection{Outline of Proofs and Techniques}\label{ss:outline_proof}
 
In this subsection we give a brief outline of the proof of the main theorems.  To describe the new ingredients involved it will be helpful to give a comparison to the proofs of previous results in the area, in particular the dimension estimate \eqref{e:sing_set_haus_est} of Federer and the Minkowski and $L^p$ estimates \eqref{e:sing_set_Lp_est} of \cite{ChNa2}.

Indeed, the starting point for the study of singular sets for solutions of geometric equations typically looks the same, that is, one needs a monotone quantity.  In the case of integral varifolds $I^m$ with bounded mean curvature, we consider the volume density
\begin{align}
\theta_r(x) \equiv r^{-m}\mu_I(B_r(x))\, . 
\end{align}
For simplicity sake let us take $M\equiv \dR^n$ and $I^m$ to be stationary in this discussion, which is really of no loss except for some small technical work.  Then $\frac{d}{dr}\theta_r\geq 0$, so $\theta_r(x)$ is precisely a monotone quantity, and it is independent of $r$ if and only if $I$ is $0$-symmetric, see Section \ref{ss:monotonicity} for more on this.  Interestingly, this is the only information one requires to prove the dimension estimate \eqref{e:sing_set_haus_est}.  Namely, since $\theta_r(x)$ is monotone and bounded, it must converge as $r$ tends to zero.  Therefore, if we consider the sequence of scales $\er_\alpha=2^{-\alpha}$ then we have for each $x$ that
\begin{align}\label{e:convergence}
\lim_{\alpha\to \infty} \Big|\theta_{\er_\alpha}(x)-\theta_{\er_{\alpha+1}}(x)\Big|\to 0\, .
\end{align}
From this one can conclude that {\it every} tangent cone of $I$ is $0$-symmetric.  This fact combined with some very general dimension reduction arguments originating with Federer from geometric measure theory \cite{simon_GMT}, yield the dimension estimate \eqref{e:sing_set_haus_est}.

The improvement in \cite{ChNa2} of the Hausdorff dimension estimate \eqref{e:sing_set_haus_est} to the Minkowski content estimate \eqref{e:sing_set_mink_delta_est}, and from there the $L^p$ estimate of \eqref{e:sing_set_Lp_est}, requires exploiting more about the monotone quantity $\theta_r(x)$ than just that it limits as $r$ tends to zero.  Indeed, an effective version of \eqref{e:convergence} says that for each $\delta>0$ there exists $N(\Lambda,\delta)>0$ such that 
\begin{align}\label{e:convergence_effective}
\Big|\theta_{\er_\alpha}(x)-\theta_{\er_{\alpha+1}}(x)\Big|<\delta\, 
\end{align}
holds for all except for at most $N$ scales $\alpha\in \{\alpha_1,\ldots,\alpha_N\}\subseteq\dN$.  These {\it bad} scales where \eqref{e:convergence_effective} fails may differ from point to point, but the number of such scales is uniformly bounded.  This allows one to conclude that, for all but at most $N(\Lambda,\epsilon)$-scales, $B_{\er_{alpha}}(x)$ is $(0,\epsilon)$-symmetric, see Section \ref{ss:cone_splitting}.  To exploit this information, a new technique other than dimension reduction was required in \cite{ChNa2}.  Indeed, in \cite{ChNa2} the quantitative $0$-symmetry of \eqref{e:convergence_effective} was instead combined with the notion of cone splitting and an energy decomposition in order to conclude the estimates \eqref{e:sing_set_Lp_est},\eqref{e:sing_set_mink_delta_est}.  Since we will use them in this paper, the quantitative $0$-symmetry and cone splitting will be reviewed further in Section \ref{ss:cone_splitting}.\\

Now let us begin to discuss the results of this paper.  The most challenging aspect of this paper is the proof of the estimates on the quantitative stratifications of theorems \ref{t:main_quant_strat_stationary} and \ref{t:main_eps_stationary}, and so we will focus on these in our outline.  Let us first observe that it might be advantageous to replace \eqref{e:convergence_effective} with a version that forces an actual rate of convergence, see for instance \cite{Simon_CylTanUniq}, \cite{Simon_mult_1_min}.  More generally, if one is in a context where an effective version of tangent cone uniqueness can be proved then this may be exploited.  In fact, in the context of critical sets of elliptic equations one can follow exactly this approach, see the authors' work \cite{NaVa} where versions of theorems \ref{t:main_quant_strat_stationary} and \ref{t:main_eps_stationary} were first proved in this context.  However, in the general context of this paper such an approach fails, as tangent cone uniqueness is not available, and potentially not correct.\\

Instead, we will first replace \eqref{e:convergence_effective} with the following relatively simple observation.  Namely, for each $x$ there exists $N(\Lambda,\delta)$ and a finite number of scales $\{\alpha_1,\ldots,\alpha_N\}\subseteq \dN$ such that
\begin{align}\label{e:convergence_effective2}
\sum_{\alpha_{j}< \alpha<\alpha_{j+1}} \Big|\theta_{\er_{alpha}}(x)-\theta_{\er_{alpha+1}}(x)\Big| < \delta\, .
\end{align}
That is, not only does the mass density drop by less than $\delta$ between these scales, but the sum of the all the mass density drops is less than $\delta$ between these scales.\\

Unfortunately, exploiting \eqref{e:convergence_effective2} in order to prove estimates on the singular set turns out to be {\it substantially} harder to use than exploiting either \eqref{e:convergence} or even \eqref{e:convergence_effective}.  In essence, this is because it is not a local assumption in terms of scale, and one needs estimates which can see many scales simultaneously, but which do not require any form of tangent cone uniqueness statements.  Accomplishing this requires four new ideas, two of which have been introduced in the last few years in \cite{ChNa2},\cite{NaVa}, and two of which are new to this paper.  The first point is to replace the study of the singular set with the study of the quantitative singular set, as introduced in \cite{ChNa2} for integral varifolds.  It will be clear from the proofs that there is no direct way to apply the ideas of this paper to the singular set itself without decomposing it into these quantitative pieces.  The sharp estimates themselves also depend on a new covering argument, first introduced by the authors in \cite{NaVa}.  We will only briefly discuss the covering in the outline, but this covering argument has the advantage of giving very sharp packing estimates on sets and exploits well the condition \eqref{e:convergence_effective2}, see Section \ref{s:covering_main} for more on this in the context of this paper.  The disadvantage of the strategy of this covering is that it requires comparing balls of arbitrarily different sizes.  In \cite{NaVa} this was accomplished by proving an effective tangent cone uniqueness statement in the context of critical sets of elliptic equations.  Unfortunately, in the context of this paper it is not clear such a statement exists.\\

Thus, we arrive at discussing the new ideas supplied in this paper.  As discussed, in order to apply the strategy of the covering argument of \cite{NaVa} we need to be able to estimate collections of balls of potentially arbitrarily different radii.  Accomplishing this requires two ingredients, a rectifiable-Reifenberg type theorem, and a new $L^2$-best subspace approximation theorem for integral varifolds with bounded mean curvature, which will allow us to apply the rectifiable-Reifenberg.  Let us discuss these two ingredients separately.\\

We begin by briefly discussing the $W^{1,p}$-Reifenberg and rectifiable-Reifenberg theorems, which are stated in Section \ref{s:biLipschitz_reifenberg}.  Recall that the classical Reifenberg theorem, reviewed in Section \ref{ss:reifenberg}, gives criteria under which a set becomes $C^\alpha$-\hol equivalent to a ball $B_1(0^k)$ in Euclidean space.  In the context of this paper, it is important to improve on this result so that we have gradient and volume control of our set.  
Let us remark that there have been many generalizations of the classical Reifenberg theorem in the literature, see for instance \cite{Toro_reif} \cite{davidtoro}, however those results have hypotheses which are much too strong for the purposes of this paper.  
Instead, we will follow the approach introduced by the authors in \cite{NaVa+} to prove rectifiability and Minkowski-type bounds on the singular sets of harmonic maps. In particular, we need to improve the $C^\alpha$-equivalence to a $W^{1,p}$-equivalence.  This is strictly stronger by Sobolev embedding, and if $p>k$ then this results in volume estimates and a rectifiable structure for the set.  
More generally, we will require a version of the theorem which allows for more degenerate structural behavior, namely a rectifiable-Reifenberg theorem.  In this case, the assumptions will conclude that a set $S$ is rectifiable with volume estimates.  Of course, what is key about this result is that the criteria will be checkable for our quantitative stratifications, thus let us discuss this criteria briefly.  
Roughly, if $S\subseteq B_1(0^n)$ is a subset equipped with the $k$-dimensional Hausdorff measure $\lambda^k$, then let us define the $k$-dimensional distortion of $S$ by
\begin{align}
D^k_S(y,s)\equiv s^{-2}\inf_{L^k}s^{-k}\int_{S\cap B_{s}(x)}d^2(y,L^k)\,d\lambda^k(y)\, ,
\end{align}
where the $\inf$ is over all $k$-dimensional affine subspaces of $\dR^n$.  That is, $D^k$ measures how far $S$ is from being contained in a $k$-dimensional subspace.  Our rectifiable-Reifenberg then requires this be small on $S$ in an integral sense, more precisely that
\begin{align}\label{e:reifenberg_integral_estimate}
&r^{-k}\int_{S\cap B_r(x)}\sum_{\er_\alpha\leq r} D^k(y,\er_\alpha) d\lambda^k(x)<\delta^2\, .
\end{align}
For $\delta$ sufficiently small, the conclusions of the rectifiable-Reifenberg theorem \ref{t:reifenberg_W1p_holes} are that the set $S$ is rectifiable with effective bounds on the $k$-dimensional measure.  Let us remark that one cannot possibly conclude better than rectifiable under this assumption, see for instance the Examples of Section \ref{ss:examples}.\\

Thus, in order to prove the quantitative stratification estimates of theorems \ref{t:main_quant_strat_stationary} and \ref{t:main_eps_stationary}, we will need to verify that the integral conditions \eqref{e:reifenberg_integral_estimate} hold for the quantitative stratifications $S^k_\epsilon(I)$, $S^k_{\epsilon,r}(I)$ on all balls $\B{r}{x}$.  In actuality the proof is more complicated.  We will need to apply a discrete version of the rectifiable-Reifenberg, which will allow us to build an iterative covering in Section \ref{s:covering_main} of the quantitative stratifications, and each of these will satisfy \eqref{e:reifenberg_integral_estimate}.  This will allow us to keep effective track of all the estimates involved.  However, let us for the moment just focus on the main estimates which allows us to turn \eqref{e:reifenberg_integral_estimate} into information about our integral varifolds, without worrying about such details.\\  

Namely, in Section \ref{s:best_approx} we prove a new and very general approximation theorem for integral varifolds with bounded mean curvature.  As always in this outline, let us assume $M\equiv \dR^n$ and that $I^m$ is stationary, the general case is no harder.  Thus we consider an arbitrary measure $\mu$ which is supported on $B_1(0^n)$.  We would like to study how closely the support of $\mu$ can be approximated by a $k$-dimensional affine subspace $L^k\subseteq \dR^n$, in the appropriate sense, and we would like to estimate this distance by using properties of $I$.  Indeed, if we assume that $B_8(0^n)$ is {\it not} $(k+1,\epsilon)$-symmetric with respect to $I$, then for an arbitrary $\mu$ we will prove in theorem \ref{t:best_approximation} that
\begin{align}\label{e:measure_L2_dist}
\inf_{L^k\subseteq \dR^n}\int d^2(x,L^k)\,d\mu\leq C\int |\theta_8(x)-\theta_{1}(x)|\,d\mu \, ,
\end{align}
where $C$ will depend on $\epsilon$, the mass of $I$, and the geometry of $M$.  That is, if $I$ does not have $k+1$ degrees of symmetry, then how closely the support of an arbitrary measure $\mu$ can be approximated by a $k$-dimensional subspace can be estimated by looking at the mass drop of $I$ along $\mu$.  In applications, $\mu$ will be the restriction to $B_1$ of some discrete approximation of the $k$-dimensional Hausdorff measure on $S^k_\epsilon$, and thus the symmetry assumption on $I$ will hold for all balls centered on the support of $\mu$.\\

In practice, applying \eqref{e:measure_L2_dist} to \eqref{e:reifenberg_integral_estimate} is subtle and must be done inductively on scale.  Additionally, in order to prove the effective Hausdorff estimates $\lambda^k(S^k_\epsilon\cap B_r)\leq Cr^{k}$ we will need to use the Covering lemma \ref{l:covering} to break up the quantitative stratification into appropriate pieces, and we will apply the estimates to these.  This decomposition is based on a covering scheme first introduced by the authors in \cite{NaVa}.  Thus for the purposes of our outline, let us assume we have already proved the Hausdorff estimate $\lambda^k(S^k_\epsilon\cap B_r)\leq Cr^{k}$, and use this to be able to apply the rectifiable-Reifenberg in order to conclude the rectifiability of the singular set.  This may feel like a cheat, however it turns out the proof of the Hausdorff estimate will itself be proved by a very similar tactic, though will also require an inductive argument on scale and use the discrete rectifiable-Reifenberg of theorem \ref{t:reifenberg_W1p_discrete} in replace of the rectifiable-Reifenberg of theorem \ref{t:reifenberg_W1p_holes}.\\

Thus let us choose a ball $B_r$ and let $E\equiv \sup_{B_{r}} \theta_r(y)$.  Let us consider the subset $\tilde S^k_{\epsilon}\subseteq S^k_\epsilon\cap B_r$ defined by
\begin{align}
\tilde S^k_{\epsilon}\equiv\{y\in S^k_\epsilon\cap B_r: \theta_0(y)>E-\eta\}\, ,
\end{align}
where $\eta=\eta(n,K,H,\Lambda,\epsilon)$ will be chosen appropriately later.  We will show now that $\tilde S^k_{\epsilon}$ is rectifiable.  Since $\eta$ is fixed and the ball $B_r$ is arbitrary, the rectifiability of all of $S^k_\epsilon$ will follow quickly by an easy covering argument.  Thus, let us estimate \eqref{e:reifenberg_integral_estimate} by plugging in \eqref{e:measure_L2_dist} and the Hausdorff estimate to conclude:
\begin{align}\label{e:outline:2}
r^{-k}\int_{\tilde S^k_\epsilon}&\sum_{\er_\alpha\leq r} D^k(x,\er_\alpha)\,d\lambda^k \notag\\
&= r^{-k}\int_{\tilde S^k_\epsilon}\sum_{\er_\alpha\leq r}\Big(\inf_{L^k}\er_\alpha^{-2-k}\int_{\tilde S^k_\epsilon\cap B_{\er_\alpha}(x)}d^2(y,L^k)d\lambda^k(y) \Big) d\lambda^k(x)\notag\\
&\leq Cr^{-k}\int_{\tilde S^k_\epsilon}\sum_{\er_\alpha\leq r}\Big(\er_\alpha^{-k}\int_{\tilde S^k_\epsilon\cap B_{\er_\alpha}(x)}|\theta_{8\er_\alpha}(y)-\theta_{\er_\alpha}(y)|d\lambda^k(y) \Big) d\lambda^k(x)\notag
\end{align}
\begin{align}
&= Cr^{-k}\sum_{\er_\alpha\leq r}\er_\alpha^{-k}\int_{\tilde S^k_\epsilon}\lambda^k(\tilde S^k_\epsilon\cap B_{\er_\alpha}(y))|\theta_{8\er_\alpha}(y)-\theta_{\er_\alpha}(y)|\,d\lambda^k(y)\, \notag\\
&\leq Cr^{-k}\int_{\tilde S^k_\epsilon\cap B_r(x)}\sum_{\er_\alpha\leq r}|\theta_{8\er_\alpha}(y)-\theta_{\er_\alpha}(y)|\, d\lambda^k(y)\,\notag\\
&\leq Cr^{-k}\int_{\tilde S^k_\epsilon}|\theta_{8r}(y)-\theta_{0}(y)|\, d\lambda^k(y)\notag\\
&\leq Cr^{-k}\lambda_k(\tilde S^k_\epsilon)\cdot \eta\notag\\
&<\delta^2\, ,
\end{align}
where in the last line we have chosen $\eta=\eta(n,K,H,\Lambda,\epsilon)$ so that the estimate is less than the required $\delta$ from the rectifiable-Reifenberg.  Thus we can apply the rectifiable-Reifenberg of theorem \ref{t:reifenberg_W1p_holes} in order to conclude the rectifiability of the set $\tilde S^k_\epsilon$, which in particular proves that $S^k_\epsilon$ is itself rectifiable, as claimed.

\vspace{1.5cm}
  
\section{Preliminaries}\label{s:Prelim}

\subsection{Integral Varifolds and Mean Curvature}\label{ss:int_varifolds}

Let us begin by giving a very brief introduction to integral varifolds, the first variation, and their relationship to mean curvature.  See \cite{delellis_allard,simon_GMT} for a more complete and nice introduction.  We begin with the definition of an integral varifold:

\begin{definition}
Given a smooth Riemannian manifold $(M^n,g)$, consider the set of pairs $(S,\theta)$, where $S$ is a countably $m$-dimensional rectifiable set in $M^n$ and $\theta$ is a positive function locally integrable wrt $\lambda^m_S$. We define the equivalence relation $(S,\theta)\sim(S',\theta')$ by asking that $\lambda^m (S\setminus S') + \lambda^m(S'\setminus S) =0$ and $\theta=\theta'$ $\lambda^m$-a.e. on $S\cap S'$. An $m$-dimensional rectifiable varifold $I^m$ in $(M^n,g)$ is the equivalence class of pairs $(S,\theta)$.  We say that $I$ is an \textit{integral} $m$-rectifiable varifold, or simply integral $m$-varifold, if the function $\theta$ takes values in the positive integers $\lambda^m_S$-a.e.

\end{definition}
\begin{remark}
From now on, with an abuse of notation, we will denote the varifold $I$ simply by one representative in its equivalence class $(S,\theta)$. 
\end{remark}

We associate to $I$ the measure $\mu_I \equiv \theta\, d\lambda^m_S$, where $\lambda^m_S$ is the $m$-dimensional Hausdorff measure restricted to $S$.  Then the total mass of $I$ can be denoted 
\begin{align}
\mu_I(M) = \int_M d\mu_I=\int_M \theta d\lambda^m_S\, .
\end{align}

To define the mean curvature of an integral varifold we begin by recalling the notion of the first variation.  Given a smooth vector field $X$ on $M$ with compact support, let $\phi_t^X$ be the family of diffeomorphisms generated by $X$, then we can define the first variation of $I$ as the distribution
\begin{align}
\delta \mu_I(X) \equiv \frac{d}{dt}\Big|_{t=0} |\phi_{t\#}I| = \int_M \text{div}_{I}X\,d\mu_I\, ,
\end{align}
where $\text{div}_I(X)$ is the divergence of $I$ on the tangent space of $I_S$, which is well defined a.e.  We can now say that $I$ has mean curvature bounded by $H$ (or first variation bounded by $H$) on an open set $U\subseteq M$ if for every smooth vector field whose support is in $U$ we have that
\begin{align}
\big|\delta \mu_I(X)\big|\leq H\int |X|\,d\mu_I\, .
\end{align}
We say that $I$ is a stationary integral varifold on $U$ if the mean curvature is bounded by $0$, and in particular we have that $I$ is then a critical point of the area functional, with respect to variations in $U$.  
\vspace{.5cm}

An important compactness theorem for integral $m$-varifolds due to Allard states that given a sequence of integral varifolds with uniformly bounded mass and mean curvature has a converging subsequence in the same class. Here we state the theorem for the reader's convenience, and refer to \cite{Allard_firstvariation} or \cite[chapter 8]{simon_GMT} for the proof and more details on the subject.
\begin{theorem}[Allard compactness theorem]\label{th_All_cmpt}
 Let $I_j$ be a sequence of integral $m$-varifolds in $(M^n,g)$ with mass bound $\mu_{I_j}(M)\leq \Lambda$ and mean curvature uniformly bounded by $H$. Then, up to passing to a subsequence, $I_j$ converges in the sense of varifolds to some integral $m$-varifold $I$ with $\mu_I(M)\leq \Lambda$ and mean curvature bounded by $H$.
\end{theorem}

\subsection{Minimizing integral currents}
Next, let us consider the class of integral currents. These objects arise naturally in the study of minimal surfaces, here we briefly recall their definition and main properties. We refer the reader to \cite{Fed,morgan_GMT,delellis_areamin} for a more complete introduction on currents.

Integral currents arise naturally in the study of the Plateau's problem. The lack of compactness in the family of classical manifolds with respect to their volume makes it natural to introduce a sort of ``weak version'' of these objects in order to apply classical variational methods to prove the existence of minimizers. In this spirit, De Rham defines the currents as duals of smooth forms in a domain.
\begin{definition}
Let $m \leq n$. Given a domain $\Omega\subseteq M^n$, let $\Lambda^m(\Omega)$ be the space of smooth compactly supported $m$-forms on $\Omega$ with the strong topology. We denote by $\norm{\lambda}$ the comass of the form $\lambda$, i.e.
\begin{gather}
 \norm{\lambda}_c= \max\cur{\abs{\ps{\lambda(p)}{v_1\wedge\cdots \wedge v_m}} \ \ \text{with} \ \ p\in \Omega \ \ \text{and} \ \ \abs{v_1\wedge\cdots v_m}=1}\, ,
\end{gather}
where $\abs{v_1\wedge\cdots v_m}$ is the $m$-dimensional Hausdorff measure in $T_p(M)$ of the parallelogram determined by the tangent vectors $v_1,\cdots, v_m$.

An $m$-dimensional current $I$ is a continuous linear functional $I:\Lambda^m\to \R$.
\end{definition}
By integration, it is evident that any smooth $m$-dimensional orientable submanifold can be viewed as a current. The boundary $\partial I$ is defined (when it exists) as the only $m-1$ dimensional current such that for all $\lambda\in \Lambda^{m-1}$ the integration by parts holds:
\begin{gather}
 I(d \lambda)=\partial I (\lambda)\, .
\end{gather}
It is natural to associate to each current $I$ and open set $A\subseteq \Omega$ a mass $\norm{I}(A)$ by setting
\begin{gather}
 \norm{I}(A)=\sup\cur{ \frac{\abs{I(\lambda)}}{\norm{\lambda}_c}\, \ s.t. \ \ \lambda \in \Lambda^m\, , \ \lambda\neq 0\, , \ \operatorname{supp}(\lambda)\subset A}\, ,
\end{gather}
where $\norm{\lambda}_c$ is the comass of $\lambda$.

\begin{definition}
 We say that an $m$-dimensional current $I$ on $\Omega\subset \R^{m+n}$ is an \textit{integer current} if $\norm{I}(\Omega)<\infty$ and there exists a sequence of $C^1$ oriented $m$-dimensional submanifolds $M_i\subset \R^{m+n}$, a sequence of pairwise disjoint closed sets $K_i\subset M_i$ and a sequence of integers $k_i$ such that
  \begin{gather}
   I(\lambda) = \sum_i k_i \int_{K_i} \lambda
  \end{gather}
The current $I$ is said to be an \textit{integral current} if both $I$ and $\partial I$ are integer currents.
\end{definition}
It is clear that we can naturally associate a varifold $V_I$ to each integral current $I$. Note that the mass of the current $\norm{I}(A)$ coincides with the mass in the varifold sense $\mu_{V_I}(A)$.

We say that an integral current $I$ is \textit{minimizing} if it minimizes the mass among all other integral currents with the same boundary, in particular 
\begin{definition}
 Given an integral current $I$ on $\Omega$, we say that $I$ is area-minimizing in $\Omega$ if $\norm{I}(A)\leq \norm{J}(A)$ for all integral currents $J$ such that $\partial I=\partial J$ in $\Omega$ and $A\equiv \supp{I-J}$ is a compact subset of $\Omega$.
\end{definition}
It is worth mentioning that integral currents enjoy an important compactness property with respect to their mass. Indeed, given a sequence $I_i$ of integral currents with $\norm{I_i}(\Omega)+\norm{\partial I_i}(\Omega)<C<\infty$, there exists a converging subsequence, where the convergence is intended in the weak-$*$ topology with respect to $\Lambda^m(\Omega)$. This compactness property, proved by Federer and Flemming in \cite{fedfle}, immediately implies the existence of minimizers of the generalized Plateau problem, in the sense that
\begin{theorem}
 Given an $m-1$ integer rectifiable currents $J$ such that there exists an integral current $I$ with $\partial I=J$ and $\norm{I}<\infty$, there exists an integral current $I'$ minimizing the norm $\norm{I'}$ of all currents with $\partial I'=J$.
\end{theorem}

\subsection{Bounded Mean Curvature and Monotonicity}\label{ss:monotonicity}

For the purposes of this paper, the most important property of an integral varifold with bounded mean curvature is the existence of a monotone quantity at each point.  For simplicity let us first consider the case of a stationary integral varifold $I^m$ in $\dR^n$.  Then for $x\in I$ and $r>0$ we can consider the density function
\begin{align}
\theta_r(x) = r^{-m}\mu_I\big(B_r(x)\big)\, .
\end{align}
Then we have that $\theta_r(x)$ is monotone increasing and for $0<s\leq r$
\begin{align}
 \theta_r(x)-\theta_s(x) = \int_{B_r(x)\setminus \overline{B_s(x)} } d_x^{-m}\langle N_y I, n_x\rangle^2\,d\mu_I(y)\, ,
\end{align}
 where $N_yI$ is the orthogonal complement of the tangent space of $I$ at $y$, which is defined $\mu_I$-a.e., $d_x(y)=|x-y|$ is the distance function to $x$, and $n_x(y)=\frac{y-x}{|y-x|}$ is the normal vector field from $x$.  
From this it is easy to check that $\theta_r(x)$ is independent of $r$ iff $I$ is $0$-symmetric, see \cite[section 5]{Allard_firstvariation} or \cite{Almgren_InteriorRegMin}.  More generally, if $\theta_r(x)=\theta_s(x)$ then $I$ is $0$-symmetric on the annulus $A_{s,r}(x)$.  In Section \ref{ss:cone_splitting} we will recall a quantitative version of this statement introduced in \cite{ChNa2}.  Motivated by this, we see that what we are really interested in is the amount the density drops from one scale to the next, and thus we define for $0<s\leq r$
\begin{align}
W_{s,r}(x) \equiv \theta_r(x)-\theta_s(x)\geq 0\, .
\end{align}
Often times we will want to enumerate our choice of scale, so let us define the scales $\er_\alpha\equiv 2^{-\alpha}$ for $\alpha\geq 0$, and the corresponding mass density drop:
\begin{align}
W_{\alpha}(x) \equiv W_{\er_{\alpha},\er_{\alpha-3}}(x)\equiv \theta_{\er_{\alpha-3}}(x)-\theta_{\er_{\alpha}}(x)\geq 0\, .
\end{align}

From this one can prove that at every point, every tangent cone is $0$-symmetric, which is the starting point for the dimension estimate \eqref{e:sing_set_haus_est} of Federer.  In Section \ref{ss:cone_splitting} we discuss quantitative versions of this point, first introduced in \cite{ChNa2} and used in this paper as well, and also generalizations which involve higher degrees of symmetry.  These points were first used in \cite{ChNa2} to prove Minkowski estimates on the quantitative stratification of a stationary varifold.  They will also play a role in our arguments, though in a different manner.\\

In the general case when $M\neq \dR^n$ and/or the mean curvature is only bounded, essentially the same statements may be made, however $\theta_r(x)$ is now only {\it almost} monotone, meaning that $e^{Cr}\theta_r(x)$ is monotone for some constant $C=C(n,K,H)$ which depends only on the geometry of $M$ and the mean curvature bound $H$, see for example \cite{delellis_allard,ColdMin_min}. In particular, for the almost monotonicity of the normalized volume see \cite[pag 234]{ColdMin_min}, where the authors carry out the computations for two dimensional surfaces using the Hessian comparison theorem.

\subsection{\texorpdfstring{Quantitative $0$-Symmetry and Cone Splitting}{Quantitative 0-Symmetry and Cone Splitting}}\label{ss:cone_splitting}

In this subsection we review some of the quantitative symmetry and splitting results of \cite{ChNa2}, in particular those which will play a role in this paper.

The first result we will discuss acts as an {\it effective} formulation of the fact that every tangent cone is $0$-symmetric.  Namely, the quantitative $0$-symmetry of \cite{ChNa2} says that for each $\epsilon>0$ and point, that away from a finite number of scales every ball looks $(0,\epsilon)$-symmetric.  To be precise, let us consider the radii $\er_\alpha\equiv 2^{-\alpha}$, and then the statement is the following:\\

\begin{theorem}[Quantitative $0$-Symmetry \cite{ChNa2}]\label{t:quantitative_0_symmetry}
Let $I^m$ be an integral varifold on $M^n$ satisfying the curvature bound \eqref{e:manifold_bounds}, the mean curvature bound \eqref{e:mean_curvature_bound}, and the mass bound $\mu_I(B_2(p))\leq \Lambda$.  Then for each $\epsilon>0$ the following hold:
\begin{enumerate}
\item There exists $\delta(n,\Lambda,K,H,\epsilon)>0$ such that for each $x\in B_1(p)$ and $0<r\leq r(n,\Lambda,K,H,\epsilon)$, if we have $|\theta_{r}(x)-\theta_{\delta r}(x)|<\delta$, then $B_r(x)$ is $(0,\epsilon)$-symmetric.
\item For each $x\in B_1(p)$ there exists a finite number of scales $\{\alpha_1,\ldots,\alpha_N\}\subseteq \dN$ with $N\leq N(n,\Lambda,K,H,\epsilon)$ such that for $r\not\in (\er_{\alpha_j}/2,2\er_{\alpha_j})$ we have that $B_{r}(x)$ is $(0,\epsilon)$-symmetric.
\end{enumerate}
\end{theorem}
\begin{remark}
In \cite{ChNa2} the result is stated for a stationary varifold, however the verbatim (quick) proof holds just as well for integral varifolds with bounded mean curvature. For the reader's convenience, here we give a sketch of it.
\end{remark}
\begin{proof}
 Assume for simplicity that the ambient space is Euclidean, i.e., $K=0$, and $p=0$. Consider by contradiction a sequence of varifolds $I_i$ in $\B 2 0$ with uniformly bounded mass and mean curvature, and consider a sequence of balls $\B {r_i}{x_i}$ with $x_i\in \B 1 0$ such that $r_i\leq i^{-1}$ and $\theta(x_i,r_i)-\theta(x_i,i^{-1}r_i)\leq i^{-1}$ but $\B {r_i}{x_i}$ is not $(0,\epsilon)$-symmetric, for some $\epsilon>0$ fixed.
 
 After rescaling $\B {r_i}{x_i}\to \B 1 0$, we obtain a sequence $\tilde I_i$ of varifolds of bounded mass with mean curvature $H_i\to 0$ and such that $\theta(0,1)-\theta(0,i^{-1})\leq i^{-1}$. Allard compactness theorem ensures that $\tilde I_i$ converges weakly to some $\tilde I$ with bounded mass. Given the pinching condition on $\theta$, we obtain that for all $r>0$:
 \begin{gather}
  \lim_{i\to \infty} \int_{\B 1 0 \setminus \B r 0}d_x^{-m}\langle N_y \tilde I_i, n_x\rangle^2\,d\mu_{\tilde I_i}(y) =0\, ,
 \end{gather}
thus $\tilde I$ is a $0$-symmetric varifold. Since $\tilde I_i$ converges weakly to $\tilde I$, we arrived at a contradiction.

A mildly more technical but morally verbatim argument works when the ambient space is not Euclidean but has sectional curvature bounds.
\end{proof}

Another technical tool that played an important role in \cite{ChNa2} was that of cone splitting.  This will be used in this paper when proving the existence of unique tangent planes of symmetry for the singular set, so we will discuss it here.  In short, cone splitting is the idea that multiple $0$-symmetries add up to give rise to a $k$-symmetry.  To state it precisely let us give a careful definition of the notion of independence of a collection of points:

\begin{definition}\label{d:independent_points}
We say a collection of points $\{x_1,\ldots,x_\ell\}\subset \dR^n$ is independent if they are linearly independent.  We say the collection is $\tau$-independent at $x$ if $d(x_{k+1},\text{span}\{x_1,\ldots,x_k\})>\tau$ for each $k$.  If $\{x_1,\ldots,x_\ell\}\subset M$ then we say the collection is $\tau$-independent with respect to $x$ if $d(x,x_j)<\inj(x)$ and the collection is $\tau$-independent when written in exponential coordinates at $x$.
\end{definition}

Now we are in a position to state the effective cone splitting of \cite{ChNa2}:

\begin{theorem}[Cone Splitting \cite{ChNa2}]\label{t:con_splitting}
Let $I^m$ be an integral varifold on $M^n$ satisfying the curvature bound \eqref{e:manifold_bounds}, the mean curvature bound \eqref{e:mean_curvature_bound}, the mass bound $\mu_I(B_3(p))\leq \Lambda$, and let $\epsilon,\tau >0$ be fixed.  Then there exists $\delta(n,\Lambda,\epsilon,\tau)>0$ such that if $K+H<\delta$ and $x_1,\ldots,x_k\in B_1(p)$ are such that
\begin{enumerate}
\item $B_2(x_j)$ are $(0,\delta)$-symmetric, and $\B 2 p$ is $(0,\delta)$-symmetric
\item $\{x_1,\ldots,x_k\}$ are $\tau$-independent at $p$,
\end{enumerate}
then $B_1(p)$ is $(k,\epsilon)$-symmetric.
\end{theorem}
\begin{remark}
The assumption $K+H<\delta$ is of little consequence, since this just means focusing the estimates on balls of sufficiently small radius after rescaling.
\end{remark}
\begin{proof}
 The proof can be carried out with a simple compactness argument similar to the one used in the proof of theorem \ref{t:quantitative_0_symmetry}.
\end{proof}

\vspace{.5 cm}

We end with the following, which one can view as a quantitative form of dimension reduction.

\begin{theorem}[Quantitative Dimension Reduction]\label{t:quant_dim_red}
Let $I^m$ be an integral varifold on $M^n$ satisfying the curvature bound \eqref{e:manifold_bounds}, the mean curvature bound \eqref{e:mean_curvature_bound}, and the mass bound $\mu_I(B_2(p))\leq \Lambda$.  Then for each $\epsilon>0$ there exists $\delta(n,\Lambda,\epsilon), r(n,\Lambda,\epsilon)>0$ such that if $K+H<\delta$ and $B_2(p)$ is $(k,\delta)$-symmetric with respect to some $k$-plane $V^k$, then for each $x\in B_1(p)\setminus B_\epsilon(V^k)$ we have that $B_r(x)$ is $(k+1,\epsilon)$-symmetric.
\end{theorem}
The proof of this theorem is standard and it follows from theorem \ref{t:quantitative_0_symmetry} and a contradiction argument.
\vspace{.5 cm}

\subsection{\texorpdfstring{$\epsilon$-regularity for Minimizing Hypersurfaces}{epsilon-regularity for Minimizing Hypersurfaces}}\label{ss:eps_reg}

In this subsection we quickly review the $\epsilon$-regularity theorem of \cite{ChNa2}, which itself follows quickly from the difficult work of \cite{Simons_MinVar},\cite{Fed} after an easy contradiction argument.  This will be our primary technical tool in upgrading the structural results on varifolds with bounded mean curvature to the regularity results for area minimizing hypersurfaces.  The main theorem of this subsection is the following:

\begin{theorem}[$\epsilon$-Regularity \cite{Simons_MinVar},\cite{Fed},\cite{ChNa2}]\label{t:eps_reg}
Let $I^{n-1}\subseteq B_2$ be a minimizing hypersurface with $\partial I \cap \B 2 0 = 0$ satisfying the bounds \eqref{e:manifold_bounds}, and the mass bound $\mu_I(B_2(p))\leq \Lambda$.  Then there exists $\epsilon(n,\Lambda)>0$ such that if $K<\epsilon$ and $B_2(p)$ is $(n-7,\epsilon)$-symmetric, then $$r_I(p)\geq 1\, .$$
\end{theorem}
\begin{remark}
The assumption $K<\epsilon$ is of little consequence, since this just means focusing the estimates on balls of sufficiently small radius after rescaling. Note that since $I$ is a minimizing hypersurface, we know that $H=0$ in this case.
\end{remark}
\vspace{.5cm}

\subsection{Hausdorff, Minkowski, and packing Content}\label{ss:haus_mink}

In this subsection we give a brief review of the notions of Hausdorff, Minkowski, and packing content.  We will also use this to recall the definition of Hausdorff measure.  The results of this subsection are completely standard, but this gives us an opportunity to introduce some notation for the paper.  For a more detailed reference, we refer the reader to \cite{mattila,Fed}. Let us begin with the notions of content:

\begin{definition}\label{d:content}
Given a set $S\subseteq \dR^n$ and $r>0$ we define the following:
\begin{enumerate}
\item The $k$-dimensional Hausdorff $r$-content of $S$ is given by
\begin{align}
\lambda^k_r(S)\equiv \inf\Big\{\sum \omega_k r_i^k : S\subseteq \bigcup E_i \text{ and } \operatorname{diam}(E_i)\leq 2r\Big\}\, .
\end{align}
\item The $k$-dimensional Minkowski $r$-content of $S$ is given by
\begin{align}
m^k_r(S)\equiv (2r)^{k-n}\Vol \ton{\B r S}\, .
\end{align}
\item The $k$-dimensional packing $r$-content of $S$ is given by
\begin{align}
p^k_r(S)\equiv \sup\Big\{\sum \omega_k r_i^k : \, x_i\in S\text{ and }\{B_{r_i}(x_i)\} \text{ are disjoint}\text{ and }r_i\leq r\Big\}\, .
\end{align}
\end{enumerate}
\end{definition}

These definitions make sense for any $k\in [0,\infty)$, though in this paper we will be particularly interested in integer valued $k$.  Notice that if $S$ is a compact set then $\lambda^k_r(S),m^k_r(S)<\infty$ for any $r>0$, and that we always have the relations
\begin{align}
\lambda^k_r(S) \lesssim m^k_r(S) \lesssim p^k_r(S)\, .
\end{align}
In particular, bounding the Hausdorff content is less powerful than bounding the Minkowski content, which is itself less powerful than bounding the packing content.  

Primarily in this paper we will be mostly interested in content estimates, because these are the most effective estimates.  However, since it is classical, let us go ahead and use the Hausdorff content to define a measure.  To accomplish this, let us more generally observe that if $r\leq r'$ then $\lambda^k_r(S)\geq \lambda^{k}_{r'}(S)$. In particular, we can define the limit
\begin{align}
\lambda^k_0(S)\equiv \lim_{r\to 0} \lambda^k_r(S)=\sup_{r>0}\lambda^k_r(S)\, .\notag
\end{align}

It turns out that $\lambda^k_0$ is a genuine measure.  

\begin{definition}\label{d:haus_meas}
Given a set $S\subseteq \dR^n$ we define its $k$-dimensional Hausdorff measure by $\lambda^k(S)\equiv \lambda^k_0(S)$.
\end{definition}

Similar constructions can be carried out for the Minkowski and packing content. In particular, we can define 
\begin{gather}
 \overline{m}^k_0 (S)\equiv \limsup_{r\to 0} m^k_r(S)\, ,\quad \underline{m}^k_0 (S)\equiv \liminf_{r\to 0} m^k_r(S)\, ,\\
 p^k_0(S)=\lim_{r\to 0}p^k_r(S)=\inf_{r>0} p^k_r(S)\, . 
\end{gather}

\begin{definition}\label{d:dimension}
Given a set $S\subseteq \dR^n$ we define its Hausdorff and Minkowski dimension (or box-dimension) by
\begin{align}
\dim_H S\equiv \inf\Big\{k\geq 0: \lambda^k_0(S)=0\Big\}\, ,\notag\\
\dim_M S\equiv \inf\Big\{k\geq 0: \overline m^k_0(S)=0\Big\}\, .
\end{align}
\end{definition}
\begin{remark}
 Note that we could define an upper and lower Minkowski dimension by
 \begin{gather}
 \overline{\dim}_M S\equiv \inf\Big\{k\geq 0: \overline m^k_0(S)=0\Big\}\, ,\quad  \underline{\dim}_M S\equiv \inf\Big\{k\geq 0: \underline m^k_0(S)=0\Big\}\, .
 \end{gather}
In general, $\underline{\dim}_M S\leq \overline{\dim}_M S$, where the inequality may be strict. However, for the purposes of this paper we will only be interested in the \textit{upper} Minkowski dimension.
\end{remark}

As an easy example consider the rationals $\dQ^n\subseteq \dR^n$.  Then it is a worthwhile exercise to check that $\dim_H \dQ^n = 0$, while $\dim_M \dQ^n = n$.  \\

A very important notion related to measures is the \textit{density} at a point. Although this is standard, for completeness we briefly recall the definition of Hausdorff density, and refer the reader to \cite[chapter 6]{mattila} for more on this subject.

\begin{definition}
 Given a set $S\subset \R^n$ which is $\lambda^k$-measurable, and $x\in \R^n$, we define the $k$-dimensional upper and lower density of $S$ at $x$ by
 \begin{gather}
  \theta^{\star k}(S,x)  = \limsup_{r\to 0} \frac{\lambda^k(S\cap \B{r}{x})}{\omega_k r^k}\, ,\quad 
  \theta^{k}_\star (S,x) = \liminf_{r\to 0} \frac{\lambda^k(S\cap \B{r}{x})}{\omega_k r^k}\, .
 \end{gather}
\end{definition}
In the following, we will use the fact that for almost any point in a set with finite $\lambda^k$-measure, the density is bounded from above and below.
\begin{proposition}[ \cite{mattila}]\label{prop_dens}
 Let $S\subset \R^n$ be a set with $\lambda^k(S)<\infty$. Then for $k$-a.e. $x\in S$:
 \begin{gather}
  2^{-k}\leq \theta^{\star k}(S,x)\leq 1\, ,
 \end{gather}
while for $k$-a.e. $x\in \R^n \setminus S$
\begin{gather}
 \theta^{\star k}(S,x)=0\, .
\end{gather}
\end{proposition}
\vspace{.5 cm}

\subsection{Examples}\label{ss:examples}

In this subsection we present a few examples which motivate the sharpness of our results.  

\subsubsection{The Simons Cone and Sharp Estimates}

In order to study the sharpness of the estimates of theorem \ref{t:main_min_weak_L7} we study some examples.  One may use either the Simons cone or the Lawson cone for this analysis.  The Simons cone $C\subseteq \dR^8$ is a cone over the surface
\begin{align}
S^3\Big(\frac{1}{\sqrt{2}}\Big)\times S^3\Big(\frac{1}{\sqrt{2}}\Big)\subseteq S^7\, .
\end{align}
It has been shown \cite{BGG_simonscone} that $C$ is an area minimizing cone.  It is easy to check that for $x\in C$ we have that $|A|(x)=\sqrt 6 |x|^{-1}$, where $A$ is the second fundamental form (see \cite[remark B.3]{simon_GMT}).  In particular, we get that $|A|\in L^7_{weak}$, but $|A|\not\in L^7_{loc}$, showing that the estimates of theorem \ref{t:main_min_weak_L7} are sharp.

\subsubsection{Rectifiable-Reifenberg Example I}

Let us begin with an easy example, which shows that the rectifiable conclusions of theorem \ref{t:reifenberg_W1p_holes} is sharp.  That is, one cannot hope for better structural results under the hypothesis.  Indeed, consider any $k$-dimensional subspace $V^k\subseteq \dR^n$, and let $S\subseteq V^k\cap B_2(0^n)$ be an arbitrary measurable subset.  Then clearly $D(x,r)\equiv 0$ for each $x$ and $r>0$, and thus the hypotheses of theorem \ref{t:reifenberg_W1p_holes} are satisfied, however $S$ clearly need not be better than rectifiable.  In the next example we shall see that $S$ need not even come from a single rectifiable chart, as it does in this example.\\

\subsubsection{Rectifiable-Reifenberg Example II}\label{sss:reifenberg}

With respect to the conclusions of theorem \ref{t:reifenberg_W1p_holes} there are two natural questions regarding how sharp they are.  First, is it possible to obtain more structure from the set $S$ than rectifiable?  In particular, in theorem \ref{t:reifenberg_W1p} there are topological conclusions about the set, is it possible to make such conclusions in the context of theorem \ref{t:reifenberg_W1p_holes}?  In the last example we saw this is not the case.  Then a second question is to ask whether we can at least find a single rectifiable chart which covers the whole set $S$.  
This example taken from \cite[counterexample 12.4]{davidtoro} shows that the answer to this question is negative as well. \\ 

To build our examples let us first consider a unit circle $S^1\subseteq \dR^3$.  Let $M^2\supset S^1$ be a smooth M\"{o}bius strip around this circle, and let $S_\epsilon\subseteq M^2\cap B_\epsilon(S^1)\equiv M^2_\epsilon$ be an arbitrary $\lambda^2$-measurable subset of the M\"{o}bius strip, contained in a small neighborhood of the $S^1$.  In particular, $Area(S_\epsilon)\leq C\epsilon\to 0$ as $\epsilon\to 0$.  It is not hard, though potentially a little tedious, to check that assumptions of theorem \ref{t:reifenberg_W1p_holes} hold for $\delta\to 0$ as $\epsilon\to 0$.\\

However, we have learned two points from these example.  First, since $S_\epsilon$ was an arbitrary measurable subset of a two dimensional manifold, we have that it is $2$-rectifiable, however that is the most which may be said of $S_\epsilon$.  That is, structurally speaking we cannot hope to say better than $2$-rectifiable about the set $S_\epsilon$.  More than that, since $S_\epsilon$ is a subset of the M\"{o}bius strip, we see that even though $S_\epsilon$ is rectifiable, we cannot even cover $S_\epsilon$ by a single chart from $B_1(0^2)$, as a M\"{o}bius strip is not orientable. See \cite{davidtoro} for more on this.\\

\vspace{.5cm}

\subsection{The Classical Reifenberg theorem}\label{ss:reifenberg}

In this Section we recall the classical Reifenberg theorem, as well as some more recent generalizations.  The Reifenberg theorem gives criteria on a closed subset $S\subseteq B_2\subseteq \dR^n$ which determine when $S\cap B_1$ is bi-H\"older to a ball $B_1(0^k)$ in a smaller dimensional Euclidean space.  The criteria itself is based on the existence of {\it good} best approximating subspaces at each scale.  We start by recalling the Hausdorff distance.
\begin{definition}\label{d:haus_dist}
Given two sets $A,B\subseteq \R^n$, we define the Hausdorff distance between these two by
\begin{gather}
 d_H(A,B)=\inf \cur{r\geq 0 \ \ s.t. \ \ A\subset \B{r}{B} \ \ \text{and} \ \ B\subset \B{r}{A} }\, .
\end{gather}
Recall that $d_H$ is a distance on closed sets, meaning that $d_H(A,B)=0$ implies $\overline A = \overline B$.
\end{definition}
\vspace{.5cm}

The classical Reifenberg theorem says the following:\\

\begin{theorem}[Reifenberg theorem \cite{reif_orig,simon_reif}]\label{t:classic_reifenberg}
For each $0<\alpha<1$ and $\epsilon>0$ there exists $\delta(n,\alpha,\epsilon)>0$ such that the following holds.  Assume $0^n\in S\subseteq B_2\subseteq \dR^n$ is a closed subset, and that for each $x\in S\cap \B 1 0 $ and $r\in (0,1]$ we have
\begin{align}\label{e:reifenberg_L_infty_orig}
\inf_{L^k} d_H\big( S\cap B_r(x),L^k\cap B_r(x)\big)<\delta\, r\, ,
\end{align}
where the $\inf$ is taken over all $k$-dimensional affine subspaces $L^k\subseteq \dR^n$.  Then there exists $\phi:B_1(0^k)\to S$ which is a $C^\alpha$ bi-H\"older homeomorphism onto its image with $[\phi]_{C^\alpha},[\phi^{-1}]_{C^\alpha}<1+\epsilon$ and $S\cap B_1\subseteq \phi(B_1(0^k))$.

\end{theorem}

\begin{remark}\label{rem:classic_reifenberg_+}
 In fact, one can prove a little more. In particular, under the hypothesis of the previous theorem, there exists a closed subset $S'\subset \R^n$ such that $S'\cap \B 1 0 = S \cap \B 1 0$ and which is homeomorphic to a $k$-dimensional subspace $0^n\in T_0\subseteq \R^n$ via the $C^\alpha$ bi-H\"older homeomorphism $\phi:T_0\to S'$. Moreover, $\abs{\phi(x)-x}\leq C(n)\delta$ for all $x\in T_0$ and $\phi(x)=x$ for all $x\in T_0\setminus \B 2 0$.
\end{remark}

One can paraphrase the above to say that if $S$ can be well approximated on every ball by a subspace in the $L^\infty$-sense, then $S$ must be bi-\hol to a ball in Euclidean space.\\

Let us also mention that there are several more recent generalizations of the classic Reifenberg theorem.  In \cite{Toro_reif}, the author proves a strengthened version of \eqref{e:reifenberg_L_infty_orig} that allows one to improve bi-\hol to bi-Lipschitz.  Unfortunately, for the applications of this paper the hypotheses of \cite{Toro_reif} are much too restrictive.  We will require a weaker condition than in \cite{Toro_reif}, which is more integral in nature, see theorem \ref{t:reifenberg_W1p}.  In exchange, we will improve the bi-\hol of the classical Reifenberg to $W^{1,p}$.  \\

We will also need a version of the classical Reifenberg which only assumes that the subset $S$ is contained near a subspace, not conversely that the subspace is also contained near $S$.  In exchange, we will only conclude the set is rectifiable.  A result in this direction was first proved in \cite{davidtoro}, but again the hypotheses are too restrictive for the applications of this paper, and additionally there is a topological assumption necessary for the results of \cite{davidtoro}, which is not reasonable in the context in this paper.  We will see how to appropriately drop this assumption in theorem \ref{t:reifenberg_W1p_holes}. \\  

\section{\texorpdfstring{The $W^{1,p}$-Reifenberg theorem}{The W(1,p)-Reifenberg theorem}}\label{s:biLipschitz_reifenberg}

In this Section we recall some Reifenberg-type theorems first introduced in \cite{NaVa+}.  In \cite{NaVa+} we focused our attention on proving the rectifiable-Reifenberg of theorem \ref{t:reifenberg_W1p_holes}, and in this section we will focus our attention on proving the discrete Reifenberg of theorem \ref{t:reifenberg_W1p_discrete}.  The proofs of the two results are very similar.\\

\subsection{\texorpdfstring{Statement of Main $W^{1,p}$-Reifenberg and rectifiable-Reifenberg Results}{Statement of Main W 1p - Reifenberg and rectifiable-Reifenberg Results}}\label{ss:reifenberg_statements}


Before turning to the statements of the theorems, let us introduce some definitions in order to keep the statements as clean and intuitive as possible.



\begin{definition}
Let $\mu$ be a measure on $B_2$ with $r>0$ and $k\in \dN$.  Then we define the $k$-dimensional displacement by
\begin{align}\label{eq_deph_D}
D^{k}_\mu(x,r)\equiv \inf_{L^k}r^{-(k+2)}\int_{B_{r}(x)}d^2(y,L^k)\,d\mu(y)\, ,
\end{align}
if $\mu(B_r(x))\geq\epsilon_n r^k\equiv (1000n)^{-7n^2}  r^k$, and $D^{k}_\mu(x,r)\equiv 0$ otherwise, where the $\inf$'s are taken over all $k$-dimensional affine subspaces $L^k\subseteq \dR^n$.  If $S\subseteq B_2$, then we can define its $k$-displacement $D^{k}_S(x,r)$ by associating to $S$ the $k$-dimensional Hausdorff measure $\lambda^k_S$ restricted to $S$. 

Sometimes, we will omit the index $k$ and the subscript $\mu$ or $S$ when there can be no risk of confusion from the context. In particular, we will often write $D(x,r)$ for $D^{k}_\mu (x,r)$.
\end{definition}

\begin{remark}
One can replace $\epsilon_n$ by any smaller lower bound, and the proofs and statements will all continue to hold.
\end{remark}

\begin{remark}
Notice that the definitions are scale invariant.  In particular, if we rescale $B_r\to B_1$ and let $\tilde S$ be the induced set, then $D^{k}_S(x,r)\to D^{k}_{\tilde S}(x,1)$.
\end{remark}
\begin{remark}\label{rem_mon}
Notice the monotonicity given by the following:  If $\mu'\leq \mu$, then $D^k_{\mu'}(x,r)\leq D^k_\mu(x,r)$.
\end{remark}


\begin{remark}\label{r:D_scale_control}
 It is immediate to see from the definition that, up to dimensional constants, $D^k_\mu(x,r)$ is controlled on both sides by $D^k_\mu(x,r/2)$ and $D^k_\mu(x,2r)$.  In particular, if $\mu(B_{r}(x))\geq \gamma_k r^k=\omega_k 40^{-k} r^k>>\epsilon_n r^k$, then for all $y\in \B r x$,  $D^k_\mu(x,r)\leq 2^{k+2} D^k_\mu(y,2r)$. As a corollary we have the estimate
 \begin{align}\label{eq_estDint}
  &D^k_\mu(x,r)\leq 2^{k+2} \fint_{\B r x } D^k_\mu(y,2r)d\mu(y)\, .
 \end{align}
\end{remark}

\vspace{.5cm}

Before introducing the results which are really needed for the paper, it is worth mentioning the $W^{1,p}$-Reifenberg theorem obtained in \cite{NaVa+}. This is a natural generalization of the Reifenberg and gives intuition and motivation for the rest of the statements, which are essentially more complicated versions of it.  We do not prove the theorem in this paper, instead we refer the reader to \cite[theorem 3.2]{NaVa+}.

\begin{theorem}[$W^{1,p}$-Reifenberg]\cite[theorem 3.2]{NaVa+}\label{t:reifenberg_W1p}
For each $\epsilon>0$ and $p\in [1,\infty)$ there exists $\delta(n,\epsilon,p)>0$ such that the following holds.  Let $S\subseteq B_4\subseteq \dR^n$ be a closed subset with $0^n\in S$, and assume for each $x\in S\cap B_1$ and $B_r(x)\subseteq B_4$ that
\begin{align}
&\inf_{L^k} d_H\big( S\cap B_r(x),L^k\cap B_r(x)\big)<\delta r \, , \label{e:reifenberg_L_infty} \\
&\int_{S\cap B_r(x)}\,\ton{\int_0^r D^k_S(y,s)\,\frac{ds}{s}}\, d\lambda^k(y)<\delta^2r^{k}\, . \label{e:reifenberg_excess_L2}
\end{align}
Then the following hold:
\begin{enumerate}
\item there exists a mapping $\phi:\R^k\to \R^n$ which is a $1+\epsilon$ bi-$W^{1,p}$ map onto its image and such that $S\cap \B 1 {0^n} = \phi(B_1(0^k))\cap \B 1 {0^n}$.
\item $S\cap B_1(0^n)$ is countably $k$-rectifiable.
\item For each ball $B_r(x)\subseteq B_1$ with $x\in S$ we have 
\begin{gather}\label{eq_lambda_lower_upper}
(1-\epsilon)\omega_k r^k\leq \lambda^k(S\cap B_r(x))\leq (1+\epsilon)\omega_k r^k\, .
\end{gather}
\end{enumerate}
\end{theorem}
\begin{remark}
Results $(2)$ and $(3)$ both follow from $(1)$.  
We get $(3)$ by applying the result of $(1)$ to all smaller balls $B_r(x)\subseteq B_1$, since the assumptions of the theorem hold on these balls as well.  
\end{remark}
\begin{remark}
Note that, for $p>k$, a bi-$W^{1,p}$ map is a bi-$C^\alpha$ map, in particular we see that $\phi(B_1(0^k))$ is homeomorphic to the ball $B_1(0^k)$. 
\end{remark}

\begin{remark}
 As it is easily seen, the requirement that $S$ is closed is essential for this theorem, and in particular for the lower bound on the Hausdorff measure. As an example, consider any set $S\subseteq \R^k$ which is dense but has zero Hausdorff measure. In the following theorems, we will not be concerned with lower bounds on the measure, and we will be able to drop the closed assumption.
\end{remark}

\vspace{.5cm}

We are going to state another generalization of Reifenberg's theorem, more discrete in nature, which will be particularly important in the proof of the main theorems of this paper:\\

\begin{theorem}[Discrete Reifenberg]\cite[theorem 3.4]{NaVa+}\label{t:reifenberg_W1p_discrete}
There exists $\delta(n)>0$ and $D(n)$ such that the following holds.  Let $\{B_{r_s}(x_s)\}_{s\in S}\subseteq B_2$ be a collection of disjoint balls with $x_s\in \B 1 0$, and let $\mu\equiv \sum_{s\in S}\omega_k r^k_s \delta_{x_s}$ be the associated measure.  Assume that for each $B_r(x)\subseteq B_4$ with $\mu(\B{r}{x})\geq \gamma_k  r^k=\omega_k (r/40)^k$ we have
\begin{align}\label{e:reifenberg_L_2_discrete}
&\int_{B_r(x)}\ton{\int_0^r D^k_\mu(y,t)\,{\frac{dt}{t}}}\, d\mu(y)<\delta^2 r^{k}\, .
\end{align}
Then we have the estimate 
\begin{align}
\sum_{s\in S} r_s^k<D(n)\, .
\end{align}
\end{theorem}
\begin{remark}
Instead of \eqref{e:reifenberg_L_2_discrete} we may assume the estimate
\begin{align}
\sum_{\er_\alpha\leq r/2}\int_{B_r(x)}D^k_\mu(y,\er_\alpha)\, d\mu(y)<\delta^2 r^k\, .
\end{align}
In the applications, this will be the more convenient phrasing.
\end{remark}

In order to prove rectifiability of the strata, we will also need the following version of Reifenberg's theorem. The proof of this theorem relies on the same ideas as the discrete-Reifenberg, for this reason we do not report it here and we refer the interested reader to \cite[theorem 3.3]{NaVa+}.\\

\begin{theorem}[Rectifiable-Reifenberg]\cite[theorem 3.3]{NaVa+}\label{t:reifenberg_W1p_holes}
For every $\epsilon>0$, there exists $\delta(n,\epsilon)>0$ such that the following holds.  Let $S\subseteq B_4\subseteq \dR^n$ be a $\lambda^k$-measurable subset, and assume for each $B_r(x)\subseteq B_4$  with $\lambda^k(S\cap B_r(x))\geq \gamma_k r^k$ that
\begin{align}\label{e:reifenberg_displacement_L2}
\int_{S\cap B_r(x)}\,\ton{\int_0^r D^k_S(y,s)\,\frac{ds}{s}}\, d\lambda^k(y)<&\delta^2r^{k}\, . 
\end{align}
Then the following holds:
\begin{enumerate}
\item For each ball $B_r(x)\subseteq B_1$ with $x\in S$ we have 
\begin{gather}\label{eq_lambda_obj}
\lambda^k(S\cap B_r(x))\leq (1+\epsilon)\omega_k r^k \, .
\end{gather}
\item $S\cap B_1(0^n)$ is countably $k$-rectifiable.
\end{enumerate}
\end{theorem}
\begin{remark}
Notice that for the statement of the theorem we do not need control over balls which already have small measure.  This will be quite convenient for the applications.
\end{remark}
\begin{remark}
Instead of \eqref{e:reifenberg_displacement_L2} we may assume the essentially equivalent estimate
\begin{align}\label{eq_sum2^alpha}
\sum_{\er_\alpha\leq r/2}\int_{S\cap B_r(x)}D^k_S(y,\er_\alpha)\, d\lambda^k(y)<\delta^2 r^k\, .
\end{align}
In the applications, this will be the more convenient phrasing.
\end{remark}
\vspace{1cm}

\subsection{Explanatory example}
In order to understand better the idea behind the improvement of the Reifenberg theorem, we use the famous snow-flake as a test case.

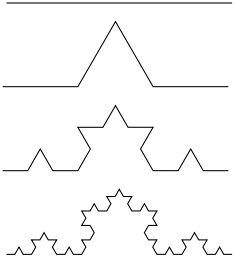
\begin{wrapfigure}{l}{37mm}
\begin{center}
 \begin{tikzpicture}[]
  \draw decorate{
       (0,0) - +(0:3)  };     
\end{tikzpicture}

\vspace{2mm}
\begin{tikzpicture}[decoration=Koch snowflake]
  \draw decorate{
       (0,0) -- ++(0:3)  };     
\end{tikzpicture}
\vspace{2mm}

\begin{tikzpicture}[decoration=Koch snowflake]
  \draw decorate{ decorate{
       (0,0) -- ++(0:3)  }};
\end{tikzpicture}
\vspace{2mm}

\begin{tikzpicture}[decoration=Koch snowflake]
  \draw decorate{decorate{ decorate{
       (0,0) -- ++(0:3)  }}};
\end{tikzpicture}
\end{center}
\end{wrapfigure}

\noindent The construction of a snowflake of parameter $\eta>0$ is well known (see for example \cite[section 4.13]{mattila}). Take the unit segment $[0,1]\times\{0\}\subseteq \R^2$, and replace the middle part $[1/3,2/3]\times\{0\}$ with the top part of the isosceles triangle with base $[1/3,2/3]\times\{0\}$ and of height $\eta \cdot \operatorname{lenght}([1/3,2/3]\times \{0\})$. In other words, you are replacing the segment $[1/3,2/3]\times\{0\}$ with the two segments joining $(1/3,0)$ to $(1/2,\eta/3)$, and $(1/2,\eta/3)$ to $(2/3,0)$. Then repeat this construction inductively on each of the $4$ straight segments in the new set. Here on the left hand side you can see the very classical picture of the first three steps in the construction of the standard snowflake, with $\eta=\sqrt 3 /2 $.

It is clear that the length of the curve at step $i$ is equal to the length at step $i-1$ times $2/3+\sqrt{1+\eta^2}/3$, so the length of the snowflake will be infinity for any $\eta>0$. This is a simple application of the Pythagorean theorem, and the extra square power on $\eta$ comes from the fact that at each step we are adding some length $\eta$ to the curve, but in a direction perpendicular to it.

However, if we replace the fixed parameter $\eta$ with a variable parameter $\eta_i$, we see immediately that the length of the limit curve will be finite if and only if $\sum \eta_i^2<\infty$.

This suggests that the finiteness of the Hausdorff measure of the set $S$ is related to the summability properties of $D^k_S(x,\er_\alpha)$ over scales. 

Indeed, if we introduce the following $L^\infty$ analogue of the distortion $D$:
\begin{gather}
 \bar D^k_S(x,r)\equiv r^{-2}\inf_{L^k} \qua{d_H\ton{S\cap \B r x ,L^k\cap \B r x}}^2\,
\end{gather}
and require that the sum over scales of $D(x,r)$ is small in the sense that
\begin{gather}
\sum_{\alpha=0}^\infty \sup_{x\in S\cap \B 1 0} D^k_S(x,\er_\alpha) <\delta^2\, ,
\end{gather}
then we obtain that $S$ is a bi-Lipschitz image of a $k$-dimensional disk. This result was proved by Toro in \cite{Toro_reif}. 

However, one can also lower these requirements and replace the $L^\infty$ norms with more flexible $L^2$ norms and integrals, and still obtain finiteness of the $k$-dimensional Hausdorff measure and rectifiability under the less restrictive assumptions of theorem \ref{t:reifenberg_W1p_holes}. Moreover, by a simple covering argument, it is not necessary to ask control over $d_H\ton{S\cap \B r x ,L^k\cap \B r x}$, but just on $d(x, L^k)$ for $x\in S$. In other words, it is only important to have control over how close $S$ is to some $k$-dimensional subspace, and not vice-versa. Thus it is not a problem to have ``holes'' in the set $S$.

As mentioned above, similar results to the ones stated in this section were proved in \cite{davidtoro}, but still the theorems proved there have slightly stronger assumptions than the ones used here, and on the other hand obtain stronger topological results than the ones needed for this paper.

\section{Technical Constructions toward New Reifenberg Results}\label{ss:best_comparison}
In this section, we prove some technical lemmas needed for dealing with the relation between best $L^2$ subspaces.  These elementary results will be used in many of the estimates of subsequent sections.
\subsection{Hausdorff distance and subspaces}
We start by recalling some standard facts about affine subspaces in $\R^n$ and Hausdorff distance.

\begin{definition}
 Given two linear subspaces $L,V\subseteq \R^n$, we define the Grassmannian distance between these two as
 \begin{gather}
  d_G(L,V)= d_H(L\cap \B 1 0, V\cap \B 1 0 )=d_H\ton{L\cap \overline{\B 1 0}, V\cap \overline{\B 1 0} }\, .
 \end{gather}
Note that if $\dim(L)\neq \dim(V)$, then $d_G(L,V)=1$.
\end{definition}
\vspace{.3 cm}

For general subsets in $\R^n$, it is evident that $A\subseteq \B{\delta}{B}$ does not imply $B\subseteq \B{c\delta}{A}$. However, if $A$ and $B$ are affine spaces with the same dimension, then it is not difficult to see that this property holds.  More precisely:

 \begin{lemma}\label{lemma_hdv}
 Let $V,\, W$ be two $k$-dimensional affine subspaces in $\R^n$, and suppose that $V\cap \B {1/2}{0}\neq \emptyset$. There exists a constant $c(k,n)$ such that if $V\cap \B 1 0\subseteq \B{\delta}{W\cap \B 1 0}$, then $W\cap \B 1 0\subseteq \B{c\delta}{V\cap \B 1 0}$. Thus in particular $d_H(V\cap \B 1 0,W\cap \B 1 0)\leq c\delta$.
\end{lemma}
\begin{proof}
 The proof relies on the fact that $V$ and $W$ have the same dimension. Let $x_0\in V$ be the point of minimal distance from the origin. By assumption, we have that $\norm{x_0}\leq 1/2$. Let $x_1,\cdots,x_k\in V\cap \overline{\B 1 0}$ be a sequence of points such that
 \begin{gather}
  \norm{x_i-x_0}= 1/2\, \quad \text{ and for }\, i\neq j\, ,  \quad \ps{x_i-x_0}{x_j-x_0}=0\, .
 \end{gather}
In other words, $\cur{x_i-x_0}_{i=1}^k$ is an affine base for $V$. Let $\cur{y_i}_{i=0}^k\subseteq W\cap \overline{\B 1 0}$ be such that $d(x_i,y_i)\leq \delta$. Then
 \begin{gather}
  \norm{y_i-y_0}\geq 1/2-2\delta\, \quad \text{ and for }\, i\neq j\, ,  \quad \abs{\ps{y_i-y_0}{y_j-y_0}}\leq 4\delta+4\delta^2\, .
 \end{gather}
This implies that for $\delta\leq\delta_0(n)$, $\cur{y_i-y_0}_{i=1}^k$ is an affine base for $W$ and for all $y\in W$
\begin{gather}
 y=y_0+ \sum_{i=1}^k \alpha_i (y_i-y_0)\, , \quad \abs{\alpha_i}\leq 10 \norm{y-y_0}\, .
\end{gather}
Now let $y\in W\cap \overline{\B 1 0}$ be the point of maximum distance from $V$, and let $\pi$ be the projection onto $V$ and $\pi^\perp$ the projection onto $V^\perp$, which is the linear subspace orthogonal to $V$. Then
\begin{gather}
 d(y,V)= d(y,\pi(y))=\norm{\pi^\perp (y-x_0)}\leq \norm{\pi^{\perp}(y_0-x_0) }+\sum_{i=1}^k \abs{\alpha_i}\norm{\pi^\perp (y_i-y_0)}\leq c'(n,k)\delta\, .
\end{gather}
Since $y\in \overline{\B 1 0}$ and $\norm{x_0}\leq 1/2$, by a simple geometric argument $\pi(y)\in V \cap \B{1+c'\delta}{0}$, and thus $d(y,V\cap B_1(0))\leq 2c'\delta\equiv c\delta$. This proves the claim. 
 \end{proof}
\vspace{.5 cm}
Next we will see that the Grassmannian distance between two subspaces is enough to control the projections with respect to these planes. In order to do so, we recall a standard estimate.
\begin{lemma}
 Let $V,W$ be linear subspaces of a Hilbert space. Then $d_G(V,W)=d_G\ton{V^\perp,W^\perp}$.
\end{lemma}
\begin{proof}
We will prove that $d_G\ton{V^\perp,W^\perp}\leq d_G\ton{V,W}$. By symmetry, this is sufficient. 

Take $x\in V^\perp$ such that $\norm{x}=1$, and consider that $d(x,W^\perp)=\norm{\pi_W(x)}$. Let $z=\pi_W(x)$ and $y=\pi_V(z)$. We want to show that if $d_G(V,W)\leq \epsilon<1$, then $\norm{z}\leq \epsilon$. We can limit our study to the space spanned by $x,y,z$, and assume wlog that $x=(1,0,0)$, $y=(0,b,0)$ and $z=(a,b,c)$. By orthogonality between $z$ and $z-x$, we have
\begin{gather}
 a^2+b^2+c^2 +(1-a)^2+b^2+c^2=1 \, \quad \Longrightarrow \quad a=a^2+b^2+c^2 \, ,
\end{gather}
and since $z\in W$, we also have $\norm{z-y}\leq \epsilon \norm{z}$, which implies
\begin{gather}
 a^2+c^2\leq \epsilon^2 \ton{a^2+b^2+c^2} \, \quad \Longrightarrow \quad a^2+c^2 \leq \frac{\epsilon^2}{1-\epsilon^2} b^2\, .
\end{gather}
Since the function $f(x)=x^2/(1-x^2)$ is monotone increasing for $0\leq x <1$, we can define $0\leq \alpha <1$ in such a way that
\begin{gather}
 a^2+c^2 = \frac{\alpha^2}{1-\alpha^2} b^2\, , \quad a=a^2+b^2+c^2 = \frac{1}{1-\alpha^2} b^2\, .
\end{gather}
Note that necessarily we will have $\alpha\leq \epsilon$. Now we have
\begin{gather}
 \frac{1}{(1-\alpha^2)^2} b^4 =a^2\leq \frac{\alpha^2}{1-\alpha^2} b^2 \quad \Longrightarrow \quad b^2\leq \alpha^2\ton{1-\alpha^2} \quad \Longrightarrow \quad \norm z^2 = a^2+b^2+c^2\leq \alpha^2\leq \epsilon^2\, .
\end{gather}
This proves that $V^\perp\cap\B 1 0 \subset \B{\epsilon}{W^\perp}$. In a similar way, one proves the opposite direction.
\end{proof}

As a corollary, we prove that the Grassmannian distance $d_G(V,W)$ is equivalent to the distance given by $\norm{\pi_V-\pi_W}$.
\begin{lemma}\label{lemma_epsproj}
Let $V,W$ be linear subspaces of $\R^n$. Then for every $x\in \R^n$, 
 \begin{gather}
  \norm{\pi_V(x)-\pi_W(x)}\leq 2d_G(V,W)\norm x\, .
 \end{gather}
In particular, if $x\in W^\perp$, then $\norm{\pi_V(x)}\leq 2d_G(V,W)\norm x$.

Conversely, we have
\begin{gather}
 d_G(V,W)\leq \sup_{x\in \R^n\setminus \{0\}} \cur{\frac{\norm{\pi_V(x)-\pi_W(x)}}{\norm x}}\, .
\end{gather}
\end{lemma}
 \begin{proof}
The proof is just a corollary of the previous lemma. Assume wlog that $\norm{x}=1$, and let $x=y+z$ where $y=\pi_V(x)$ and $z=\pi_{V^\perp}(x)$. Then
\begin{gather}
 \norm{\pi_V(x)-\pi_W(x)}=\norm{y-\pi_W(y)-\pi_W(z)} \leq \norm{y-\pi_W(y)}+\norm{z-\pi_{W^\perp}(z)} = d(y,W)+d(z,W^\perp)\, .
\end{gather}
Since $\norm y^2 + \norm z^2 =\norm x^2= 1$, by the previous lemma we get the first estimate.

The reverse estimate is an immediate consequence of the definition. 
\end{proof}
\vspace{.5 cm}

\subsection{\texorpdfstring{Distance between $L^2$ best planes}{Distance between L2 best planes}}
Here we study the distance between best approximating subspaces for our measure $\mu$ on different balls. Let us begin by fixing our notation for this subsection, and pointing out the interdependencies of the constants chosen here.  Throughout this subsection, our choice of scale $\rho=\rho(n,M)>0$ is a constant which will eventually be fixed according to lemma \ref{lemma_alpharho}.  For applications to future sections, it is sufficient to know that we can take $\rho(n,M)= 10^{-10} (100n)^{-n} M^{-1}$. We also point out that in Section \ref{sec_proof_main}, we will fix $M=40^n$, and so $\rho$ will be a constant depending only on $n$. In particular, we can use the very coarse estimate
\begin{gather}\label{eq_rho_rough}
\rho=10^{-10}(100n)^{-3n}\, .
\end{gather}
We will also introduce a threshold value $\gamma_k= \omega_k 40^{-k}$.  The dimensional constant $\gamma_k$ is chosen simply to be much smaller than any covering errors which will appear.

We will consider a positive Radon measure $\mu$ supported on $S\subseteq \B 1 0$, and use $D(x,r)\equiv D^{k}_{\mu}(x,r)$ to bound the distances between best $L^2$ planes at different points and scales. By definition let us denote by $V(x,r)$ a best $k$-dimensional plane on $\B r x $, i.e., a $k$-dimensional affine subspace minimizing $\int_{\B r x} d(x,V)^2 d\mu$. Note that, in general, this subspace may not be unique. We want to prove that, under reasonable hypothesis, the distance between $V(x,r)$ and $V(y,r')$ is small if $d(x,y)\sim r$ and $r'\sim r$.\\

In order to achieve this, we will need to understand some minimal properties of $\mu$.  First, we need to understand how concentrated $\mu$ is on any given ball. For this reason, for some $\rho>0$ and all $x\in \B 1 0$ we will want to consider the upper mass bound 

\begin{gather}\label{eq_rho}
 \mu(\B \rho x )\leq M \rho^k\,\,\, \forall x\in B_1(0)\, .
\end{gather}

However, an upper bound on the measure is not enough to guarantee best $L^2$-planes are close, as the following example shows:
\begin{example}
 Let $V,V'$ be $k$-dimensional subspaces, $0\in V\cap V'$, and set $S=\ton{V\cap \B 1 0 \setminus \B{1/10}{0}}\cup S'$, where $S'\subseteq V'\cap \B {1/10}{0}$ and $\mu=\lambda^k|_S$. Then evidently $D(0,1)\leq \lambda^k(S')$ and $D(0,1/10)=0$, independently of $V$ and $V'$. However, $V(0,1)$ will be close to $V$, while $V(0,1/10)=V'$. Thus, in general, we cannot expect $V(0,1)$ and $V(0,1/10)$ to be close if $\mu(\B {1/10}{0})$ is too small.
\end{example}

Thus, in order to prove that the best planes are close, we need to have some definite amount of measure on the set, in such a way that $S$ ``effectively spans'' a $k$-dimensional subspace, where by effectively span we mean the following:
\begin{definition}
 Given a sequence of points $p_i\in \R^n$, we say that $\cur{p_i}_{i=0}^k$ $\alpha-$effectively span a $k$-dimensional affine subspace if for all $i=1,\cdots,k$
 \begin{gather}\label{eq_eff_span}
  \norm{p_i-p_0}\leq \alpha^{-1}\, , \quad p_i\not \in \B{\alpha}{p_0+\operatorname{span}\cur{p_1-p_0,\cdots,p_{i-1}-p_0}}\, .
 \end{gather}
\end{definition}
 Note that this definition is basically an affine version of Definition \ref{d:independent_points}. The definition implies that the vectors $p_i-p_0$ are linearly independent in a quantitative way. In particular, we obtain immediately that
\begin{lemma}\label{lemma_effspa}
 If $\cur{p_i}_{i=0}^k$ $\alpha$-effectively span the $k$-dimensional affine subspace $$V=p_0+\operatorname{span}\cur{p_1-p_0,\cdots,p_k-p_0}\, ,$$ then for all $x\in V$ there exists a \textit{unique} set $\cur{\alpha_i}_{i=1}^k$ such that
\begin{gather}
 x=p_0+\sum_{i=0}^k \alpha_i (p_i-p_0)\, , \quad \abs{\alpha_i}\leq c(n,\alpha)\norm{x-p_0}\, .
\end{gather}
\end{lemma}
\begin{proof}
 The proof is quite straightforward. Since $\cur{p_i-p_0}_{i=1}^k$ are linearly independent, we can apply the Gram-Schmidt orthonormalization process to obtain an orthonormal basis $e_1,\cdots,e_k$ for the linear space $\operatorname{span} \cur{p_i-p_0}_{i=1}^k$. By induction and \eqref{eq_eff_span}, it is easy to check that for all $i$
 \begin{gather}
  e_i = \sum_{j=1}^i \alpha'_{ij} (p_j-p_0)\, , \quad \abs{\alpha'_{ij}}\leq c(n,\alpha)\, .
 \end{gather}
Now the estimate follows from the fact that for all $x\in V$
\begin{gather}
 x=p_0+\sum_{i=1}^k \ps{x-p_0}{e_i} e_i\, .
\end{gather}

\end{proof}

With these definitions, we are ready to prove that in case $\mu$ is not too small, then its support must effectively span something $k$-dimensional.

\begin{lemma}\label{lemma_alpharho}
Let $\gamma_k=\omega_k 40^{-k}$. There exists a $\rho_0(n,\gamma_k,M)=\rho_0(n,M)$ such that if \eqref{eq_rho} holds for some $\rho\leq \rho_0$ and if $\mu(\B 1 0)\geq \gamma_k$, then for every affine subspace $V\subseteq \R^n$ of dimension $\leq k-1$, there exists an $x\in S\cap \B 1 0$ such that $\B{10\rho}{x}\cap V = \emptyset$ and $\mu\ton{\B{\rho}{x}\cap \B 1 0}\geq c(n,\rho) = c(n)\rho^n>0$. 
 \end{lemma}
\begin{proof}
Let $V$ be any $k-1$-dimensional subspace, and consider the set $\B {11\rho}{V}$. Let $B_i=\B{\rho}{x_i}$ be a sequence of balls that cover the set $\B{11\rho}{V}\cap \B 1 0$ and such that $B_i/2\equiv \B {\rho /2}{x_i}$ are disjoint and $x_i\in \B{11\rho}{V}\cap \B 1 0$. If $N$ is the number of these balls, then a standard covering argument gives
\begin{align}
 &N\omega_n \rho^n/2^n\leq \omega_{k-1} (1+\rho)^{k-1} \omega_{n-k+1} (12\rho)^{n-k+1}\leq 24^n \omega_{k-1} \omega_{n-k+1} \rho^{n-k+1}\, \quad \notag\\
 \Longrightarrow \quad &N\leq 48^{n} \frac{\omega_{k-1} \omega_{n-k+1}}{\omega_n} \rho^{1-k}\, .
\end{align}
By \eqref{eq_rho}, the measure of the set $\B{11\rho}{V}$ is bounded by
\begin{gather}
 \mu(\B{11\rho}{V})\leq \sum_i \mu\ton{B_i}\leq MN \rho^k \leq 48^n \frac{\omega_{k-1} \omega_{n-k+1}}{\omega_n} M \rho \leq 10^{5} (50n)^n M \rho =  c(n)M\rho \, .
\end{gather}
where the next-to-last estimate is an extremely rough bound on the constants involved. Thus if 
\begin{gather}
 \rho\leq 10^{-5}(50n)^{-n}\gamma_k/(4M)\, , 
\end{gather}
then $\mu(\B{11\rho}{V})\leq \gamma_k/4$.  In particular, we get that there must be some point of $S$ not in $\B{11\rho}{V}$.  More effectively, let us consider the set $S\cap \B 1 0 \setminus \B {11\rho}{V}$. This set can be covered by at most $c(n,\rho)=4^{n}\rho^{-n}$ balls of radius $\rho$ centered in $x\in S\cap \B 1 0 \setminus \B {11\rho}{V}$, and we also see that
\begin{gather}
 \mu\ton{\B {1}{0}\setminus \B{11\rho}{V}}\geq \frac{3\gamma_k}{4}\, .
\end{gather}
Thus, there must exist at least one ball of radius $\rho$ centered in $x$ and disjoint from $\B{10\rho}{V}$ such that
\begin{gather}
 \mu\ton{\B{\rho}{x}\cap \B 1 0}\geq \frac{3\gamma_k}{4} 4^{-n}\rho^n\geq c(n)\rho^n\, .
\end{gather}
\end{proof}
\vspace{.5 cm}

%
%

Now if at two consecutive scales there are some balls on which the measure $\mu$ {\it effectively spans} $k$-dimensional subspaces, we show that these subspaces have to be close together.

\begin{lemma}\label{lemma_vw}
 Let $\mu$ be a positive Radon measure and assume $\mu(\B 10 )\geq \gamma_k $. Additionally, let $B_\rho(x)\subset \B 1 0$ be a ball such that $\mu(\B \rho x )\geq \gamma_k \rho^k$ and for each $y\in \B \rho x$ we have $\mu(\B {\rho^2} y )\leq M \rho^{2k}$, where $\rho\leq \rho_0$.  Then if $A=V(0,1)\cap \B \rho x$ and $B=V(x,\rho)\cap \B \rho x$ are $L^2$-best subspace approximations of $\mu$ with $d(x,A)<\rho/2$, then 
 \begin{gather}\label{eq_distD}
  d_H(A,B)^2 \leq c(n,\rho,M) \ton{D^k_\mu(x,\rho)+D^k_\mu(0,1)}\, .
 \end{gather}
\end{lemma}
\begin{proof}
Let us begin by observing that if $c(n,\rho,M)>4\rho^2\delta^{-1}(n,\rho,M)$, which will be chosen later, then we may assume without loss of generality that 
\begin{align}\label{e:lemma_vw:1}
D^k_\mu(x,\rho)+D^k_\mu(0,1)\leq \delta = \delta(n,\rho,M)\, ,
\end{align}
since otherwise \eqref{eq_distD} is trivially satisfied. Moreover, note that $\gamma_k >> \epsilon_n$, so equation \eqref{eq_deph_D} is valid on $\B \rho x $ and on $\B 1 0$. \\

We will estimate the distance $d_H(A,B)$ by finding $k+1$ balls $\B{\rho^2}{y_i}$ which have enough mass and effectively span in the appropriate sense $V(x,\rho)$. Given the upper bounds on $D^k_\mu$, we will then be in a position to prove our estimate.

\vspace{5mm}

Consider any $\B {\rho^2} {y}\subseteq \B 1 0$ with $\mu\ton{\B {\rho^2} {y}}>0$ and let $p(y)\in \B{\rho^2}{y}$ be the center of mass of $\mu$ restricted to $\B {\rho}{x}\cap \B {\rho^2}{y}$.
Let also $\pi(p)$ be the orthogonal projection of $p$ onto $V(x,\rho)$.  By Jensen's inequality:
\begin{gather}\label{eq_vw1}
 d(p(y),V(x,\rho))^2=d(p(y),\pi(p(y)))^2=d\ton{\fint_{B_{\rho^2}(y)}z\,d\mu(z),V(x,\rho)}^2 \leq \frac{1}{\mu(\B{\rho^2}{y})}\int_{\B {\rho^2}{y}} d(z,V(x,\rho))^2 d\mu(z)\, . 
\end{gather}

Using this estimate and lemma \ref{lemma_alpharho} (or better its rescaled version applied to $\B {\rho}{x}$), we want to prove that there exists a sequence of $k+1$ balls $\B{\rho^2}{y_i}$ with $y_i\in \B{\rho}{x}$ such that
\begin{enumerate}
\def\theenumi{\roman{enumi}}
 \item\label{it_1} $\mu\ton{\B \rho {x} \cap \B{\rho^2}{y_i}}\geq c(n,\rho,M)>0 $
 \item\label{it_2} $\cur{\pi(p(y_i))}_{i=0}^k\equiv\cur{\pi_i}_{i=0}^k$ effectively spans $V(x,\rho)$. In other words for all $i=1,\cdots,k$, $\pi_i\in V(x,\rho)$ and 
 \begin{gather}\label{eq_vw2}
\pi_i\not \in \B{5\rho^2}{\pi_0+\operatorname{span}\ton{\pi_1-\pi_0,\cdots,\pi_{i-1}-\pi_0}}\, .
 \end{gather}
\end{enumerate}
We prove this statement by induction on $i=0,\cdots,k$. For $i=0$, the statement is trivially true since $\mu(\B{\rho}{x})\geq \gamma_k \rho^k$. In order to find $y_{i+1}$, consider the subspace $V^{(i)} = \pi_0 + \operatorname{span}\ton{\pi_1-\pi_0,\cdots,\pi_{i}-\pi_0}$. By lemma \ref{lemma_alpharho} applied to the ball $\B{\rho}{x}$, there exists some $\B{\rho^2}{y_{i+1}}$ such that $\mu\ton{\B {\rho}{x}\cap \B{\rho^2}{y_{i+1}}}\geq c(n,\rho,M)>0 $, $y_{i+1}\in \B{\rho}{x}$ and 
 \begin{gather}
 y_{i+1}\not \in \B{10\rho^2}{\pi_0+\operatorname{span}\ton{\pi_1-\pi_0,\cdots,\pi_{i}-\pi_0}}\, .
 \end{gather}
By definition of center of mass, it is clear that $d(y_{i+1},p(y_{i+1}))\leq \rho^2$. Moreover, by item \eqref{it_1} and equation \eqref{eq_vw1}, we get
\begin{gather}
 d(p(y_{i+1}),V(x,\rho))^2\leq c \int_{\B {\rho}{x}\cap \B {\rho^2}{y_{i+1}}} d(z,V(x,\rho))^2 d\mu(z) \leq c D^k_\mu(x,\rho)\leq c\delta\, .
\end{gather}
Thus by the triangle inequality we have $d(y_{i+1},\pi_{i+1})\leq 2\rho^2$ if $\delta\leq \delta_0(n,\rho,M)$ is small enough. This implies \eqref{eq_vw2}.
Using similar estimates, we also prove $d(p(y_{i+1}),V(0,1))^2\leq c' D^k_\mu(0,1)$ for all $i=-1,0,\cdots,k-1$. Thus by the triangle inequality
\begin{gather}
 d(\pi_{i+1},V(0,1))\leq d(\pi_{i+1},p(y_{i+1}))+d(p(y_{i+1}),V(0,1))\leq c(n,\rho,M)\ton{D^k_\mu(x,\rho)+D^k_\mu(0,1)}^{1/2}\, .
\end{gather}

\vspace{5mm}

Now consider any $y\in V(x,\rho)$. By item \eqref{it_2} and lemma \ref{lemma_effspa}, there exists a unique set $\cur{\beta_i}_{i=1}^k$ such that
\begin{gather}
 y=\pi_0+\sum_{i=1}^k \beta_i (\pi_i-\pi_0)\, , \quad \abs{\beta_i}\leq c(n,\rho)\norm{y-\pi_0}\, .
\end{gather}
Hence for all $y\in V(x,\rho)\cap \B{\rho}{x}$, we have
\begin{gather}
 d(y,V(0,1))\leq d(\pi_0,V(0,1)) + \sum_i \abs{\beta_i}  [d(\pi_i,V(0,1))+d(\pi_0,V(0,1))]\leq c(n,\rho,M)\ton{D^k_\mu(x,\rho)+D^k_\mu(0,1)}^{1/2}\, .
\end{gather}
By lemma \ref{lemma_hdv}, this completes the proof of \eqref{eq_distD}. 
\end{proof}
\vspace{.5 cm}

\subsection{\texorpdfstring{Comparison between $L^2$ and $L^\infty$ planes}{Comparison between L2 and L-infinity planes}} Given $\B r x $, we denote as before by $V(x,r)$ one of the $k$-dimensional subspace minimizing $\int_{\B r x } d(y,V)^2 d\mu$. Suppose that the support of $\mu$ satisfies a uniform one-sided Reifenberg condition, i.e. suppose that there exists a $k$-dimensional plane $L(x,r)$ such that $x\in L(x,r)$ and
\begin{gather}\label{eq__}
\supp{\mu} \cap \B{r}x \subseteq \B{\delta r}{L(x,r)}\, . 
\end{gather}
Then, by the same technique used in lemma \ref{lemma_vw}, we can prove that
\begin{lemma}\label{lemma_LV}
  Let $\mu$ be a positive Radon measure with $\mu\ton{\B 1 0}\geq \gamma_k$ and such that for all $\B \rho y \subseteq \B 1 0$ we have $\mu(\B \rho y )\leq M \rho^k$ and \eqref{eq__}. Then
 \begin{gather}
  d_H(L(0,1)\cap \B 1 0, V(0,1)\cap \B 1 0)^2 \leq c(n,\rho,M)\ton{\delta^2+D^k_\mu(0,1)}\, .
 \end{gather}
\end{lemma}
\vspace{.5 cm}

\subsection{bi-Lipschitz equivalences}
In this subsection, we study a particular class of maps with nice local properties. These maps are a slightly modified version of the maps which are usually exploited to prove Reifenberg's theorem, see for example \cite{reif_orig,Toro_reif,davidtoro}, \cite[section 10.5]{morrey} or \cite{simon_reif}. The estimates in this section are standard in literature.

We start by defining the functions $\sigma$. For some $0<r\leq 1$, let $\cur{x_i}$ be an $r/10$-separated subset of $\R^n$, i.e., 
\begin{enumerate}
\def\theenumi{\roman{enumi}}
 \item \label{item_1}$d(x_i,x_j)\geq r/10$.
\end{enumerate}
Let also $p_i$ be a points in $\R^n$ with
\begin{enumerate}
\def\theenumi{\roman{enumi}}\setcounter{enumi}{1}
 \item $p_i\in \B{10r}{x_i}$
\end{enumerate}
and let $V_i$ be a sequence of $k$-dimensional linear subspaces.

By standard theory, it is easy to find a locally finite smooth partition of unity $\lambda_i:\dR^n\to [0,1]$ such that
\begin{enumerate}\label{deph_sigma_0}
\def\theenumi{\roman{enumi}}\setcounter{enumi}{2}
 \item \label{item_3r}$\supp{\lambda_i}\subseteq \B {3r}{x_i}$ for all $i$,
 \item for all $x\in \bigcup_{i} \B{2 r}{x_i}$, $\sum_i \lambda_i(x)=1$ and $\sum_i \lambda_i(x')\leq 1$ for all $x'\in\dR^n$ ,
 \item $\sup_i \norm{\nabla \lambda_i}_\infty \leq c(n)/r$ ,
 \item \label{item_last} if we set $1-\psi(x)=\sum_i \lambda_i(x)$, then $\psi$ is a nonnegative smooth function with $\norm{\nabla \psi}_\infty \leq c(n)/r$ .
\end{enumerate}
Note that by \eqref{item_3r}, and since $x_i$ is $r$-separated, there exists a constant $c(n)$ such that for all $x$, $\lambda_i(x)>0$ for at most $c(n)$ different indexes.

For convenience of notation, set $\pi_V(v)$ to be the orthogonal projection onto the linear subspace $V$ of the free vector $v$, and set
\begin{gather}
\pi_{p_i,V_i}(x)=p_i+\pi_{V_i}(x-p_i)\, .
\end{gather}
In other words, $\pi_{p_i,V_i}$ is the affine projection onto the affine subspace $p_i+V_i$. Recall that $\pi_{V_i}$ is a linear map, and so the gradients of $\pi_{V_i}$ and of $\pi_{p_i,V_i}$ at every point are equal to $\pi_{V_i}$. 

\begin{definition}\label{deph_sigma}
 Given $\cur{x_i,p_i,\lambda_i}$ satisfying \eqref{item_1} to \eqref{item_last}, and given a family of linear $k$-dimensional spaces $V_i$, we define a smooth function $\sigma:\R^n\to \R^n$ by
 \begin{gather}  
  \sigma(x)= x+\sum_i \lambda_i(x) \pi_{V_i^\perp}\ton{p_i-x} = \psi(x) x +\sum_i \lambda_{i}(x)\pi_{p_i,V_i}\ton{x}\, .
 \end{gather}
\end{definition}
By local finiteness, it is evident that $\sigma$ is smooth. Moreover, if $\psi(x)=1$, then $\sigma(x)=x$. It is clear that philosophically $\sigma$ is a form of ``smooth interpolation'' between the identity and the projections onto the subspaces $V_i$. It stands to reason that if $V_i$ are all close together, then this map $\sigma$ is close to being an orthogonal projection in the region $ \bigcup_{i} \B{2 r}{x_i}$. 

\begin{lemma}\label{lemma_sigma_simon}
 Suppose that there exists a $k$-dimensional linear subspace $V\subseteq \R^n$ and a point $p\in \R^n$ such that for all $i$
 \begin{gather}\label{eq_sigmadelta}
  d_G(V_i,V)\leq \delta\, , \quad d(p_i,p+V)\leq \delta\, .
 \end{gather}
Then the map $\sigma$ restricted to the set $U=\psi^{-1}(0)=\ton{\sum_i \lambda_i}^{-1}(1)$ can be written as
\begin{gather}
 \sigma(x)=\pi_{p,V}(x)+e(x)\, ,
\end{gather}
and $e(x)$ is a smooth function with 
\begin{gather}
 \norm{e}_\infty + \norm{\nabla e}_\infty\leq c(n)\delta/r=c(n,r)\delta\, .
\end{gather}
\end{lemma}
\begin{remark}
Thus, on $U$ we have that $\sigma$ is the affine projection onto $V$ plus an error which is small in $C^1$.
\end{remark}

\begin{proof}
 On the set $U$, we can define 
 \begin{gather}
  e(x)=\sigma(x)-\pi_{p,V}(x)=-\pi_{p,V}(x)+ \sum_i \lambda_i(x) \cdot \ton{\pi_{p_i,V_i}(x)}\notag\\
  =\sum_i \lambda_i(x) \cdot \ton{p_i-p-\pi_V(p_i-p) +\pi_V(p_i) - \pi_{V_i}(p_i) + \pi_{V_i}(x)  -\pi_V(x)}\, .
 \end{gather}
 By \eqref{eq_sigmadelta} and lemma \ref{lemma_epsproj}, we have the estimates
 \begin{gather}\label{eq_222}
  \norm{p_i-p-\pi_V(p_i-p)}\leq\delta\, , \quad \norm{\pi_V(x-p_i)-\pi_{V_i}(x-p_i)}\leq 2\delta \norm{x-p_i}\leq 20 \delta r\, . 
 \end{gather}
 This implies
 \begin{gather}
  \norm{e}_{L^\infty(U)} \leq c(n)(1+13 r)\delta \leq c(n)\delta\, .
 \end{gather}

As for $\nabla e$, we have
\begin{gather}
 \nabla e = \sum_i \nabla \lambda_i(x) \cdot \ton{p_i-p-\pi_V(p_i-p) +\pi_V(p_i) - \pi_{V_i}(p_i) + \pi_{V_i}(x)  -\pi_V(x)} + \sum_i \lambda_i(x) \nabla \ton{\pi_{V_i}(x)-\pi_V(x)}\, .
\end{gather}
The first sum is easily estimated, and since $\ps{\nabla (\pi_{W})|_x }{w}=\pi_{W}(w)$, we can still apply lemma \ref{lemma_epsproj} and conclude:
\begin{gather}
 \norm{\nabla e}_{L^\infty(U)} \leq \frac{c(n)}{r} \delta\, .
\end{gather}
\end{proof}
\vspace{.5 cm}
As we have seen, $\sigma$ is in some sense close to the affine projection to $p+V$. In the next lemma, which is similar in spirit to \cite[squash lemma]{simon_reif}, we prove that the image through $\sigma$ of a graph over $V$ is again a graph over $V$ with nice bounds. 
\begin{lemma}[squash lemma]\label{lemma_squash}
 Fix $\rho\leq 1$ and some $\B{r/\rho}{y}\subseteq \R^n$, let $I=\cur{x_i}\cap \B{5r/\rho}{y}$ be an $r/10$-separated set and define $\sigma$ as in Definition \ref{deph_sigma}. Suppose that there exists a $k$-dimensional subspace $V$ and some $p\in \R^n$ such that $d(y,p+V)\leq \delta r$ and for all $i$:
 \begin{gather}\label{eq_111}
  d(p_i,p+V)\leq \delta r\, \quad \text{ and } \quad d_G(V_i,V)\leq \delta\, .
 \end{gather}
 Suppose also that there exists a $C^1$ function $g:V\to V^\perp$ such that $G\subseteq \R^n$ is the graph $$G=\cur{p+x+g(x)\, \ \ \text{for } \ \ x\in V}\cap \B{r/\rho}{y}\, ,$$ and $r^{-1}\norm{g}_\infty + \norm{\nabla g}_\infty \leq \delta'$. There exists a $\delta_0(n)>0$ sufficiently small such that if $\delta \leq \delta_0\rho $ and $\delta'\leq 1$, then
 \begin{enumerate}
 \def\theenumi{\roman{enumi}}
  \item \label{it_s1} $\forall z\in G$, $ r^{-1}\abs{\sigma(z)-z}\leq c(n)(\delta +\delta')\rho^{-1}$, and $\sigma$ is a $C^1$ diffeomorphism from $G$ to its image,
  \item \label{it_s2} the set $\sigma(G)$ is contained in a $C^1$ graph $\cur{p+x+\tilde g(x)\, , \ \ x\in V}$ with
  \begin{gather}
   r^{-1}\norm{\tilde g}_\infty + \norm{\nabla \tilde g}_\infty \leq c(n) (\delta+\delta')\rho^{-1}\, .
  \end{gather}
  \item \label{it_s3} moreover, if $U'$ is such that $\B{c(\delta +\delta')\rho^{-1}}{U'}\subseteq \psi^{-1}(0)$, then the previous bound is \textit{independent} of $\delta'$, in the sense that
  \begin{gather}
   r^{-1}\norm{\tilde g}_{L^\infty(U'\cap V)} + \norm{\nabla \tilde g}_{L^\infty(U'\cap V)}\leq c(n) \delta\rho^{-1}\, .
  \end{gather}
  For example, if $\delta'\leq \delta_0(n)\rho^{-1}$, we can take $U'=\bigcup_{i} \B{1.5 r}{x_i}$.
  \item \label{it_s4} the map $\sigma$ is a bi-Lipschitz equivalence between $G$ and $\sigma(G)$ with bi-Lipschitz\newline constant $\leq 1+c(n)(\delta+\delta')^2\rho^{-2}$.
 \end{enumerate}
\end{lemma}

\begin{proof}
For convenience, we fix $r=1$ and $p=0$. By notation, given any map $f:\R^n\to \R^m$, $p\in \R^n$ and $w\in T_p(\R^n)=\R^n$, we will denote by $\nabla|_p f[w]$ the gradient of $f$ evaluated at $p$ and applied to the vector $w$.

Recall that 
 \begin{gather}
  \sigma(x+g(x))=\psi(z) (x+g(x)) + \sum_{x_i\in I} \lambda_i(z) \ton{\pi_{p_i,V_i}(x+g(x))}\, , \quad 1-\psi(x)=\sum_{x_i\in I} \lambda_i(x)\, ,
 \end{gather}
where we have set for convenience $z=z(x)=x+g(x)$. Define $h(x)$ by 
\begin{gather}
 (1-\psi(z))x + h(x) \equiv \sum_{i} \lambda_i(z) \ton{\pi_{p_i,V_i}(x+g(x))}\, .
\end{gather}
Set also $h^T(x)=\pi_V(h(x))$ and $h^\perp(x)=\pi_{V\perp}(h(x))$. By projecting the function $\sigma(x+g(x))$ onto $V$ and its orthogonal complement we obtain
\begin{gather}
 \sigma(x+g(x))\equiv \sigma^T(x)+\sigma^\perp(x) \, ,\notag\\
 \sigma^T(x)= x + h^T(x)\, , \quad \sigma^\perp(x)=\psi(z)g(x)+h^\perp(x)\, .
\end{gather}

We claim that if $\delta'\leq 1$, then
\begin{gather}\label{eq_hT}
 \norm{h^T(x)}+\norm{\nabla {h^T(x)}}\leq \frac{c\delta}{\rho}\, ,
\end{gather}
where this bound is \textit{independent} of $\delta'$ as long as $\delta'\leq 1$. Indeed, for all $x\in V$ we have
\begin{gather}
 h^T(x)=\pi_V \qua{\sum_i \lambda_i(z) \ton{\pi_{p_i,V_i}(x+g(x))  - x}  }=\sum_i \lambda_i(z) \pi_V \qua{\ton{\pi_{p_i,V_i}(x)  - \pi_{V}(x) } +\pi_{V_i}(g(x))}
\end{gather}
Given \eqref{eq_111} and lemma \ref{lemma_epsproj}, with computations similar to \eqref{eq_222}, we get $\norm{h^T(x)}\leq c\delta (1+\rho^{-1})\leq c\delta \rho^{-1}$. As for the gradient, we get for any vector $w\in V$
\begin{gather}
 \nabla h^T|_x [w] = \pi_V \qua{ \sum_i \nabla\lambda_i |_z \qua{w+\nabla g|_x [w]} \ton{\pi_{p_i,V_i}(x+g(x))  - x} +\sum_i \lambda_i(z)\ton{\pi_{V_i}\ton{w+\nabla g[w]} -w }}\, ,
\end{gather}
In particular, we obtain
\begin{gather}
 \norm{\nabla h^T|_x [w]}\leq  \sum_i \norm{\nabla\lambda_i} \ton{1+\norm{\nabla g}}{\norm w}\norm{\pi_{p_i,V_i}(x+g(x))  - x} +\sum_i \lambda_i(z)\ton{\norm{\pi_{V_i}\ton{w}-w} + \norm{\pi_{V_i}\ton{\nabla g[w]}}}\, .
\end{gather}
For the first term, we can estimate
\begin{gather}
 \norm{\nabla \lambda_i}\leq c(n)\, , \quad \norm{\nabla g}\leq \delta'\leq 1\, , \quad \norm{\pi_{p_i,V_i}(x+g(x))  - x}\leq \norm{\pi_{p_i,V_i}(x)-x}+\norm{\pi_{V_i}(g(x))}\, .
\end{gather}
Since $x\in V$ with $\norm x \leq \rho^{-1}$, and $g(x)\in V^\perp$, by \eqref{eq_111} and lemma \ref{lemma_epsproj} we obtain
\begin{gather}
 \norm{\pi_{p_i,V_i}(x+g(x))  - x}\leq c\delta \rho^{-1}\, .
\end{gather}
As for the second term, we have
\begin{gather}
 \norm{\pi_{V_i}\ton{w}-w}\leq  c\delta \norm w\, , \quad \norm{\pi_{V_i}\ton{\nabla g[w]}}\leq c\delta \delta' \norm w \leq c \delta \norm w\, .
\end{gather}
Summing all the contributions, we obtain \eqref{eq_hT} as wanted.

Thus we can apply the inverse function theorem on the function $\sigma^T(x):V\to V$ and obtain a $C^1$ inverse $Q$ such that for all $x\in V$, $\norm{Q(x)-x}+ \norm{\nabla Q-id }\leq c(n)\delta \rho^{-1}$ , and if $\psi(x+g(x))=1$, then $Q(x)=x$ . So we can write that for all $x\in V$
\begin{gather}
 \sigma(x+g(x))=\sigma^T(x)+\tilde g(\sigma^T(x))\, \quad \text{where} \quad \tilde g(x)= \sigma^\perp (Q(x))=h^\perp(Q(x))+\psi\ton{z(Q(x))}g(Q(x)) \, .
\end{gather}

Arguing as above, we see that $h^\perp(x)$ is a $C^1$ function with 
\begin{gather}\label{eq_hp}
 \norm{h^\perp(x)}+\norm{\nabla h^\perp(x)}\leq \frac{c\delta}{\rho}\, ,
\end{gather}
and this bound is independent of $\delta'$ (as long as $\delta'\leq 1$).

Thus the function $\tilde g:V\to V^\perp$ satisfies for all $x$ in its domain 
\begin{gather}
\norm{\tilde g(x)}+ \norm{\nabla \tilde g(x)}\leq c(n)(\delta+\delta') \rho^{-1}\, ,
\end{gather}
Moreover, for those $x$ such that $\psi(Q(x)+g(Q(x)))=0$, the estimates on $\tilde g$ are independent of $\delta'$, in the sense that $\norm{\tilde g(x)}+ \norm{\nabla \tilde g(x)}\leq c(n)\delta \rho^{-1}$ . Note that by the previous bounds we have
\begin{gather}
 \norm{Q(x)+g(Q(x)) -x}\leq c (\delta+\delta')\rho^{-1}\, ,
\end{gather}
and so if $\B {c(\delta+\delta')\rho^{-1}}{U'}\subset \psi^{-1}(0)$, then for all $x\in U'\cap V$, $\psi(Q(x)+g(Q(x)))=0$. This proves items \eqref{it_s2}, \eqref{it_s3}. As for item \eqref{it_s1}, it is an easy consequence of the estimates in \eqref{eq_hT}, \eqref{eq_hp}.

Now since both $G$ and $\sigma(G)$ are Lipschitz graphs over $V$, it is clear that the bi-Lipschitz map induced by $\pi_V$ would have the right bi-Lipschitz estimate. Since $\sigma$ is close to $\pi_V$, it stands to reason that this property remains true. In order to check the estimates, we need to be a bit careful about the horizontal displacement of $\sigma$.

\paragraph{bi-Lipschitz estimates}
In order to prove the estimate in \eqref{it_s4}, we show that for all $z=x+g(x)\in G$ and for all unit vectors $w\in T_z(G)\subset \R^n$, we have
\begin{gather}\label{eq_nabla_2lip}
 \abs{ \norm{\nabla \sigma|_z [w]}^2-1}\leq c(\delta+\delta')^2\, .
\end{gather}
First of all, note that if $\psi(z)=1$, then $\sigma$ is the identity, and there's nothing to prove.

In general, we have that
\begin{gather}
 \nabla \sigma|_z [w]= \underbrace{\ton{\psi(z) w + \sum_i \lambda_i (z) \pi_{V_i}[w]}}_{:=A} + \underbrace{\ton{ z \nabla \psi [w]+ \sum_i \pi_{p_i,V_i} (z) \nabla \lambda_i [w]}}_{:=B} \, .
\end{gather}
Since $\psi(z)+\sum_i \lambda_i(z)=1$ everywhere by definition, we have
\begin{gather}
 \norm{B} =\norm{ \sum_i (\pi_{p_i,V_i} (z) -z) \nabla \lambda_i [w]}\leq c \sup_i \cur{\norm{\pi_{p_i,V_i} (z) -z}}\leq c(\delta+\delta')\, .
\end{gather}
This last estimate comes from the fact that $G$ is the graph of $g$ over $V$ with $\norm{g}_\infty \leq \delta'$. Moreover, we can easily improve the estimate for $B$ in the horizontal direction using lemma \ref{lemma_epsproj}. Indeed, since $\pi_{p_i,V}(z)-z=-\pi_{V_i^\perp}(z-p_i)$, we have
\begin{gather}
 \norm{\pi_V B} =\norm{ \sum_i \pi_V \ton{\pi_{p_i,V_i} (z) -z} \nabla \lambda_i [w]}\leq c \sup_i \cur{\norm{\pi_V\ton{\pi_{p_i,V_i} (z) -z}}} \\
 \leq c \sup_i \cur{\norm{\pi_V\ton{\pi_{V_i^\perp} (x+g(x)) -\pi_{V_i^\perp}(p_i)}}}\leq  c(\delta^2+\delta'\delta)\, .\notag
\end{gather}
As for $A$, by adapting the proof of lemma \ref{lemma_sigma_simon}, we get $\norm{A-\pi_V[w]}\leq c\ton{\delta+\delta'}$. Moreover, also in this case we get better estimates for $A$ in the horizontal direction. Indeed, we have
\begin{gather}
 \norm{\pi_V(A) -\pi_V[w]} = \norm{\psi(z) \pi_V [w] +\sum_i \ton{ \lambda_i(z) \pi_V[\pi_{V_i}[w]]} -\pi_V[w]} = \norm{\sum_i  \lambda_i(z) \ton{\pi_V[\pi_{V_i}[w] -\pi_V[w]]}}\, .
\end{gather}
Now let $w=\pi_V[w]+\pi_{V^\perp}[w]=w_V+w_{V^\perp}$. Then we have
\begin{gather}
 \norm{\pi_V(A) -\pi_V[w]}\leq \sum_i  \lambda_i(z) \ton{\norm{\pi_V[\pi_{V_i}[w_V] -w_V]} + \norm{\pi_V[\pi_{V_i}[w_{V^\perp}]]}}\\
 \notag =\sum_i  \lambda_i(z) \ton{\norm{\pi_V[\pi_{V_i^\perp}[w_V]]} + \norm{\pi_V[\pi_{V_i}[w_{V^\perp}]]}}\, .
\end{gather}
Since $G$ is the Lipschitz graph of $g$ over $V$ with $\norm{\nabla g}\leq c\delta'$, then $\norm{\pi_{V^\perp}[w]}\leq c\delta'$. Then, by lemma \ref{lemma_epsproj}, we have
\begin{gather}
 \norm{\pi_V(A) -\pi_V[w]}\leq c\sum_i  \lambda_i(z) \ton{\delta^2+\delta\delta'}\, .
\end{gather}
Summing up, since $\norm{\pi_V[w]}\leq \norm w =1$, we obtain that
\begin{gather}
 \abs{\norm{\nabla\sigma|_z[w]}^2 -1} = \abs{\norm{\pi_{V^\perp}\nabla\sigma|_z[w]}^2+\norm{\ton{\pi_{V}\nabla\sigma|_z[w] - \pi_V[w]}+\pi_V[w]}^2  -1} \notag \\
 \leq c(\delta+\delta')^2+\abs{\norm{\pi_V[w]}^2 -1}=c(\delta+\delta')^2+\norm{\pi_{V^\perp}[w]}^2\leq c(\delta+\delta')^2\, .
\end{gather}
\end{proof}

\vspace{.5 cm}

\subsection{\texorpdfstring{Pointwise Estimates on $D$}{Pointwise Estimates on D}}  We wish to see in this subsection how \eqref{e:reifenberg_L_2_discrete} implies pointwise estimates on $D$, which will be convenient in the proof of the generalized Reifenberg results.  Indeed, the following is an almost immediate consequence of Remark \ref{r:D_scale_control}:

\begin{lemma}\label{l:D_pointwise_bound}
Assume $B_r(x)\subseteq B_2(0)$ satisfies $\mu(B_r(x))\geq \gamma_k r^k>>\epsilon_n r^k$ and $\int_{B_{2r}(x)}D^k_\mu(y,2r)\,d\mu(y)<\delta^2 (2r)^k$.  Then there exists $c(n)$ such that $D^k_\mu(x,r)<c\delta^2$.  In particular, if \eqref{e:reifenberg_L_2_discrete} holds then for every $B_{r}(x)\subseteq B_1(0)$ such that $\mu(B_r(x))\geq 4^k\epsilon_n r^k$ we have that $D^k_\mu(x,r)<c\delta^2$.
\end{lemma}
\begin{proof}
 First of all, note that in all of our theorems we just investigate properties of $\mu|_{\B 1 0}$, thus the first assumption is not too restrictive. Under this assumption, it is easy to see that for all $x\in \B 2 0$ and $r\geq 1/16$, we have
 \begin{gather}
  D(x,r)\leq 16^{-k-2} D(0,1)=c(k) D(0,1)\, .
 \end{gather}

\end{proof}

\vspace{.5 cm}

\section{Proof of theorem \ref{t:reifenberg_W1p_discrete}: The Discrete-Reifenberg}\label{sec_proof_main}
First of all, note that, by definition of $\mu$, the statement of this theorem is equivalent to
\begin{gather}
 \mu(\B 1 0)\leq D(n)/\omega_k\, .
\end{gather}
In the proof, we will fix the constant $C_1(k)\leq 40^k\omega_k$ and therefore the positive scale $\rho(n,C_1(k))=\rho(n)<1$ according to lemma \ref{lemma_alpharho}.  For convenience, we will assume that $\rho=2^q$, $q\in \N$, so that we will be able to use the sum bounds \eqref{eq_sum2^alpha} more easily.


\paragraph{Reduction to a quantized measure}
It will be convenient to consider a measure $\mu\equiv \sum_{s\in S}\omega_k r^k_s \delta_{x_s}$ where $r_s$ is ``quantized'', meaning that
\begin{gather}\label{eq_quantum}
 r_s\in \rho(n)^\N=\cur{t\in \R \ \ s.t. \ \ t=\rho(n)^m \ \ \text{for some} \ m\in \N}\, .
\end{gather}
For this reason, we define $\tilde r_s = \max \cur{t\in \rho(n)^\N \ \ s.t. \ \ t\leq r_s}$ and
\begin{gather}
 \tilde \mu = \sum_{s\in S} \omega_k \tilde r_s^k \delta_{x_s}\, .
\end{gather}
It is clear that $\tilde \mu\leq \mu \leq \rho(n)^k \tilde \mu$, and given the monotonicity of $D_\mu^k$ wrt $\mu$ explained in remark \ref{rem_mon}, the bound \eqref{e:reifenberg_L_2_discrete} is still valid with $\tilde \mu$ in place of $\mu$.

We will prove the theorem for $\tilde \mu$, and in particular we will prove that
\begin{gather}
 \tilde \mu(\B 1 0)\leq C_1(k)\, .
\end{gather}
With this bound, it is evident that
\begin{gather}
 \mu(\B 1 0)\leq \rho(n)^k \tilde \mu(\B 1 0)\leq C_1(k)\rho(n)^k \equiv D(n)/\omega_k\, .
\end{gather}
In other words, we can assume without essential loss of generality that $r_s\in \rho(n)^\N$. 

For convenience of notation, we will still denote $\tilde \mu$ simply by $\mu$.

\vspace{5mm}
\paragraph{Bottom scale} In the proof, it will be convenient to assume that $r_s\geq \bar r >0$. It is clear that, by means of a simple limiting argument, this assumption is not restrictive.
In particular, fix any positive radius $\bar r= r_A=\rho^A$ for some $A\in \N$, and consider the measure $\mu_{\bar r}\leq \mu$ defined by
\begin{gather}
 \mu_{\bar r} = \sum_{s \ \ s.t. \ \ r_s\geq \bar r} \omega_k r_s^k \delta_{x_s} \, .
\end{gather}
Note that this is a finite sum if $\bar r$ is positive. By Remark \ref{rem_mon}, we see that $\mu_{\bar r}$ satisfies all the hypothesis of this theorem, and since $\mu_{\bar r} \nearrow \mu$, if we prove uniform bounds on $\mu_{\bar r}(\B 1 0)$ which are \textit{independent} of $\bar r$, we can conclude the theorem. 

For this reason, in the rest of the proof we will assume for simplicity that $r_s\geq \bar r =\rho^A>0$ for all $s$. 
\vspace{5mm}

\subsection{First induction: upwards} We are going to prove inductively on $j=A,\cdots,0$ that for all $x\in \B 1 0 \subset \R^n$ and $\er_j=\rho^j\leq 1$, either $\B{\er_j}{x}$ is contained in one of the balls $\cur{\B{r_s}{x_s}}_{s\in S}$, or we have the bound
\begin{gather}\label{eq_srx}
 \mu\ton{\B{\er_j}{x}} \leq C_1(k) \er_j^k\, .
\end{gather}
Note that, for $j=A$, this bound follows from the definition of the measure $\mu$ and the assumption that $r_s\geq \bar r$. Note also that this implies $\er_j\leq 2r_s$ for all $s\in S$ and $x_s\in \overline{\B {\er_j}{x}}$.

Clearly, we can assume wlog that $\mu\ton{\B{\er_j}{x}} \geq \gamma_k \er_j^k$, otherwise there is nothing to prove. This observation will be essential in order to apply lemma \ref{lemma_vw}.

Moreover, as long as we are trying to prove \eqref{eq_srx}, we can replace wlog $\mu$ with $\mu|_{\B {\er_j}{x}}$. Indeed, by Remark \ref{rem_mon}, all the hypotheses of theorem \ref{t:reifenberg_W1p_discrete} hold also for any restriction of $\mu$, in particular equation \eqref{e:reifenberg_L_2_discrete}. Thus, from now on, $\mu$ will indicate $\mu|_{\B {\er_j}{x}}$ and $S=\supp \mu \subset \overline{\B {\er_j}{x}}$.

\begin{remark}
 Note that if $\nu=\mu|_{\B {\er_j}{x}}$, then all the  $D^k_{\nu}$ on balls $\B s y \supset \B {\er_j}{x}$ are controlled. Indeed, we have
 \begin{gather}
  D^k_\mu (x,\er_j)=D^k_\nu (x,\er_j) = \er_j^{-k-2} \int_{\B {\er_j}{x}} d(z,V(x,\er_j))^2 d\mu(z) = \er_j^{-k-2} \int_{\B {s}{y}} d(z,V(x,\er_j))^2 d\nu(z) \geq \frac{s^{k+2}}{\er_j^{k+2}} D^k_{\nu}(y,s)\, .
 \end{gather}
In particular, this implies that if $\B {\er_j}{x}\subseteq \B{s}{y}$, then
\begin{gather}
 D^k_{\nu}(y,s)\leq \ton{\frac{\er_j}{s}}^{k+2} D^k_\nu(x,\er_j)=\ton{\frac{\er_j}{s}}^{k+2} D^k_\mu(x,\er_j)\, .
\end{gather}
In turn, as long as $\mu(\B {\er_j}{x})\leq c(n) \er_j^k$, we also have the bound
\begin{align}
&\int_{\R^n}\ton{\int_0^\infty D^k_\nu(z,t)\,{\frac{dt}{t}}}\, d\nu(z)< c \delta^2 s^{k}\, .
\end{align}
\end{remark}

\subsection{Rough estimate} Fix some $j$, and suppose that \eqref{eq_srx} holds on all scales below $\er_j$, i.e., for all $y\in \B 1 0$ and $\bar r \leq \er_{i} \leq \er_{j}$, $\mu(\B {\er_{i}}{y})\leq C_1(k) \er_{i}^k$.
 
Let us first observe that we can easily obtain a bad upper bound on $\mu\ton{\B{\chi \er_{j}}{x}}$ for any fixed $\chi>1$. Consider the points in $\cur{x_s}_{s\in S} \cap \B{\chi \er_{j}}{x}$, and divide them into two groups: the ones with $r_s\leq \er_{j}$ and the ones with $r_s> \er_{j}$. Note that by \eqref{eq_quantum}, $r_s>\er_j$ is equivalent to $r_s\geq \er_{j-1}$.

For the first group, cover them by balls $\B{\er_{j}}{z_i}$ such that $\B{\er_{j}/2}{z_i}$ are disjoint. Since there can be at most $c(n,\chi)$ balls of this form, and for all of these balls the upper bound \eqref{eq_srx} holds, we have an induced upper bound on the measure of this set.

As for the points with $r_s> \er_{j}$, by construction there can be only $c(n,\chi)$ many of them, and we also have the bound $r_s\leq 2\chi \er_{j}$. Summing up the two contributions, we get the very rough estimate
\begin{gather}\label{eq_lambdarough}
 \mu\ton{\B{\chi \er_{j}}{x}}\leq C_2(n,\chi) \er_{j}^k\, ,
\end{gather}
where $C_2>>C_1$.  Note that, as long as the inductive hypothesis holds, $C_2$ is independent of $j$.  However, it is clear that successive repetitions of the above estimate will not lead to \eqref{eq_srx}.
\vspace{.5 cm}

\subsection{Second induction: downwards. Outline of the proof}

Suppose that \eqref{eq_srx} is true for all $x\in \B 1 0$ and $i=j+1,\cdots,A$. Fix $x\in \R^n$, and consider the set $B=\B{\er_j}{x}$. Recall that we always assume that $B$ is not contained in one of the balls $\B{r_s}{x_s}$, otherwise the bound \eqref{eq_srx} might fail for trivial reasons.  
We are going to build by induction on $i\geq j$ a sequence of smooth maps $\sigma_i:\R^n\to \R^n$ and smooth $k$-dimensional manifolds $T_i$ which will serve as approximations for the support of $\mu$ at scale $\er_{i}$. Let us outline the inductive procedure now, and introduce all the relevant terminology.  Everything described in the remainder of this subsection will be discussed more precisely over the coming pages.  To begin with, we will have at the first step that
\begin{align}
 &\sigma_j=id,\, \notag\\
 &T_j=V(x,\er_j)\subset \R^n\, ,
\end{align}
where $V(x,\er_j)$ is one of the $k$-dimensional affine subspaces which minimizes $\int_{\B{\er_j}{x}} d^2(y,V)\,d\mu$.  Thus, the first manifold $T_j$ is a $k$-dimensional affine subspace which best approximates $\B{\er_j}{x}$.  At future steps we can recover $T_{i+1}$ from $T_{i}$ and $\sigma_{i+1}$ from the simple relation
\begin{align}
 &T_{i+1}=\sigma_{i+1}(T_{i})\, .
\end{align}
We will see that $\sigma_{i+1}$ is a diffeomorphism when restricted to $T_{i}$, and thus each additional submanifold $T_{i+1}$ is also diffeomorphic to $\dR^k$.  As part of our inductive construction we will build at each stage a Vitali covering of $T_i$ given by
\begin{align}\label{e:downward_induction:1}
\B{\er_{i}}{T_i}\cap \B{\er_j}{x}\sim \bigcup_{t=j}^i\ton{\bigcup_{y\in I_b^t}\B {\er_{t}}{y}\cup \bigcup_{x_s\in I_f^t} \B{r_s}{x_s}}\cup \bigcup_{y\in I_g^i} \B{\er_{i}}{y}\, ,
\end{align}
where $I_g$, $I_b$, and $I_f$ represent the {\it good}, {\it bad}, and {\it final} balls in the covering. Final balls are balls belonging to the original covering $\B{r_s}{x_s}$ such that $r_s\in[\er_{i},\er_{i-1})$ (equivalently, by \eqref{eq_quantum}, $r_s=\er_i$), and the other balls in the covering are characterized as good or bad according to how much measure they carry. Good balls are those with large measure, bad balls the ones with small measure. More precisely, we have
\begin{align}\label{e:downward_induction:2}
&\mu\big(\B {\er_{i}}{y}\big)\geq \gamma_k \er_{i}^k\, ,\quad\text{if}\quad y\in I^i_g\, ,\notag\\
&\mu\big(\B {\er_{i}}{y}\big)<\gamma_k \er_{i}^k\, ,\quad\text{if}\quad y\in I^i_b\, .
\end{align}

We will see that, over each good ball $\B{\er_{i}}{y}$ in this covering, $T_i$ can be written as a graph over the best approximating subspace $V(y,\er_{i})$ with good estimates.\\

Our goal in these constructions is the proof of \eqref{eq_srx} for the ball $  B=\B{\er_j}{x}$, and thus we will need to relate the submanifolds $T_i$, and more importantly the covering \eqref{e:downward_induction:1}, to the set $  B$.  Indeed, this covering of $T_i$ almost covers the set $  B$, at least up to an excess set $  E_{i-1}$.  That is,
\begin{align}
  \supp \mu \cap B\subseteq   E_{i-1} \cup \bigcup_{t=j}^i\ton{\bigcup_{y\in I_b^t}\B {\er_{t}}{y}\cup \bigcup_{x_s\in I_f^t} \B{r_s}{y}}\cup \bigcup_{y\in I_g^i} \B{\er_{i}}{y}\, .
\end{align}
We will see that the set $E_{i-1}$ consists of those points of $  B$ which do not satisfy a uniform Reifenberg condition.  Thus in order to prove \eqref{eq_srx} we will need to estimate the covering \eqref{e:downward_induction:1}, as well as the excess set $  E_{i-1}$.\\

Let us now outline the main properties used in the inductive construction of the mapping $\sigma_{i+1}:\dR^n\to \dR^n$, and hence $T_{i+1}=\sigma_{i+1}(T_i)$.  As is suggested in \eqref{e:downward_induction:1}, it is the good balls and not the bad and final balls which are subdivided at further steps of the induction procedure.  In order to better understand this construction let us begin by analyzing the good balls $\B{\er_{i}}{y}$ more carefully.  On each such ball we may consider the best approximating $k$-dimensional subspace $  V(y,\er_{i})$.  Since $\B{\er_{i}}{y}$ is a good ball, one can check that most of $  \operatorname{supp}(\mu)\cap \B{\er_{i}}{y}$ must satisfy a uniform Reifenberg and reside in a small neighborhood of $  V(y,\er_{i})$.  We denote those points which don't by $  E(y,\er_{i})$, see \eqref{eq_E} for the precise definition.  Then we can define the next step of the excess set by
\begin{align}
  E_i =   E_{i-1}\cup \bigcup_{y\in I^i_g}   E(y,\er_{i})\, . 
\end{align}
Thus our excess set represents all those points which do not lie in an appropriately small neighborhood of the submanifolds $T_i$.  With this in hand we can then find a submanifold $T'_i\subseteq T_i$, which is roughly defined by 
\begin{align}
T'_i \approx T_{i}\setminus \ton{\bigcup_{t=j}^{i+1}\bigcup_{y\in I_b^{t}}\B{\er_{t}/6}{y}\cup \bigcup_{t=j}^{i+1}\bigcup_{x_s\in I_f^{t}}\B{r_s/6}{x_s} }\, ,
\end{align}
see \eqref{e:downward_induction:map_manifold:1} for the precise inductive definition, such that
\begin{align}
  \supp \mu \cap B \subseteq   E_{i}\cup \bigcup_{t=j}^i\ton{\bigcup_{y\in I_b^t}\B {\er_{s}}{y}\cup \bigcup_{x_s\in I_f^t} \B{r_s}{y}} \cup B_{\er_{i+1}/4}\big(T'_i\big) \equiv   R_i \cup \bigcup B_{\er_{i+1}/4}\big(T'_i\big)\, ,
\end{align}
where $  R_i$ represents our remainder term, and consists of those balls and sets which will not be further subdivided at the next stage of the induction.  

The basic idea is that if in our induction we find a bad ball or a final ball $\B {\er} x$, we know that the measure carried by this ball is bounded by $C\er^k$. Because of this upper bound, in order to get the final estimate on $\mu$ we do not need to further analyze the measure inside any of these balls. However, we do need to keep track of the measure carried by these balls in successive induction steps. This is why every time we find a bad or final ball, we create a corresponding ``hole'' in the manifold $T_i$, and obtain as a result $T_i'$. By construction, the $k$-dimensional measure of these holes is comparable to $\mu(\B {\er} x )$, and thus the $k$-dimensional measure of $T_i$ (without holes) already ``includes'' the $\mu$-measure of the final and bad balls at all bigger scales.

Now in order to finish the inductive step of the construction, we can cover $B_{\er_{i+1}/4}\big(T'_i\big)$ by some Vitali set
\begin{align}
B_{\er_{i+1}/4}\big(T'_i\big)\subseteq \bigcup_{y\in I} \B{\er_{i+1}}{y}\, ,
\end{align}
where $y\in I\subseteq T'_i$.  We may then decompose the ball centers 
\begin{align}
I^{i+1}=I^{i+1}_g\cup I^{i+1}_b\cup I^{i+1}_f\, ,
\end{align}
based on \eqref{e:downward_induction:2}.  Now we will use Definition \ref{deph_sigma} and the best approximating subspaces $V(y,\er_{i+1})$ to build $\sigma_{i+1}:\dR^n\to \dR^n$ such that
\begin{align}
\text{supp}\{\sigma_{i+1}-Id\}\subseteq \bigcup_{y\in I^{i+1}_g}\B{3\er_{i+1}}{y}\, .
\end{align}
In order to prove the final bounds, we need to track the measure of the approximating manifolds $T_i$ as $i$ goes to infinity. We can use the local bi-Lipschitz estimates for $\sigma_{i}$ at scale $\er_{i}$ and integrate them along each manifold $T_i$ to obtain uniform bounds on $\lambda^k(T_i)$ as $i$ goes to infinity. This completes the outline of the inductive construction.  

\subsection{First steps in the induction}

In order to make the proof more understandable, we give in detail the proof of the first steps in the downwards induction, which contains most of the necessary ideas to carry out the whole construction. 

Fix any $\B {\er_j}{x}$. Without loss of generality, we assume that
\begin{gather}
 \mu(\B {\er_j}{x}) \geq 2\gamma_k \er_j^k\, ,
\end{gather}
otherwise we clearly have the measure estimate we want to prove.

With this condition, it makes sense to talk about a best $L^2$ approximating subspace for the support of $\mu$ on $\B {\er_j}{x}$. Denote this subspace by $V(x,\er_j)\equiv T_j$. 

Now, we want to cover the support of $\mu$ with balls of radius $9\er_{j+1}/10$ (roughly one scale smaller) in such a way to have good $k$-dimensional packing estimates on these balls. The idea is that condition \eqref{e:reifenberg_L_2_discrete} will force the measure $\mu$ to be \textit{almost} supported in a small tubular neighborhood of $T_j$, up to small measure. So we split our ball in
\begin{gather}
 \B{\er_j}{x}=\ton{\B{\er_j}{x} \cap \B {\er_{j+1}/11} {V(x,\er_j)} }\cup\ton{\B{\er_j}{x} \cap \B {\er_{j+1}/11} {V(x,\er_j)} ^C }\, .
\end{gather}
The second part in this splitting is in some sense preventing the measure $\mu$ to satisfy an $L^\infty$ Reifenberg condition. We call this part \textit{excess set}, in particular
\begin{gather}\label{eq_excess_j}
 E(x,\er_j)= \B{\er_j}{x} \cap \B {\er_{j+1}/11} {V(x,\er_j)} ^C \, .
\end{gather}
Although $\mu(E)$ can be positive, it cannot be too big. Indeed, we have the trivial estimate
\begin{align}
\int_{\B{\er_j}{x}\setminus   E(x,\er_j)} &d(y,  V(x,\er_j))^2 \ d\mu(y)+\mu\ton{  E(x,\er_j)} (\er_{j+1}/11)^2\leq \int_{\B{\er_j}{x}} d(y,  V(x,\er_j))^2 \ d\mu(y)
= \er_j^{k+2} D^k_\mu (x,\er_j)\leq c(n,\rho) \er_{j}^{k+2} \delta\, .
\end{align}

Now, almost all of the measure $\mu$ must be concentrated in $\B{\er_j}{x} \cap \B {\er_{j+1}/11} {V(x,\er_j)}$. In order to estimate this part, we build a covering with good overlapping properties of this set by balls of radius $\geq 9\er_{j+1}/10$ centered on $V(x,\er_j)\cap \B{\er_j}{x}$. 

This covering is built in the following way. First of all, we consider separately all the balls $\cur{\B{r_s}{x_s}}_{s\in S}$ with $r_s\geq \er_{j+1}$. Recall that by construction $r_s\leq 2\er_j$ for all $s\in S$ with $x_s\in \B {\er_j}{x}$, otherwise there's nothing to prove since $\B {\er_j}{x}$ would be contained in $\B {r_s}{x_s}$ for some $s$. Note that, all of these balls are pairwise disjoint. We will call these balls \textit{final balls}, and set $I^f_j$ to be the set of centers of these balls.

In most cases, or at least in the most interesting cases, the set of final balls will be empty or very small. We complete this partial covering of $\B{\er_j}{x} \cap \B {\er_{j+1}/11} {V(x,\er_j)}$ with other balls centered on $V(x,\er_j)$ of radius $9\er_{j+1}/10$ in such a way that this covering have a Vitali property.

Thus we obtain
\begin{gather}
 \B{\er_j}{x} \cap \B {\er_{j+1}/11} {V(x,\er_j)} \subset \bigcup_{x_s\in I^f_{j} } \B{r_s}{x_s} \cup \bigcup_{q\in Q} \B {9\er_{j+1}/10}{x_q}\, .
\end{gather}

Now, we split the set $Q$ according to how much measure is contained in $\B {\er_{j+1}}{x_q}$. In particular, if $\mu\ton{\B {\er_{j+1}}{x_q}}\geq \gamma_k \er_{j+1}^k$, we say that this is a \textit{good ball}, otherwise we say that this is a \textit{bad ball}.

Now we want to build a new best approximating manifold $T_{j+1}$ at this scale. On good balls, we have a best approximating subspace $V(x_q,\er_{j+1})$, and since these balls carry enough measure, we can apply lemma \ref{lemma_vw} and obtain a quantitative estimate on the distance between $V(x,\er_j)$ and $V(x_q,\er_{j+1})$. In turn, this will allow us to apply the construction in the squash lemma \ref{lemma_squash}. In particular, we have a smooth map $\sigma$ defined on $\R^n$, and moreover $\sigma(T_j)\equiv T_{j+1}$ is a diffeomorphism onto its image when restricted to $T_j$ with good quantitative bi-Lipschitz estimates. The details of this construction are carried out in subsection \ref{sec_sigma}.

As for bad balls, we don't have to worry too much about those, since they carry really small measure. In particular, 
\begin{gather}
 \mu\ton{\B {\er_{j+1}}{x_q}}< \lambda^k\ton{T_{j+1} \cap \B {\er_{j+1}/6}{x_q}}\, .
\end{gather}
Thus, in order to keep track of the measure carried by bad balls, we can simply keep track of the $k$-dimensional measure of the approximating manifolds $T_{i}$. In some sense, every time we hit a bad ball, we can compare its measure $\mu$ to the $k$-dimensional Hausdorff measure of the ``hole'' $T_{j+1} \cap \B {\er_{j+1}/6}{x_q}$. Since, as we will prove, for $i\geq j+1$, $T_i \cap \B {\er_{j+1}/6}{x_q}$ and $T_j \cap \B {\er_{j+1}/6}{x_q}$ are substantially equal, by estimating the measure of $T_i$ we also estimate the total measure of all the bad balls.

Moreover, since this estimate covers the measure of the whole bad ball, from this step forward we do not have to worry about $\mu|_{\B {\er_{j+1}}{x_q}}$ any longer in the induction. We can use a similar argument to track the measure of final balls. 

This construction is carried out in Subsection \ref{sec_holes}

Evidently, we cannot hope to apply these considerations also to good balls in order to get the estimates we want, because the measure of good balls is not small (a priori it could be anything). Instead, on the new good balls, we start over the same construction we outlined here (excess set, construction of the new best approximating manifold, and so on) and keep going by induction. The inductive estimates on the $k$-dimensional measure are carried out in Subsection \ref{sec_volmani}.

\subsection{Second induction: details of the construction}
Let us now describe precisely the proof of this inductive construction which will lead to \eqref{eq_srx}. For $j\leq i\leq A$, we will define a sequence of approximating manifolds $T_i$ for the support of $\mu$ and a sequence of smooth maps $\sigma_i$ such that
\begin{enumerate}
 \def\theenumi{\roman{enumi}}
 \item \label{it_i1} $\sigma_j=id$, $T_j=V(x,\er_j)\subset \R^n$,
 \item \label{it_i2} $T_i=\sigma_i(T_{i-1})$,
 \item \label{it_i3} for $i\geq j+1$ and $y\in T_{i-1}$,
 \begin{gather}\label{eq_distdelta}
  d(\sigma_{i}(y),y)\leq c\delta \er_{i}\, ,
 \end{gather}
 and $\sigma_{i}|_{T_{i-1}}$ is a diffeomorphism,
 \item \label{it_i5} for every $y\in T_i$, $T_i\cap \B{2 \er_{i}}{y}$ is the graph over some $k$-dimensional affine subspace of a smooth function $f$ satisfying
  \begin{gather}\label{eq_lipTi}
  \frac{\norm{f}_{\infty}}{\er_{i}} + \norm{\nabla f}_\infty \leq c\delta\, .
 \end{gather}
\end{enumerate}
As outlined before, the manifolds $T_i$ will be good approximations of the set $S$ up to some ``excess'' set of small measure. Moreover, we will also introduce the concept of \textit{good, bad} and \textit{final} balls (whose centers will be in the sets $I_g^i$, $I_b^i$ and $I_f^i$), a \textit{remainder set} $R_i$, and the manifolds $T_i'\subseteq T_i$. Before giving the precise definitions (which are in equations \eqref{eq_deph_Igb}, \eqref{eq_deph_If}, \eqref{eq_R} and \eqref{e:downward_induction:map_manifold:1} respectively), let us group here all the properties that we will need (and prove) for these objects, so that the reader can always come back to this page to have a clear picture of what are the objectives of the proof.
\begin{enumerate}
 \setcounter{enumi}{4} \def\theenumi{\roman{enumi}}
 \item\label{it_i4} for every $i\geq j+1$ and $y\in I_g^i$, $d(y, V(y,\er_{i}))\leq c\delta \er_{i}$, the set $T_i\cap \B{1.5 \er_{i}}{y}$ is the graph over $ V(y,\er_{i})$ of a smooth function $f$ satisfying \eqref{eq_lipTi}, where $ V(y,r)$ is one of the $k$-dimensional affine subspaces minimizing $\int_{\B{r}{y}} d^2(y, V)\,d\lambda^k$,
 \item \label{it_i6} for all $i$, we have the inclusion
 \begin{gather}\label{eq_sub}
  \supp \mu \cap B \subseteq \B{\er_{i+1}/10}{T_{i}'}\cup   R_i\, ,
 \end{gather}
\end{enumerate}
The last two properties needed are the key for the final volume estimates:
\begin{enumerate}
 \setcounter{enumi}{6} \def\theenumi{\roman{enumi}}
\item \label{it_i8} we can estimate
 \begin{gather}\label{eq_volFB}
 \lambda^k(\sigma_i^{-1}(T_i'\cap \B {2\er_j}{x}))+  \#\ton{I_b^i}\omega_k (\er_{i}/10)^k +\omega_k\sum_{x_s\in I_f^i } (r_s/10)^k\leq \lambda^k(T_{i-1}'\cap \B {2\er_j}{x})\, ,
 \end{gather}
 \item\label{it_i9} we can estimate the excess set by
 \begin{align}
 \mu\big( E(y,\er_{i})\big) \er_{i+1}^2\leq C(n) \er_{i}^{k+2}D^k_{\mu}(y,2\er_{i})\, .
 \end{align}
\end{enumerate}
\vspace{.5 cm}

At the first step of our induction, we can assume wlog that $\mu(\B {\er_j}{x}) \geq 2\gamma_k \er_j^k$. We set $I^{j}_b=I^j_f=\emptyset$, $I_g^j=\{\bar x\}$, where $\bar x$ is the center of mass of $\mu|_{\B {\er_j}{x}}$, $T_j=V(x,\er_j)=T_j'$ and $\sigma_j=id$. We set $E(x,\er_j)$ to be the excess set defined by \eqref{eq_excess_j}, and $R_j=E(x,\er_j)$. It is clear from these definitions that all the properties \eqref{it_i1}-\eqref{it_i9} are satisfied. 

Now we proceed by induction assuming that we have defined for all $t\in \cur{j,\cdots, i}$ $I_g^t,I_b^t,I_f^t$, the maps $\sigma_t$ and the manifolds $T_t,T'_t$. Moreover, we assume that we also have defined for all $t\in \cur{j,\cdots,i-1}$ the excess sets $\cur{E(y,\er_{t})}_{y\in I_g^t}$ and the remainder $R_{i-1}$.

In the induction step, we will first build $\cur{E(y,\er_{i})}_{y\in I_g^i}$ and $R_i$, and then move on to the construction of $I_g^{i+1},I_b^{i+1},I_f^{i+1}$, $\sigma_{i+1}$ and $T_{i+1},T'_{i+1}$.

\subsection{Excess set.}
Let us begin by describing the construction of the excess set.  We will only be interested here in a ball $\B{\er_{i}}{x}$ which is a good ball, in the sense that $\mu(  B_{\er_{i}}(x))\geq \gamma_k \er_{i}^k$.\\

Thus define $  V(x,r)$ to be (one of) the $k$-dimensional plane minimizing $\int_{\B{r}{x}} d(y,  V)^2 d\mu$, and define also the excess set to be the set of points which are some definite amount away from the best plane $  V$.  Precisely, 
\begin{gather}\label{eq_E}
 E(x,\er_{i})= \Big(\B{\er_{i}}{x} \setminus \B{\er_{i+1}/11}{  V}\Big)\cap S\, .
\end{gather}
The points in $\supp \mu \cap  E$ are in some sense what prevents the set $S$ from satisfying a uniform one-sided Reifenberg condition at this scale. By construction, all points in $  E$ have a uniform lower bound on the distance from $  V$, so that if we assume $\mu(  B_{\er_{i}}(x))\geq \gamma_k \er_{i}^k$, i.e. $\B{\er_{i}}{x}$ is a good ball, then we can estimate
\begin{align}\label{eq_exrough}
\int_{\B{\er_{i}}{x}\setminus   E(x,\er_{i})} &d(y,  V(x,\er_{i}))^2 \ d\mu(y)+\mu\big(  E(x,\er_{i})\big) (\er_{i+1}/11)^2\leq \int_{\B{\er_{i}}{x}} d(y,  V(x,\er_{i}))^2 \ d\mu(y)
= \er_{i}^{k+2} D^k_\mu (x,\er_{i})\, .
\end{align}
\vspace{.5 cm}

\subsection{Good, bad and final balls} Inductively, let us define the remainder set to be the union of all the previous bad balls, final balls, and the excess sets:
\begin{gather}\label{eq_R}
R_i=\bigcup_{t=j}^i\ton{\bigcup_{y\in I_b^t}\B {\er_{t}}{y} \cup \bigcup_{x_s\in I_f^t} \B{r_s}{x_s} \cup \bigcup_{y\in I_g^t}E(y,\er_{t}) }\, .
\end{gather} 
The set $  R_i$ represents everything we want to throw out at the inductive stage of the proof.  We will see later in the proof how to estimate this remainder set itself.  Note that $R_j=E(x,\er_j)$.

Now consider the points $x_s\in S$ outside the remainder set, and separate the balls $\B{r_s}{x_s}$ with radius $r_s=\er_{i+1}$ from the others by defining for $y\in I_g^i$ the sets 
\begin{gather}
I_f^{i+1}(y)=\cur{x_s\in \ton{S\setminus R_i}\cap \B{\er_{i}}{y}\ \ s.t. \ \ r_s\in [\er_{i+1},\er_{i}) }\stackrel{\eqref{eq_quantum}}{=}\cur{x_s\in \ton{S\setminus R_i}\cap \B{\er_{i}}{y}\ \ s.t. \ \ r_s = \er_{i+1} }\, ,
\end{gather}
and
\begin{gather}
J^{i+1}(y)=\cur{x_s\in \ton{S\setminus R_i}\cap \B{\er_{i}}{y}\ \ s.t. \ \ r_s<\er_{i+1}}\stackrel{\eqref{eq_quantum}}{=}\cur{x_s\in \ton{S\setminus R_i}\cap \B{\er_{i}}{y}\ \ s.t. \ \ r_s\leq\er_{i+2}}\, .
\end{gather}
From this we can construct the sets
\begin{gather}\label{eq_deph_If}
I_f^{i+1}=\cup_{y\in I_g^i} I_f^{i+1}(y)\quad \text{ and }\quad J^{i+1}=\cup_{y\in I_g^i} J^{i+1}(y)\, .
\end{gather}


Note that by construction and inductive item \eqref{it_i4}, we have 
\begin{gather}\label{e:gbf_balls:1}
 S \setminus   R_i =I^{i+1}_f\cup J^{i+1} \subset \B{\er_{i+1}/10}{T_i'} \, .
\end{gather}
Recall that, as long as we are trying to prove the estimate \eqref{eq_srx}, we can assume wlog that $\mu=\mu|_{\B {\er_j}{x}}$ and $S=\supp \mu \subset \B {\er_j}{x}$.

Let us now consider a covering of \eqref{e:gbf_balls:1} given by 
\begin{align}
 S\setminus   R_i \subseteq I^{i+1}_f \cup \bigcup_{z\in I}\B {9\er_{i+1}/10}{z}\, ,
\end{align}
where $I\subseteq T_i'$, and for any $p\neq q\in I_f^{i+1}\cup I$, $\B{r_{p}/5}{p}\cap \B{r_{q}/5}{q}=\emptyset$.  Here we denote for convenience $r_p=\er_{i+1}$ if $p\in I$, and $r_p=r_s$ if $p=x_s\in I_f^{i+1}$. Note that this second property is true by definition for $p,q\in I_f^{i+1}$, we only need to complete this partial Vitali covering with other balls of the same size. To be precise, note that by \eqref{eq_R} and \eqref{eq_E}
\begin{gather}\label{eq_333}
 \ton{S\setminus R_i } \setminus \bigcup_{z\in I_f^{i+1}}\B{3 r_z /5}{z} \subset \ton{\bigcup_{y\in I_g^i} \ton{\B {\er_{i+1}/4 }{  V(y,\er_{i})} \cap \B{9\er_{i}/10}{y}}} \bigcap \ton{\bigcup_{t=j}^{i+1}\bigcup_{z\in I^{t}_f}\B{3r_{z}/5}{z}\cup \bigcup_{t=j}^{i}\bigcup_{z\in I^{t}_b}\B{3\er_{t}/5}{z}}^C\, .
\end{gather}
Take a finite covering of this last set by balls $\cur{\B {3\er_{i+1}/10}{y}}_{y\in Y}$. Note that we can pick
\begin{gather}\label{eq_disj}
 Y\cap \ton{\bigcup_{t=j}^{i+1}\bigcup_{z\in I^{t}_f}\B{3r_{z}/5}{z}\cup \bigcup_{t=j}^{i}\bigcup_{z\in I^{t}_b}\B{3\er_{t}/5}{z}}=\emptyset\, .
\end{gather}
By item \eqref{it_i5}, $T_i$ is locally a Lipschitz graph over some $k$-dimensional subspace with \eqref{eq_lipTi}, and thus we can choose $Y\subset T_i$. 

Consider a Vitali subcovering of this set, denote $I$ the set of centers in this subcovering. Such a subcovering will have the property that the balls $\cur{\B {3\er_{i+1}/10} {y}}_{y\in I}$ will be pairwise disjoint. These balls will also be disjoint from $\bigcup_{z\in I^{i+1}_f}\B{r_{z}/5}{z}$ by \eqref{eq_disj}. The (finite version of) Vitali covering theorem ensures that $\bigcup_{y\in I}\B {9\er_{i+1}/10} {y}$ will cover the whole set in \eqref{eq_333}.

Now by construction of $I_f$ and the remainder set, all the balls $\cur{\B {r_s}{x_s}}_{s\in   S}$ with $r_s\geq \er_{i+1}$ have already been accounted for. This means that 
\begin{align}
 \ton{S\setminus R_i }\setminus \bigcup_{x_s\in I^{i+1}_f}\B{r_{s}}{x_s}\subset \bigcup_{y\in I}\B {9\er_{i+1}/10} {y}\, ,
\end{align}
as desired.


We split the balls with centers in $I$ into two subsets, according to how much measure they carry. In particular, let
\begin{gather}\label{eq_deph_Igb}
 I_g^{i+1} = \cur{y\in I \ \ s.t. \ \ \mu \ton{\B{\er_{i+1}}{y}}\geq \gamma_k \er_{i+1}^k}\, , \quad I_b^{i+1} = \cur{y\in I \ \ s.t. \ \ \mu \ton{\B{\er_{i+1}}{y}}< \gamma_k \er_{i+1}^k}\, .
\end{gather}
\vspace{.5 cm}

\subsection{Map and manifold structure.}\label{sec_sigma}
Let $\cur{\lambda_{s}^{i+1}}=\cur{\lambda_{s}}$ be a partition of unity such that for each $y_s\in I_g^{i+1}$
\begin{itemize}
 \item $\supp{\lambda_{s}}\subseteq \B{3 \er_{i+1}}{y_s}$
 \item for all $z\in \cup_{y_s\in I_g^{i+1}} \B{2 \er_{i+1}}{y_s}$, $\sum_s \lambda_{s}(z)=1$
 \item $\max_s \norm{\nabla \lambda_{s}}_\infty \leq C(n)/\er_{i+1}$.
\end{itemize}
For every $y_s\in I_g^{i+1}$, let $  V(y_s,\er_{i+1})$ to be (one of) the $k$-dimensional subspace that minimizes $\int_{\B{\er_{i+1}}{y_s}} d(z,V)^2 d\mu$. By Remark \ref{r:D_scale_control}  we can estimate
\begin{gather}
 \er_{i+1}^{-k-2}\int_{\B{\er_{i+1}}{y_s} } d(z,  V(y_s,\er_{i+1}))^2 d\mu(z) \leq D^k_\mu(y_s,\er_{i+1})\, .
\end{gather}
Let $p_s\in \B{\er_{i+1}}{y_s}$ be the center of mass of $\mu|_{\B {\er_{i+1}}{y_s}}$.  It is worth observing that $p_s\in   V(y_s,\er_{i+1})$.

Define the smooth function $\sigma_{i+1}:\R^n\to\R^n$ as in Definition \ref{deph_sigma}, i.e.,
\begin{gather}
 \sigma_{i+1}(x)= x+\sum_s \lambda_s^{i+1}(x) \pi_{  V(y_s,\er_{i+1})^\perp}\ton{p_s-x} \, .
\end{gather}

With this function, we can define inductively for $i\geq j$ the sets
\begin{align}\label{e:downward_induction:map_manifold:1}
 T_j = V(x,\er_j)\, , \quad & T_j' = T_{j}\\
 T_{i+1}=\sigma_{i+1} (T_i)\, ,\quad  & T_{i+1}'=\sigma_{i+1}\ton{T_{i}'\setminus \ton{\bigcup_{y\in I_b^{i+1}}\B{\er_{i+1}/6}{y} \cup \bigcup_{x_s\in I_f^{i+1}}\B{r_s/6}{x_s} } }\, .
\end{align}

Fix any $y\in I_g^{i+1}$, and let $z\in I_g^{i}$ be such that $\B {9\er_{i+1}/10}{y} \cap \B{9\er_{i}/10}{z}\neq \emptyset $. By induction, $T_i\cap \B{10\er_{i+1}}{y}\subseteq T_i\cap \B{1.5 \er_{i}}{z}$ is the graph of a $C^1$ function over $  V(z,\er_{i})$. Consider the points $\cur{y_t}_{t\in T_y}=I_{g}^{i+1}\cap \B{6\er_{i+1}}{y}$.
By construction it is easy to see that $d(y_t,  V(z,\er_{i}))\leq \er_{i+1}/9$, and so for all $t\in T_y$ we can apply the estimates in lemma \ref{lemma_vw} to the couple $\B {\er_{i+1}}{y_t}\subseteq \B {\er_{i}}{z}$ with $M=C_1$ by the first induction. Note that by \eqref{eq_quantum} and by construction of good and final balls, for $t\in T$ and $x_s\in \B {\er_{i+1}}{y_t}$, $r_s\leq \er_{i+2}$. Using condition \eqref{e:reifenberg_L_2_discrete} and lemma \ref{l:D_pointwise_bound}, we obtain that for all $y_s$:
\begin{gather}\label{eq_deltasquash}
 \er_{i+1}^{-1} d_H\ton{  V(z,\er_{i})\cap \B{\er_{i+1}}{y_s},  V(y_s,\er_{i+1})\cap \B{\er_{i+1}}{y_s} }\leq c\ton{D^k_{\mu} (y_s,\er_{i+1})+ D^k_{\mu} (z,\er_{i}) }^{1/2}\leq c(n,\rho,C_1)\delta\, .
\end{gather}

This implies that, if $\delta(n,\rho,C_1)$ is small enough, $T_i\cap \B{2\er_{i+1}}{y}$ is a graph also over $  V(y,\er_{i+1})$ satisfying the same estimates as in \eqref{eq_lipTi}, up to a worse constant $c$.  That is, if $\delta$ is sufficiently small, we can apply lemma \ref{lemma_squash} and prove induction point \eqref{it_i4}.\\

\paragraph{Points \eqref{it_i3} and \eqref{it_i5}} Points \eqref{it_i3} and \eqref{it_i5} are proved with similar methods. We briefly sketch the proofs of these two points.\\

Let $y\in T_{i-1}$, and recall the function $\psi_{i}\equiv 1-\sum \lambda_s$. If $\psi_{i}|_{\B{2\er_{i}}{y}}$ is identically $1$, then $\sigma_{i}|_{\B{2\er_{i}}{y}}=id$, and there is nothing to prove.

Otherwise, there must exist some $z'\in I_g^{i}\cap \B{5\er_{i}}{y}$, and thus there exists a $z\in I_g^{i-1}$ such that $\B{8\er_{i}}{y}\subseteq \B{1.5 \er_{i-1}}{z}$.  By point \eqref{it_i4} in the induction, $T_{i-1}\cap \B{1.5 \er_{i-1}}{z}$ is a Lipschitz graph over $ V(z,\er_{i-1})$. Proceeding as before, by the estimates in lemma \ref{lemma_vw} and lemmas \ref{lemma_squash}, \ref{l:D_pointwise_bound}, we obtain that $T_{i}\cap \B{2\er_{i}}{y}$ is also a Lipschitz graph over $ V(z,\er_{i-1})$ with small Lipschitz constant, and that $\norm{\sigma_{i}(p)-p}\leq c\delta \er_{i}$ for all $p\in T_{i-1}$. 

Moreover, $\sigma_{i}|_{T_{i-1}}$ is locally a diffeomorphism at scale $\er_{i}$. From this we see that $\sigma_{i}$ is a diffeomorphism on the whole $T_{i-1}$.  

It is worth to remark a subtle point. In order to prove point \eqref{it_i5}, we cannot use inductively \eqref{it_i5}, we need to use point \eqref{it_i4}. Indeed, as we have seen, given any $z\in I_g^{i-1}$, then $T_{i-1}\cap \B {1.5 \er_{i-1}}{z}$ is a Lipschitz graph of a function $f$ where $\norm{\nabla f}\leq c\delta$, and this $c$ is \textit{independent} of the induction step we are considering by \eqref{it_s3} in lemma \ref{lemma_squash}. If we tried to iterate directly the bound given by \eqref{it_i5}, the constant $c$ would depend on the induction step $i$, and thus we could not conclude the estimate we want.
\\

\subsection{\texorpdfstring{Properties of the manifolds $T_i'$}{Properties of the manifolds Ti'}}\label{sec_holes} Here we want to prove the measure estimate in \eqref{eq_volFB}. The basic idea is that bad and final balls correspond to holes in the manifold $T_i$, and each of these holes carries a $k$-dimensional measure which is proportionate to the measure inside the balls. In particular, let $\B{r}{y}$ be a bad or a final ball. In the first case, $r=\er_{i+1}$, while in the second $r\in [\er_{i+1},\er_{i})$. In either case, we will see that $y$ must be $\sim \er_{i+1}$-close to $T_i$, which is a Lipschitz graph at scale $\er_{i}$. This implies that $\mu(\B{r}{y}\cap T_i) \sim r^k$, and thus we can bound the measure of a bad or final ball with the measure of the hole we have created on $T_i$.

In detail, point \eqref{it_i6} is an immediate consequence of the definition of $  R_i$. 

In order to prove the volume measure estimate, consider that 
\begin{gather}
 T_{i}' \setminus \sigma_{i+1}^{-1}(T_{i+1}')\subseteq \ton{\bigcup_{y\in I_b^{i+1}}\B{\er_{i+1}/6}{y} \cup \bigcup_{x_s\in I_f^{i+1}}\B{r_{s}/6}{x_s} }\, .
\end{gather}
Note that the balls in the collection $\cur{\B{\er_{i+1}/5}{y}}_{y\in I_f^{i+1}\cup I_b^{i+1}}$ are pairwise disjoint. Pick any $y\in I_b^{i+1}$, and let $z\in I_g^{i}$ be such that $y\in \B{\er_{i}}{z}$. By definition, $y\in T_i'$ and $\mu(\B{\er_{i+1}}{y})< \gamma_k \er_{i+1}^k<10^{-k}\omega_k r^k_{i+1}$. Since $T_i\cap \B {2\er_{i}}{z}$ is a graph over $  V(z,\er_{i})$ with $y\in T_i$, and since $\B {\er_{i+1}/6}{y}\cap T_i$ is disjoint from the ``holes'' of $T_i'$, then
\begin{gather}
\lambda^k (T_i'\cap \B{\er_{i+1}/6}{y})\geq \omega_k 7^{-k} \er_{i+1}^k\, .
\end{gather}

A similar estimate holds for the final balls. The only difference is that if $x_s\in I_f^{i+1}$, then it is not true in general that $x_s\in T_i$. However, since $D_\mu^k(z,\er_i)\leq c\delta$ and $\mu(\cur{x_s})=\omega_k \er_{i+1}^k$, we still have that $d(x_s,V(z,\er_{i}))\leq c\delta \er_{i+1}\leq \er_{i+1}/100$. Given \eqref{eq_lipTi}, we can conclude
\begin{gather}
\lambda^k (T_i'\cap \B{\er_{i+1}/7}{x_s})\geq \omega_k 10^{-k} \er_{i+1}^k\, .
\end{gather}

Now it is evident from the definition of $T_{i+1}'$ that
\begin{gather}
 \lambda^k(\sigma_{i+1}^{-1}(T_{i+1}')\cap \B {2\er_j}{x})+  \#\ton{I_b^{i+1}\cup I_f^{i+1}}\omega_k( \er_{i+1}/10)^k \leq \lambda^k(T_{i}'\cap \B {2\er_j}{x})\, .
\end{gather}

\subsection{Volume estimates on the manifold part}\label{sec_volmani} Here we want to prove that for every measurable $\Omega\subseteq T_i$
\begin{gather}\label{eq_volT}
  \lambda^k(\sigma_{i+1}(\Omega))\leq \lambda^k(\Omega) + c(n,\rho,C_1) \int_{\B{ \er_j}{x}} D(p,2\er_{i+1}) d\mu(p)\, .
\end{gather}
The main applications will be with $\Omega=T_i$ and $\Omega=T_{i}'$. In order to do that, we need to analyze in a quantitative way the bi-Lipschitz correspondence between $T_i$ and $T_{i+1}$ given by $\sigma_{i+1}$. 

As we already know, $\sigma_{i+1}=id$ on the complement of the set $G=\cup_{y\in I_g^{i+1}} \B{5\er_{i+1}}{y} $, so we can concentrate only on this set.

Using the same techniques as before, and in particular by lemmas \ref{lemma_vw} and \ref{lemma_squash}, we can prove that for each $y\in I_g^{i+1}$, the set $ T_{i}\cap \B{5\er_{i+1}}{y}$ is a Lipschitz graph over $  V(y,\er_{i+1})$ with Lipschitz constant bounded by
\begin{gather}
 c(n,\rho,C_1) \ton{D(y,\er_{i+1}) + \sum_{z\in I_g^{i}\cap \B{5\er_{i}}{y} } D(z,\er_{i}) }^{1/2}\, .
\end{gather}
In a similar manner, we also have that $T_{i+1}\cap \B {5\er_{i+1}} {y}$ is a Lipschitz graph over $  V(y,\er_{i+1})$ with Lipschitz constant bounded by
\begin{gather}
 c(n,\rho,C_1) \ton{\sum_{z\in I_g^{i+1} \cap \B{10 \er_{i+1}}{y} } D(z,\er_{i+1}) + \sum_{z\in I_g^{i}\cap \B{5\er_{i}}{y} } D(z,\er_{i}) }^{1/2}\, .
\end{gather}
Moreover, by the bi-Lipschitz estimates of lemma \ref{lemma_squash}, we also know that $\sigma_{i+1}$ restricted to $T_i \cap \B{5\er_{i+1}}{y}$ is a bi-Lipschitz equivalence with bi-Lipschitz constant bounded by
\begin{gather}
 L(y,5\er_{i+1})\leq 1+ c\ton{\sum_{z\in I_g^{i+1} \cap \B{10 \er_{i+1}}{y} } D(z,\er_{i+1}) + \sum_{z\in I_g^{i}\cap \B{5\er_{i}}{y} } D(z,\er_{i}) }\, .
\end{gather}

In order to estimate this upper bound, we use \eqref{eq_estDint} and the definition of good balls to write
\begin{gather}
 D(z,r)\leq c\fint_{\B{r}{z}} D(p,2r)d\mu(p)\leq c(n,\rho,C_1)r^{-k}\int_{\B{r}{z}} D(p,2r)d\mu(p)\, .
\end{gather}
Since by construction any point $x\in \R^n$ can be covered by at most $c(n)$ different good balls at different scales, we can bound
\begin{gather}
 L(y,5\er_{i+1})\leq 1+\frac{c(n,\rho,C_1)}{\er_{i+1}^k} \int_{ \B {5 \er_{i}}{y}} \qua{D(p,2\er_{i+1}) + D(p,2\er_{i})}d\mu(p)\, .
\end{gather}

We can also badly estimate
\begin{gather}
 D(p,2\er_{i+1}) + D(p,2\er_{i}) \leq c(n,\rho)D(p,2\er_{i})\, .
\end{gather}

Now let $P_s$ be a measurable partition of $\Omega\cap G$ such that for each $s$, $P_s\subseteq \B{5\er_{i+1}}{y_s}$. By summing up the volume contributions of $P_s$, and since evidently $\lambda^k(P_s)\leq 7^k\omega_k \er_{i+1}^k$, we get 
\begin{gather}
 \lambda^k(\sigma_{i+1}(\Omega))= \sum_s \lambda^k(\sigma_{i+1}(P_s)) \leq \sum_s \lambda^k(P_s)\ton{1+\frac{c}{\er_{i+1}^k} \int_{ \B {5 \er_{i}}{y_s}} D(p,2\er_{i}) d\mu(p)} \notag\\
 \leq \lambda^k(\Omega) + c\int_{ \bigcup_{y_s\in I_g^{i+1}} \B {5 \er_{i}}{y_s}} D(p,2\er_{i}) d\mu(p) \notag\\
 \leq \lambda^k(\Omega) + c(n,\rho,C_1) \int_{\B{\er_j}{x}} D(p,2\er_{i}) d\mu(p)\label{eq_arg}\, .
\end{gather}

\vspace{.5 cm}

\subsection{Estimates on the excess set}
In this paragraph, we estimate the total measure of the excess set, which is defined by
\begin{gather}
   E_T = \bigcup_{i=j}^A \bigcup_{y\in I_g^i}   E(y,\er_{i})\, .
\end{gather}
At each $y$ and at each scale, we have by \eqref{eq_exrough} and \eqref{eq_estDint}
\begin{gather}
 \mu(E(y,\er_{i}))\leq c(n,\rho) \er_{i}^k D^k_{\mu}(y,\er_{i})\leq c(n,\rho) \er_{i}^k \fint_{\B{\er_{i}}{x} } D^k_{\mu}(p,2\er_{i})d\mu(p)\, .
\end{gather}
Since by definition of excess set $\B{\er_{i}}{y}$ must be a good ball, then
\begin{gather}
 \mu(E(y,\er_{i}))\leq c(n,\rho)\int_{\B{\er_{i}}{y}} D^k_{\mu}(p,2\er_{i})d\mu(p)\, . \\ \notag
\end{gather}

Now by construction of the good balls, there exists a constant $c(n)$ such that at each step $i$, each $x\in \R^n$ belongs to at most $c(n)$ good balls. Thus for each $i\geq j$, we have
\begin{gather}
 \sum_{y\in I_g^i} \mu(E(y,\er_{i}))\leq c(n,\rho) \int_{\cup_{y\in I_g^i} \B{\er_{i}}{y}} D^k_{\mu}(p,2\er_{i})d\mu(p)\leq c(n,\rho)\int_{\B{\er_j}{x} } D^k_{\mu}(p,2\er_{i})d\mu(p)\, .
\end{gather}
If we sum over all scales, we get
\begin{gather}
 \mu\ton{  E_T} \leq c(n,\rho)\sum_{i=j}^A\int_{\B{\er_j}{x} }D^k_{\mu}(p,2\er_{i})d\mu(p)\, .
\end{gather}
Since $\rho=2^q$, it is clear that
\begin{gather}\label{eq_volE}
 \mu\ton{  E_T} \leq c(n,\rho)\sum_{i=j}^A\int_{\B{\er_j}{x} }D^k_{\mu}\ton{p,2^{1-qi}}d\mu(p)\leq c(n,\rho)\delta \er_j^k\, ,
\end{gather}
since the sum in the middle is clearly bounded by \eqref{eq_sum2^alpha}.  \\

This estimate is exactly what we want from the excess set.
\vspace{.5 cm}

\subsection{Volume estimates.}  

By adding \eqref{eq_volT}, with $\Omega\equiv\sigma_{i+1}^{-1}(T_{i+1}')\cap \B {2\er_{j}}{x}$, and \eqref{eq_volFB}, we prove that for all $i=j,\cdots,A+1,...$
\begin{gather}
\lambda^k(T_{i+1}'\cap \B {2\er_j}{x})+\#\ton{I_b^{i+1}\cup I_f^{i+1}}\omega_k( \er_{i+1}/10)^k \leq \lambda^k(T_{i}'\cap \B {2\er_j}{x})+ c(n,\rho,C_1) \int_{ \B{\er_j}{x}} D(p,2\er_{i+1}) d\mu(p)\, .
\end{gather}
Adding the contributions from all scales, by \eqref{e:reifenberg_L_2_discrete} we get 
\begin{align}
 &\lambda^k(T_{i+1}'\cap \B {2\er_j}{x})+\sum_{t=j}^{i+1}\#\ton{I_b^{i+1}\cup I_f^{i+1}}\omega_k (\er_{t}/10)^k \notag \\
 &\leq \lambda^k(T_j\cap \B{2\er_j}{x})+ c(n,\rho,C_1) \sum_{s=j+1}^{i+1} \int_{ \B{\er_j}{y}\cap \B 1 0} D(p,2r_{s}) d\mu(p)\notag \\
 &\leq \lambda^k\ton{T_j\cap \B{2\er_j}{x}}\qua{1+ c(n,\rho,C_1)\delta}\,\label{eq_volTFB} \, ,
\end{align}
where in the last line we estimated $\lambda^k\ton{T_j\cap \B{2\er_j}{x}}\sim \er_j^{k}$, since $T_j$ is a $k$-dimensional subspace, and we bounded the sum using \eqref{eq_sum2^alpha}.

In the same way, we can also bound the measure of $T_{i+1}$ by
\begin{gather}
 \lambda^k(T_{i+1}\cap \B {2\er_j}{x})\leq \lambda^k\ton{T_j\cap \B{2\er_j}{x}}\qua{1+ c(n,\rho,C_1)\delta}\,\label{eq_volTFB_2} .
\end{gather}

\vspace{.5 cm}

\subsection{\texorpdfstring{Upper estimates for $\mu$.}{Upper estimates for mu}}
Since we have assumed $r_s\geq \bar r=r_A$ for all $x_s\in S$, we know that for $i=A$ our construction ends, and the whole support $S$ is contained in final and bad balls and excess sets. In other words, the sets $I_g^A=I_b^A=\emptyset$, and 
\begin{gather}
   \supp \mu \cap B = \supp \mu \cap \B{\er_j}{x} \subseteq   R_A\, .
\end{gather}
This fact and the estimates in \eqref{eq_volE} and \eqref{eq_volTFB} imply
\begin{gather}
 \mu(B) \leq \sum_{t=j}^{A}\#\ton{I_b^t} \gamma_k \omega_k \er_{t}^k + \sum_{t=j}^A \#\ton{I_f^t} \omega_k \er_{t}^k + \mu(  E_T) \notag\\
 \leq 10^k \ton{\sum_{t=j}^{A}\#\ton{I_b^t} \omega_k (\er_{t}/10)^k + \sum_{t=j}^A \sum_{x_s\in I_f^t }\omega_k (r_s/10)^k }+\mu(  E_T)\leq C_3(k)(1+c(n,\rho,C_1)\delta) \er_j^k \, .
\end{gather}
In this last estimate, we can fix $C_3(k)=20^k\omega_k$, and $C_1(k)=2C_3(k)= 2\cdot 20^k\omega_k\leq 40^k\omega_k$, and $\rho(n,C_1)$ according to lemma \ref{lemma_alpharho}. Now, it is easy to see that if $\delta(n,\rho,C_1)$ is sufficiently small, then
\begin{gather}\label{eq_volR}
 \mu(  B) \leq C_1(k) \er_j^k \, ,
\end{gather}
which finishes the proof of the downward induction, and hence the actual ball estimate \eqref{eq_srx}.
\vspace{1cm}

\section{\texorpdfstring{$L^2$-Best Approximation theorems for Stationary varifolds}{L2-Best Approximation theorems}}\label{s:best_approx}

In this Section we prove the main estimate necessary for us to be able to apply the rectifiable-Reifenberg of theorems \ref{t:reifenberg_W1p_discrete} and \ref{t:reifenberg_W1p_holes} to the singular sets $S^k_\epsilon(I)$ of the stratification induced by integral varifolds with bounded mean curvature.\\

Namely, we need to understand how to estimate on a ball $B_r(x)$ the $L^2$-distance of $S^k_\epsilon$ from the best approximating $k$-dimensional subspace.  Here we carry out the computations in the Euclidean setting with mean curvature $H=0$. The generalization to the manifold case with bounded mean curvature is straightforward.


\begin{theorem}[$L^2$-Best Approximation theorem]\label{t:best_approximation}
Let $I^m\subseteq B_{9}\subset \R^n$ be an integral varifold satisfying \eqref{e:manifold_bounds}, the mean curvature bound \eqref{e:mean_curvature_bound}, the mass bound $\mu_I(B_9(p))\leq \Lambda$, and let $\epsilon>0$.  Then there exists $\delta(n,\Lambda,\epsilon)$, $ C(n,\Lambda,\epsilon)>0$ such that if $K,H<\delta$, $r\leq 1$ and $B_{4r}(p)$ is $(0,\delta)$-symmetric but {\it not} $(k+1,\epsilon)$-symmetric, then for any finite measure $\mu$ on $I^m\cap B_r(p)$ we have that
\begin{align}\label{e:best_approx:L2_estimate}
D_\mu(p,r) = r^{-2-k} \inf_{L^k} \int_{B_r(p)} d^2(x,L^k)\,d\mu(x) \leq C r^{-k} \int_{B_r(p)} W_{8r,r}(x)\,d\mu(x)+C \delta r \frac{\mu(\B r p)}{r^k} \, , 
\end{align}
where the $\inf$ is taken over all $k$-dimensional affine subspaces $L^k\subseteq T_p M$.
\end{theorem}
\begin{remark}
The assumption $K,H<\delta$ is of little consequence, since given any $K$ this just means focusing the estimates on balls of sufficiently small radius after rescaling. Moreover, throughout this section we will assume for simplicity that $r=1$.
\end{remark}
\vspace{.5 cm}

\subsection{Symmetry and Gradient Bounds}\label{ss:symmetry_gradient}

In this subsection we study integral varifolds $I^m$ with bounded mean curvature which are {\it not} $(k+1,\epsilon)$-symmetric on some ball.  In particular, we show that this forces for each $k+1$-subspace $V^{k+1}$ that $|\pi^\perp_I[V](x)|=|\langle N_xI,V\rangle|$ has some definite size in $L^2$, where $\pi^\perp_I[V](x)$ is the projection of $V$ to the $N_xI$, the orthogonal compliment of the tangent space $T_xI$ of $I$ at $x$.  More precisely:

\begin{lemma}\label{l:best_subspace:energy_lower_bound}
Let $I^m\subseteq B_{9}$ be an integral varifold satisfying \eqref{e:manifold_bounds}, the mean curvature bound \eqref{e:mean_curvature_bound}, and the mass bound $\mu_I(B_{9}(p))\leq \Lambda$.  Then for each $\epsilon>0$ there exists $\delta(n,\Lambda,\epsilon)>0$ such that if $B_4(p)$ is $(0,\delta)$-symmetric but is {\it not} $(k+1,\epsilon)$-symmetric, then for every $k+1$-subspace $V^{k+1}\subseteq T_p M$ we have 
\begin{align}
\int_{A_{3,4}(p)} |\pi^\perp_I[V]|^2 d\mu_I \geq \delta\, ,
\end{align}
where $\pi^\perp_I[V](x)$ is the projection of $V$ to $N_xI$, the orthogonal compliment to the tangent space of $I$ at $x$.
\end{lemma}

\begin{proof}

The proof is by contradiction.  So with $n,\Lambda,\epsilon>0$ fixed let us assume the result fails.  Then there exists a sequence $I_i$ of integral varifolds of $B_{9}(p_i)$, with $B_4(p_i)$ being $(0,\delta_i)$-symmetric but not $(k+1,\epsilon)$-symmetric, and such that for some subspaces $V_i^{k+1}$ we have that
\begin{align}\label{e:best_subspace_lower_energy:1}
\int_{A_{3,4}(p_i)} |\pi^\perp_{I_i}[V_i]|^2 d\mu_{I_i} \leq \delta_i\to 0\, .
\end{align}
After rotation we can assume $V^{k+1}_i=V^{k+1}$, and then after passing to a subsequence we have that
\begin{align}
I_i\longrightarrow I\subseteq B_{9}(0^n)\, ,
\end{align}
in the sense of varifolds (or in the sense of the flat distance if $I$ is an integral current).  In particular, we have that $I$ is a stationary varifold and
\begin{align}
\int_{A_{3,4}(p_i)} |\pi^\perp_{I}[V]|^2 d\mu_{I} = 0\, .
\end{align}

On the other hand, the $(0,\delta_i)$-symmetry of the $I_i$ tells us that $I$ is $0$-symmetric.  Combining these tells us that 
\begin{align}
\int_{B_4(0^n)}|\pi^\perp_{I}[V]|^2 d\mu_{I} = 0\, ,
\end{align}
and hence we have that $I$ is $k+1$-symmetric.  Because the convergence $I_i\to I$ is in the varifold sense, this contradicts that the $I_i$ are not $(k+1,\epsilon)$-symmetric for $i$ sufficiently large, which proves the lemma.

\end{proof}
\vspace{1cm}

\subsection{\texorpdfstring{Best $L^2$-Subspace Equations}{Best L2-Subspace Equations}}\label{ss:best_subspaces}

In order to prove theorem \ref{t:best_approximation} we need to identify which subspace minimizes the $L^2$-energy, and the properties about this subspace that allow us to estimate the distance.  We begin in Section \ref{sss:second_moments}  by studying some very general properties of the second directional moments of a general probability measure $\mu\subseteq B_1(p)$.  We will then study in Section \ref{sss:projection_tangent} a quantitative notion of almost-symmetry for stationary varifolds.\\

\subsubsection{Second Directional Moments of a Measure}\label{sss:second_moments}

Let us consider a probability measure $\mu\subseteq B_1(0^n)$, and let
\begin{align}
x^i_{cm}=x^i_{cm}(\mu)\equiv \int x^i \, d\mu(x)\, ,
\end{align}
be the center of mass.  Let us inductively consider the maximum of the second directional moments of $\mu$.  More precisely:

\begin{definition}\label{d:second_moments}
Let $\lambda_1=\lambda_1(\mu)\equiv \max_{|v|^2=1} \int |\langle x-x_{cm},v\rangle|^2\, d\mu(x)$
and let $v_1=v_1(\mu)$ with $|v_1|=1$ be any vector obtaining this maximum.  Now let us define inductively the pair $(\lambda_{k+1},v_{k+1})$ from $v_1,\ldots,v_{k}$ by
\begin{align}
\lambda_{k+1}=\lambda_{k+1}(\mu)\equiv \max_{|v|^2=1,\ps{v}{v_i}=0 \ \forall i\leq k } \int |\langle x-x_{cm},v\rangle|^2\, d\mu(x)\, ,
\end{align}
where $v_{k+1}$ is any vector obtaining this maximum.
\end{definition}

Thus $v_1,\ldots,v_n$ defines an orthonormal basis of $\dR^n$, ordered so that they maximize the second directional moments of $\mu$.  Let us define the subspaces
\begin{align}\label{e:best_subspace:Vk}
V^k= V^k(\mu)\equiv x_{cm}+\text{span}\{v_1,\ldots,v_k\}\, .
\end{align}
The following is a simple but important exercise:

\begin{lemma}\label{l:best_subspace_Vk}
If $\mu$ is a probability measure in $B_1(0^n)$, then for each $k$ the functional 
\begin{align}
\min_{L^k\subseteq \dR^n} \int d^2(x,L^k)\,d\mu(x)\, ,
\end{align}
where the $\min$ is taken over all $k$-dimensional affine subspaces, attains its minimum at $V^k$.  Further, we have that
\begin{align}
\min_{L^k\subseteq \dR^n} \int d^2(x,L^k)\,d\mu(x) = \int d^2(x,V^k)\,d\mu(x) = \lambda_{k+1}(\mu)+\cdots+\lambda_n(\mu)\, .
\end{align}
\end{lemma}
Note that the best affine subspace $V^k$ will necessarily pass through the center of mass $x_{cm}$.

\vspace{.5cm}

Now let us record the following Euler-Lagrange formula, which is also an easy computation:

\begin{lemma}\label{l:best_subspace:euler_lagrange}
If $\mu$ is a probability measure in $B_1(0^n)$, then we have that $v_1(\mu),\ldots,v_n(\mu)$ satisfy the Euler-Lagrange equations:
\begin{align}\label{e:second_moment_EL}
\int \langle x-x_{cm},v_k\rangle (x-x_{cm})^i\, d\mu(x) = \lambda_k v_k^i \, ,
\end{align}
where 
\begin{align}
&\lambda_k = \int |\langle x-x_{cm},v_k\rangle|^2\, d\mu(x)\, .
\end{align}
\end{lemma}
%
%

\vspace{1cm}

\subsubsection{Projections on the tangent spaces}\label{sss:projection_tangent}

Our goal now is to study an integral varifold $I^m$ in the directions spanned by $v_1(\mu),\ldots,v_n(\mu)$, associated to a probability measure.  The main result of this subsection is the following, which holds for a general integral varifold with bounded mean curvature.
We recall that by definition
\begin{align}
W_{\alpha}(x) \equiv W_{\er_{\alpha},\er_{\alpha-3}}(x)\equiv \theta_{\er_{\alpha-3}}(x)-\theta_{\er_{\alpha}}(x)\geq 0\, ,
\end{align}
where $\er_\alpha=2^{-\alpha}$.

\begin{proposition}\label{p:best_subspace_estimates}
Let $I^m\subseteq B_{9}$ be an integral varifold 
with mass bound $\mu_I(B_{9}(p))\leq \Lambda$.  Let $\mu$ be a probability measure on $B_1(p)$ with $\lambda_k(\mu),v_k(\mu)$ defined as in Definition \ref{d:second_moments}.  Then there exists $\delta(n)>0$ such that 
\begin{align}
\lambda_k \int_{A_{3,4}(p)} |\pi^\perp_I[v_k]|^2\,d\mu_I(z) \leq \delta^{-1} \int W_0(x)\, d\mu(x)\, ,
\end{align}
where $\pi^\perp[v](x)$ is the projection of $v$ to $N_xI$, the orthogonal compliment of the tangent space $T_xI$, which exists a.e.
\end{proposition}
\begin{proof} 
Note first that there is no harm in assuming that $x_{cm}\equiv 0$.  If not we can easily translate to make this so, in which case we still have that $\text{supp}(\mu)\subseteq B_2$.  Additionally, we will simplify the technical aspect of the proof by assuming that $M\equiv \dR^n$ and $H=0$.  By working in normal coordinates the proof of the general case is no different except up to some mild technical work.  \\

Now let us fix any $z\in A_{3,4}$ in the support of $I$ and such that $N_z I$ is well defined (note that this second set coincides with the support of $I$ up to a set of $\mu_I$ measure zero). Observe that
\begin{align}
\int \langle x,v_k\rangle\, d\mu(x) = \langle x_{cm},v_k\rangle = 0\, . 
\end{align}
Then combined with lemma \ref{l:best_subspace:euler_lagrange} we can inner product both sides of \eqref{e:second_moment_EL} by $N_zI$, the normal to the tangent space of $I$ at $z$, to obtain for each $k$ and $ A_{3,4}$:
\begin{align}
\lambda_k\langle N_zI, v_k\rangle &= \int \langle x,v_k\rangle\langle N_zI,x\rangle\,d\mu(x) =\int \langle x,v_k\rangle\langle N_zI,x-z\rangle\,d\mu(x) \, .
\end{align}
We can then estimate
\begin{align}
\lambda_k^2|\langle N_zI, v_k\rangle|^2\leq \lambda_k\int |\langle N_zI,x-z\rangle|^2\,d\mu(x) \, .
\end{align}
Integrating with respect to $z$ on both sides we get the estimate
\begin{align}\label{e:best_subspace_estimate:1}
\lambda_k \int_{A_{3,4}}|\langle N_zI, v_k\rangle|^2\,d\mu_I(z)\leq &\int\int_{A_{3,4}} |\langle N_zI,x-z\rangle|^2\,d\mu_I(z)\,d\mu(x)\, .
\end{align}
Set for convenience $n_x(z)=(z-x)/\norm{z-x}$, i.e., $n_x(z)$ is the radial vector from $x$ to $z$. Now for $x\in \text{supp}(\mu)$ we can estimate
\begin{align}
\int_{A_{3,4}} |\langle N_zI,x-z\rangle|^2\,d\mu_I(z) &= \int_{A_{3,4}} |\langle N_zI,n_x(z)\rangle|^2|x-z|^{-m}|x-z|^{2+m}\,d\mu_I(z)\notag\\
&\leq C(n)\int_{A_{3,4}} |\langle N_zI,n_x(z)\rangle|^2|x-z|^{-m}\,d\mu_I(z)\notag\\
&\leq C(n)\int_{A_{1,8}(x)} |\langle N_zI,n_x(z)\rangle|^2|x-z|^{-m}\,d\mu_I(z) \notag\\
&= C(n)W_0(x)\, .
\end{align}
Applying this to \eqref{e:best_subspace_estimate:1} we get the estimate
\begin{align}\label{e:best_subspace_estimate:2}
\lambda_k \int_{A_{3,4}}|\langle N_zI, v_k\rangle|^2\,d\mu_I(z)\leq &C(n)\int W_0(x)\,d\mu(x)\, .
\end{align}
%

\end{proof}
\vspace{1cm}

\subsection{Proof of theorem \ref{t:best_approximation}}\label{ss:proof_L2_best}

Let us now combine the results of this Section in order to prove theorem \ref{t:best_approximation}.  Indeed, let $\mu$ be a measure in $B_1(p)\subseteq T_p M$. We can assume that $\mu$ is a probability measure without any loss of generality, since both sides of our estimate scale.  Let $\big(\lambda_1(\mu),v_1(\mu)\big),\ldots,\big(\lambda_n(\mu),v_n(\mu)\big)$ be the directional second moments as defined in Definition \ref{d:second_moments}, with $V^k$ the induced subspaces defined as in \eqref{e:best_subspace:Vk}.  Using lemma \ref{l:best_subspace_Vk} we have that
\begin{align}\label{e:best_subspace:proof:1}
\min_{L^k\subseteq \dR^n} \int d^2(x,L^k)\,d\mu(x) = \int d^2(x,V^k)\,d\mu(x) = \lambda_{k+1}(\mu)+\cdots+\lambda_n(\mu)\leq (n-k)\lambda_{k+1}(\mu)\, ,
\end{align}
where we have used that $\lambda_j\leq \lambda_i$ for $j\geq i$.  Therefore our goal is to estimate $\lambda_{k+1}$.  To begin with, proposition \ref{p:best_subspace_estimates} tells us that for each $j$ 
\begin{align}
\lambda_j \int_{A_{3,4}(p)} |\langle N_zI, v_j\rangle|^2\,d\mu_I(z) \leq C \int W_0(x)\, d\mu(x)\, .
\end{align}
Let us sum the above for all $j\leq k+1$ in order to obtain
\begin{align}
\sum_{j=1}^{k+1}\lambda_j\int_{A_{3,4}(p)} |\langle N_zI, v_j\rangle|^2\,d\mu_I(z) \leq (k+1)C \int W_0(x)\, d\mu(x)\, ,
\end{align}
or by using that $\lambda_{k+1}\leq \lambda_{j}$ for $k+1\geq j$ we get
\begin{align}\label{e:best_subspace:proof:2}
\lambda_{k+1}\int_{A_{3,4}(p)} |\langle N_zI, V^{k+1}\rangle|^2\,d\mu_I(z)= \lambda_{k+1} \sum_{j=1}^{k+1}\int_{A_{3,4}(p)} |\langle N_zI, v_j\rangle|^2\,d\mu_I(z)\leq C \int W_0(x)\, d\mu(x)\, .
\end{align}

Now we use that $B_{9}(p)$ is $(0,\delta)$-symmetric, but {\it not} $(k+1,\epsilon)$-symmetric in order to apply lemma \ref{l:best_subspace:energy_lower_bound} and conclude that
\begin{align}
\int_{A_{3,4}(p)} |\langle N_zI, V^{k+1}\rangle|^2\,d\mu_I(z)\geq \delta\, .
\end{align}
Combining this with \eqref{e:best_subspace:proof:2} we obtain
\begin{align}\label{e:best_subspace:proof:3}
\delta\lambda_{k+1}\leq \lambda_{k+1} \int_{A_{3,4}(p)} |\langle N_zI, V^{k+1}\rangle|^2\,d\mu_I(z)\leq C \int W_0(x)\, d\mu(x)\, ,
\end{align}
or that
\begin{align}
\lambda_{k+1}\leq C(n,\Lambda,\epsilon) \int W_0(x)\, d\mu(x)\, .
\end{align}
Combining this with \eqref{e:best_subspace:proof:1} we have therefore proved the theorem. \hfill $\square$
\vspace{1cm}

\section{The Inductive Covering lemma}\label{s:covering_main}

This Section is dedicated to the basic covering lemma needed for the proof of the main theorems of the paper.  The covering scheme is similar in nature to the one introduced by the authors in \cite{NaVa,NaVa+} in order to prove structural theorems on critical and singular sets.  Specifically, let us consider an integral varifold $I^m$ with bounded mean curvature, then we will build a covering of the quantitative stratification
\begin{align}
S^k_{\epsilon,r}(I)\cap B_1(p)\subseteq U_r\cup U_+ = U_r\cup\bigcup B_{r_i}(x_i)\, ,
\end{align}
which satisfies several basic properties.  To begin with, the set $U_+$ is a union of balls satisfying $r_i>r\geq 0$, and should satisfy the packing estimate $\omega_k\sum r_i^k\leq C$.  Each ball $B_{r_i}(x_i)$ should have the additional property that there is a definite mass drop of $I$ when compared to $B_2(p)$.  To describe the set $U_r$ we should distinguish between the case $r>0$ and $r\equiv 0$.  In the case $r>0$ we will have that $U_r = \bigcup B_{r}(x^r_i)$ is a union of $r$-balls and satisfies the Minkowski estimate $\Vol(U_r)\leq C r^{n-k}$.  In the case when $r\equiv 0$ we will have that $U_0$ is $k$-rectifiable with the Hausdorff estimate $\lambda^{k}(U_0)\leq C$.  Let us be more precise:\\

\begin{lemma}[Covering lemma]\label{l:covering}
Let $I^m\subseteq B_2$ be an integral varifold satisfying the bounds \eqref{e:manifold_bounds}, the mean curvature bound \eqref{e:mean_curvature_bound}, the mass bound $\mu_I(B_2(p))\leq \Lambda$, and assume $K+H<\delta(n,\Lambda,\epsilon)$.  Let $E = \sup_{x\in B_1(p)} \theta_1(x)$ with $\epsilon>0$, $r\geq 0$, and $k\in\dN$.  Then for all $\eta\leq \eta(n,\Lambda,\epsilon)$ and $0<r \leq 1$, there exists a covering $S^k_{\epsilon,r/100}(I)\cap B_1(p)\subseteq U = U_r\cup U_+$ such that 
\begin{enumerate}
\item $U_+=\bigcup B_{r_i}(x_i)$ with $r_i>r$ and $\sum r_i^k\leq C(n,\Lambda,\epsilon)$ .
\item $\sup_{y\in B_{ r_i}(x_i)}\theta_{\eta r_i/2}(y)\leq E-\eta$.
\item If $r>0$ then $U_r=\bigcup_1^N B_{r}(x^r_i)$ with $N\leq C(n) r^{-k}$.
\item If $r=0$ then $U_0$ is $k$-rectifiable and satisfies $\Vol(B_s\, \ton{U_0})\leq C(n) s^{n-k}$ for each $s>0$.\newline  In particular, $\lambda^k(U_0)\leq C(n)$.
\end{enumerate}
\end{lemma}
\begin{remark}
The assumption $K+H<\delta$ is of little consequence, since given any $K$ this just means focusing the estimates on balls of sufficiently small radius after rescaling.
\end{remark}

\begin{remark}
 As in the previous section, throughout this section we will assume that the ambient manifold $M$ is actually Euclidean and $I$ has zero mean curvature, that is $K=H=0$. The generalization to the general case is straightforward, but it would involve minor details that would anyway add another layer of technicalities.
\end{remark}

\begin{remark}
 Note that, up to enlarging the constants in the covering by some $C(n)$, the conclusions of this lemma clearly hold also for the set $S^k_{\epsilon,r}$, not just for the slightly smaller set $S^k_{\epsilon,r/100}$. We state this lemma with the factor $10^{-2}$ just for technical reasons.
\end{remark}

\vspace{.5 cm}

To prove the result let us begin by outlining the construction of the covering, we will then spend the rest of this section proving the constructed cover has all the desired properties.\\

Thus let us consider some $\eta>0$ fixed, and then define the mass scale for $x\in B_1(p)$ by
\begin{align}
s_x=s^{E,\eta}_x \equiv \inf\cur{r\leq t\leq 1: \sup_{B_s(x)\cap S^k_{\epsilon,r/100}}\theta_{\eta s/2}(y)\geq E-\eta\ \ \text{for all} \ \ t\leq s\leq 1}\, .
\end{align}
Note that the mass scale implicitly depends on many constants.  If $r=0$ let us define the set $U_0$ by
\begin{align}
U_0\equiv \big\{x\in S^k_{\epsilon,0}(I)\cap B_1(p): s_x=0\big\}\, ,
\end{align}
while if $r>0$ let us define
\begin{align}
U_r \equiv \bigcup B_{r}(x^r_i)\, ,
\end{align}
where 
\begin{align}
\{x_i^r\}\subseteq \big\{x\in S^k_{\epsilon,r/100}(I)\cap B_1(p): s_x=r\big\}\, ,
\end{align}
is such that $\{B_{r/8}(x^r_i)\}$ are a maximal pairwise disjoint collection of balls centered on $S^k_{\epsilon,r/100}(I)\cap B_1(p)$.\\

In order to define the covering $U_+=\{B_{r_i}(x_i)\}$ let us first consider the covering
\begin{align}
\big\{x\in S^k_{\epsilon,r/100}(I)\cap B_1(p): s_x>r\big\}\subseteq \bigcup_{s_x>r} B_{s_x/10}(x)\, ,
\end{align}
and choose from it a Vitali subcovering and set
\begin{align}
U_+\equiv \bigcup_{i\in J} B_{r_i}(x_i)\, ,
\end{align}
where $r_i\equiv s_{x_i}/2$ and $\{B_{r_i/5}(x_i)\}$ are all disjoint.  It is clear $U_+$ satisfies the version of $(2)$ given by
\begin{align}
(2')\quad \sup_{y\in B_{2r_i}(x_i)}\theta_{\eta r_i/2}(y)\leq E-\eta\, .
\end{align}
It is clear from the construction that we have built a covering
\begin{align}
S^k_{\epsilon,r/100}(I)\cap B_1(p) \subseteq U_r\cup U_+\, ,
\end{align}
so we now need to study the properties of $U_r$ and $U_+$. 

Since $\eta(n,K_N,\Lambda,\epsilon)$ is fixed, we can re-cover in a trivial way all balls $B_{r_i}(x_i)$ with smaller balls of radius $\eta r_i$ and obtain for this new covering that 
\begin{enumerate}
\item $r^{k-n}\Vol(B_r\, U_r)+\omega_k\sum r_i^k\leq C(n,K_N,\Lambda,\epsilon)\, s^k$ ,
\item $\sup_{y\in B_{r_i}(x_i)}\theta_{r_i}(y)\leq E_{x,s}-\eta$ ,
\end{enumerate}
as desired.\\

The outline of this Section is as follows.  In Sections \ref{ss:covering_U_r} and \ref{ss:covering_U_0} we will deal with estimating the set $U_r$.  In fact, from a technical standpoint $U_r$ is much easier than $U_+$ to deal with, and for the rectifiability and local $\lambda^k$-finiteness of $S^k_{\epsilon,r/100}$ this is the set which is most important.  Indeed, for $U_r$ we will be able to almost directly use the rectifiable-Reifenberg of theorems \ref{t:reifenberg_W1p_holes} and \ref{t:reifenberg_W1p_discrete}, at least when combined with an additional induction argument.  However, for the global Minkowski estimates and $k$-dimensional Hausdorff measure bounds on $S^k_{\epsilon,r/100}$ it is crucial to deal with the set $U_+$ as well.  To do this we first prove a variety of technical lemmas in Section \ref{ss:covering:energy_drop}, these will allow us to exchange $U_+$ for a more manageable collection of balls without losing much content.  Then in Section \ref{ss:covering_U_+} we will be able argue as in the $U_r$ case by applying the rectifiable Reifenberg of theorem \ref{t:reifenberg_W1p_discrete}, however we will apply it not to $U_+$ but to a more carefully chosen covering obtained by exploiting the results of Section \ref{ss:covering:energy_drop}.\\

\subsection{\texorpdfstring{Estimating $U_r$ in lemma \ref{l:covering} for $r>0$}{Estimating Ur in covering lemma  for r>0}}\label{ss:covering_U_r}

Let us begin by altering the collection $U_r=\bigcup B_{r}(x^r_i)$ slightly.  Indeed, by definition of $U_r$, we have that for each ball $B_r(x^r_i)$ there exists a point $x'_i\in B_r(x^r_i)$ such that $\theta_{\eta r}(x'_i)\geq E-\eta$.  Then we have the covering
\begin{align}
U_r\subseteq \bigcup_1^N B_{2r}(x'_i)\, ,
\end{align}
and let $U'_r\equiv \bigcup_1^{N'} B_{2r}(x'_i)$ be a maximal disjoint subset of these balls.  Let us first observe that by a standard covering argument we have $N\leq C(n)N'$, and thus to estimate $U_r$ it is enough to estimate $N'$.  Indeed, by the maximality of $U'_r$ we have that for every ball center $x^r_i\in U_r$ there exists a ball center $x'_j\in U'_r$ such that $x^r_i\in B_{4r}(x'_j)$ and $\theta_{\eta r}(x'_j)\geq E-\eta$.  Since the balls $\{B_{r/8}(x^r_i)\}$ are disjoint, we have that
\begin{align}
N\cdot \omega_n \ton{\frac{r}{8}}^n\leq \Vol(U_r)\leq \sum_1^{N'} \Vol(B_{4r}(x'_j)) \leq N'\cdot \omega_n (4r)^n\, , 
\end{align}
which in particular gives us the desired estimate $N\leq 32^n N'$.\\

Now let us proceed to estimate $N'$.  Thus consider the measure associated to $U'_r$ given by
\begin{align}
\mu = \sum \omega_k r^k \delta_{x'_i}\, .
\end{align}
Let us consider the sequence of radii $\er_\alpha = 2^{-\alpha}$.  We now wish to prove that for each $x\in B_1$ and all $r\leq \er_\alpha\leq 2^{-7}$ that
\begin{align}\label{e:covering:U_r:1}
\mu(B_{\er_\alpha}(x))\leq D(n) \er_\alpha^k\, ,
\end{align}
where $D(n)$ is from theorem \ref{t:reifenberg_W1p_discrete}.  Let us first observe that once \eqref{e:covering:U_r:1} has been proved then we have the desired estimate on $N'$.  Indeed, by a simple covering argument we obtain that $ \mu(B_{1}(0))\leq C(n)\cdot D(n) $, which in turn implies 
\begin{align}
N'\omega_n r^k = \sum_1^{N'} \omega_k r^k = \mu(B_{1})\leq C(n)\cdot D(n)\, , 
\end{align}
and this gives the claimed estimate $N'\leq C(n)r^{-k}$.\\

Thus let us now concentrate on proving \eqref{e:covering:U_r:1}.  We will prove it by induction on $\alpha$.  Let $\alpha_0$ be such that $r\leq \er_{\alpha_0}<2r$.  Let us begin by observing that the result clearly holds for $\er_{\alpha_0}$.  Therefore let us now assume we have proved \eqref{e:covering:U_r:1} for some $\er_{\alpha+1}$, and then proceed to prove it for $\er_{\alpha}$.\\

Let us first observe a rough estimate in this direction through a covering argument.  Namely, let us consider the radii $\er_{\alpha+1}\leq s\leq 4 \er_{\alpha}$. We can cover $B_s(x)$ by a collection of at most $C(n)$ balls $\{B_{\er_{\alpha+1}}(y_i)\}$, and thus by the inductive assumption we have for all $x\in B_1$ and $s\leq 4\er_{\alpha}$ that
\begin{align}\label{eq_U'rough}
\mu(B_s(x))\leq \sum \mu(B_{\er_{\alpha+1}}(y_i))\leq D(n) C(n)\er_\alpha^k\leq D'(n) s^k\, ,
\end{align}
where of course $D'(n)>>D(n)$.\\


Now let us consider a ball $B_s(y)\subseteq B_2$ with $s\leq 2\er_{\alpha}$.  If $\mu(B_s(y))\leq \epsilon_n s^k$ then $D_\mu(y,s)\equiv 0$ by definition, while if $s\leq r$ then we also have that $D_\mu(y,s)\equiv 0$, since the support of $\mu$ in $B_r(x)$ contains at most one point and thus is precisely contained in a $k$-dimensional subspace.  Thus let us consider the case when $s>r$ and $\mu(B_s(y))> \epsilon_n s^n$.  In this case notice by theorem \ref{t:quantitative_0_symmetry} that for all the points $y\in \supp \mu\cap \B s y$, the ball $\B {32s}{y}$ is $(0,\delta)$-symmetric, where $\delta=\delta(\eta|n,\Lambda)\to 0$ as $\eta\to 0$.  Thus for $\eta\leq \eta(n,\Lambda)$ we can apply theorem \ref{t:best_approximation} to see that
\begin{align}
D_\mu(y,s) \leq C(n,\Lambda,\epsilon)s^{-k}\int_{B_s(y)} W_s(z)\,d\mu(z)\, .
\end{align}
By applying this to all $r\leq t\leq s$  we have 
\begin{align}
s^{-k}\int_{B_s(x)} D_\mu(y,t)\,d\mu(y)&\leq Cs^{-k}\int_{B_s(x)}t^{-k}\int_{B_t(y)} W_t(z)\,d\mu(z)\,d\mu(y)\notag\\
&= Cs^{-k}t^{-k}\int_{B_{2s}(x)} \mu(B_t(z)) W_t(z)\,d\mu(z)\notag\\
&\leq Cs^{-k}\int_{B_{2s}(x)}W_t(z)\,d\mu(z)\, ,
\end{align}
where we have used our rough estimate \eqref{eq_U'rough} in the last line.  Let us now consider the case when $t=\er_\beta\leq s\leq 2\er_{\alpha}$.  Then we can sum to obtain:
\begin{align}
\sum_{\er_\beta\leq s} s^{-k}\int_{B_s(x)} D_\mu(y,\er_\beta)\,d\mu(y) &\leq C\sum_{r\leq \er_\beta\leq s} s^{-k}\int_{B_s(x)}W_{\er_\beta}(y)\,d\mu(y)\notag\\
&=C s^{-k}\int_{B_s(x)}\sum_{r\leq \er_\beta\leq s}W_{\er_\beta}(y)\,d\mu(y)\notag\\
&\leq C\, s^{-k}\int_{B_s(x)}\big|\theta_{8s}(y)-\theta_r(y)\big| \,d\mu(y)\leq C(n,\Lambda,\epsilon)\eta\, ,
\end{align}
where in the last line we have used both the rough estimate \eqref{eq_U'rough} on $\mu(B_s(x))$ and that in the support of $\mu$ we have that $\big|\theta_{8s}(y)-\theta_r(y)\big|\leq \eta$ by the construction of $U'_r$.  Now let us choose $\eta\leq \eta(n,\Lambda,\epsilon)$ such that we have
\begin{align}\label{e:covering:U_r:2a}
\sum_{\er_\beta\leq s} s^{-k}\int_{B_s(x)} D_\mu(y,\er_\beta)\,d\mu(y) &\leq \delta^2\, ,
\end{align}
where $\delta$ is taken from theorem \ref{t:reifenberg_W1p_discrete}.  Since the estimate \eqref{e:covering:U_r:2a} holds for all $B_{s}\subseteq B_{2\er_{\alpha}}(x)$, we can therefore apply theorem \ref{t:reifenberg_W1p_discrete} to conclude the estimate
\begin{align}
\mu(B_{\er_{\alpha}}(x))\leq D(n) \er_{\alpha}^k\, .
\end{align}
This finishes the proof of \eqref{e:covering:U_r:1}, and hence the proof of estimate of $U_r$ for $r>0$ in lemma \ref{l:covering}. \hfill $\square$
\vspace{1cm}

\subsection{\texorpdfstring{Estimating $U_0$ in lemma \ref{l:covering}}{Estimating U0 in covering lemma}}\label{ss:covering_U_0}

Let us begin by proving the Minkowski estimates on $U_0$.  Indeed, observe that for any $r>0$ that $U_0\subseteq U_r$, and thus we have the estimate
\begin{align}
\Vol(B_r\,\ton{U_0})\leq \Vol(B_r\, \ton{U_r})\leq \omega_n (2r)^n\cdot N\leq C(n,\Lambda,\epsilon) r^{n-k}\, ,
\end{align}
which proves the Minkowski claim.  In particular, we have as a consequence the $k$-dimensional Hausdorff measure estimate
\begin{align}
\lambda^k(U_0)\leq C(n,\Lambda,\epsilon)\, .
\end{align}
In fact, let us conclude a slightly stronger estimate, since it will be a convenient technical tool in the remainder of the proof.  If $B_s(x)$ is any ball with $x\in B_1$ and $s<\frac{1}{2}$, then by applying the same proof to the rescaled ball $B_s(x)\to B_1(0)$, we can obtain the Hausdorff measure estimate
\begin{align}\label{eq_U0est}
\lambda^k(U_0\cap B_s(x))\leq C s^k\, .
\end{align}

To finish the construction we need to see that $U_0$ is rectifiable.  We will in fact apply theorem \ref{t:reifenberg_W1p_holes} in order to conclude this.  To begin with, let $\mu\equiv \lambda^k\big|_{U_0}$ be the $k$-dimensional Hausdorff measure, restricted to $U_0$.  Let $B_s(y)$ be a ball with $y\in B_1\cap \supp \mu $ and $s<\frac{1}{10}$, now we will now argue in a manner similar to Section \ref{ss:covering_U_r}.  Thus, if $\mu(B_s(y))\leq \epsilon_n s^k$ then $D_{\mu}(y,s)\equiv 0$, and otherwise we then have by theorem \ref{t:best_approximation} that for $\eta\leq \eta(n,\Lambda)$
\begin{align}
D_\mu(y,s) \leq C(n,\Lambda,\epsilon)s^{-k}\int_{B_s(y)} W_s(z)\,d\mu(z)\, .
\end{align}
By applying this to all $t\leq s$  we have 
\begin{align}
s^{-k}\int_{B_s(x)} D_\mu(y,t)\,d\mu(y)&\leq Cs^{-k}\int_{B_s(x)}t^{-k}\int_{B_t(y)} W_t(z)\,d\mu(z)\,d\mu(y)\notag\\
&= Cs^{-k}t^{-k}\int_{B_{2s}(x)} \mu(B_t(z)) W_t(z)\,d\mu(z)\notag\\
&\leq Cs^{-k}\int_{B_{2s}(x)}W_t(z)\,d\mu(z)\, ,
\end{align}
where we have used our estimate \eqref{eq_U0est} in the last line.  Let us now consider the case when $t=\er_\beta=2^{-\beta}\leq s$.  Then we can sum to obtain:
\begin{align}
\sum_{\er_\beta\leq s} s^{-k}\int_{B_s(x)} D_\mu(y,\er_\beta)\,d\mu(y) &\leq C\sum_{\er_\beta\leq s} s^{-k}\int_{B_s(x)}W_{\er_\beta}(y)\,d\mu(y)\notag\\
&=C s^{-k}\int_{B_s(x)}\sum_{\er_\beta\leq s}W_{\er_\beta}(y)\,d\mu(y)\notag\\
&\leq C\, s^{-k}\int_{B_s(x)}\big|\theta_{8s}(y)-\theta_0(y)\big|\,d\mu(y)\leq C(n,\Lambda,\epsilon)\eta\, ,
\end{align}
where we have used two points in the last line.  First, we have used our estimate $\mu(B_s(x))\leq Cs^k$.  Second, we have used that by the definition of $U_0$, for each point in the support of $\mu$ we have that $\big|\theta_s(y)-\theta_0(y)\big|\leq \eta$.  Now let us choose $\eta\leq \eta(n,\Lambda,\epsilon)$ such that we have
\begin{align}\label{e:covering:U_r:2b}
\sum_{\er_\beta\leq s} s^{-k}\int_{B_s(x)} D_\mu(y,\er_\beta)\,d\mu(y) &\leq \delta^2\, ,
\end{align}
where $\delta$ is chosen from theorem \ref{t:reifenberg_W1p_holes}.  Thus, by applying theorem \ref{t:reifenberg_W1p_holes} we see that $U_0$ is rectifiable, which finishes the proof of lemma \ref{l:covering} in the context of $U_0$.  \hfill $\square$
\vspace{1cm}

\subsection{\texorpdfstring{Technical Constructions for Estimating $U_+$}{Technical Constructions for Estimating U+}}\label{ss:covering:energy_drop}

Estimating the set $U_+$ is in spirit similar to the estimates obtained in the last subsections on $U_r$ and $U_0$.  However, the estimate on $U_+$ itself is a bit more delicate, and we cannot directly apply the discrete Reifenberg of theorem \ref{t:reifenberg_W1p_discrete} to this set.  Instead, we will need to replace $U_+$ with a different covering at each stage, which will be more adaptable to theorem \ref{t:reifenberg_W1p_discrete}.  This subsection is dedicated to proving a handful of technical results which are important in the construction of this new covering.\\

Throughout this subsection we are always working under the assumptions of lemma \ref{l:covering}.  Let us begin with the following point, which is essentially a consequence of the continuity of the mass:

\begin{lemma}\label{l:covering:energy_bound}
For each $\eta'>0$ there exists $R(n,\Lambda,\eta')>3$ such that with $\eta\leq \eta(n,\Lambda,\eta')$ we have for each $z\in \B {r_i}{x_i}$ with $\B {\eta Rr_i}{z}\subseteq \B 2 0$ the estimate
\begin{align}
\theta_{\eta Rr_i}(z)>E-\eta'\, .
\end{align}
\end{lemma}
\begin{proof}
The proof relies on a straight forward comparison using the definition and monotonicity of the mass.  Namely, let $x,y\in B_{1/2}$ with $s<1$ and let us denote $d\equiv d(x,y)$.  Then we have the estimate
\begin{align}\label{e:covering:energy_bound:1}
\theta_s(y) = s^{-k}\int_{B_s(y)}d\mu_I \leq s^{-k}\int_{B_{s+d}(x)}\,d\mu_I= \ton{\frac{s}{s+d}}^{-k}\theta_{s+d}(x)\, .
\end{align}
To apply this in our context, let us note for each $x_i$ in our covering, that by our construction of $U_+$ there must exist $y_i\in B_{r_i}(x_i)$ such that $\theta_{\eta(R-2)r_i}(y_i)\geq \theta_{\eta r_i}(y_i)\geq E-\eta$.  Let us now apply \eqref{e:covering:energy_bound:1} to obtain
\begin{align}
\theta_{\eta Rr_i}(z)\geq \ton{\frac{\eta Rr_i}{\eta (R-2)r_i}}^{-k}\theta_{\eta (R-2)r_i}(y_i)\geq \ton{\frac{R}{R-2}}^{-k}\big(E-\eta\big)\, . 
\end{align}
If $R=R(n,\Lambda,\eta')>0$ and $\eta\leq \eta(n,\Lambda,\eta')$, then we obtain from this the claimed estimate.
\end{proof}

In words, the above lemma is telling us that even though we have no reasonable control over the size of $\theta_{r_i}(x_i)$, after we go up a controlled number of scales we can again assume that the mass is again close to $E$.\\

In the last lemma the proof was based on continuity estimates on the mass $\theta$.  In the next lemma we wish to show an improved version of this continuity under appearance of symmetry.  Precisely:\\

\begin{lemma}[Improved Continuity of $\theta$]\label{l:covering:improved_continuity}
Let $I^m\subseteq B_2$ be an integral varifold satisfying the bounds \eqref{e:manifold_bounds}, the mean curvature bound \eqref{e:mean_curvature_bound}, and the mass bound $\mu_I(B_4(p))\leq \Lambda$.  Then for $0<\tau,\beta,\gamma<1$ there exists $\delta(n,\Lambda,\gamma,\beta,\tau)>0$ such that if there exists $x_1,\ldots,x_k\in I\cap B_1(p)$ which are $\tau$-independent at $x_0$ with $|\theta_{3}(x_j)-\theta_{\delta}(x_j)|<\delta$, 
and if $V^k$ is the $k$-dimensional affine subspace $x_0+\operatorname{span}\cur{x_1-x_0,\ldots,x_k-x_0}$, then for all $x,y\in I\cap B_1(p)\cap B_{\delta}(V^k)$ and $10^{-4}\beta \leq s\leq 1$ we have that $|\theta_{s}(x)-\theta_{s}(y)|<\gamma$.
\end{lemma}
\begin{proof}
The proof is by contradiction.  Thus, imagine no such $\delta$ exists.  Then there exists a sequence of integral varifolds $I_i$ on $B_4(p_i)$ satisfying $I_i(B_4(p_i))\leq \Lambda$ and such that
\begin{enumerate}
\item there exists $x_{i,1},\ldots,x_{i,k}\in B_1(p_i)$ which are $\tau$-independent at $x_{i,0}$ with $\abs{\theta^{I_i}_{\delta_i}(x_{i,j})-\theta^{I_i}_3(x_{i,j})}<\delta_i\to 0$,
\end{enumerate}
however we have that there exists $x_i,y_i\in I\cap B_1(p)\cap B_{\delta_i}(V^k_i)$ and $10^{-4}\beta\leq s_i\leq 1$ such that $\abs{\theta^{I_i}_{s_i}(x_i)-\theta^{I_i}_{s_i}(y_i)}\geq \gamma$.  Since $\partial I\cap B_2(p)=\emptyset$, the masses are uniformly bounded and the bound on the mean curvature $\delta_i$ converges to zero, we may apply Allard's compactness theorem and obtain a converging subsequence (where for convenience we will not change indexes) $I_i\to I$ as well as the collection
\begin{align}
&x_{i,j}\to x_j\in B_1(0^n)\, ,\notag\\
&x_i\to x,\, y_i\to y\in V=x_0+\text{span}\{x_1-x_0,\ldots,x_k-x_0\}\, ,\notag\\
&s_i\to s\, .
\end{align}
In particular, we have in this limit that
\begin{align}\label{e:covering:improved_cont:1}
&\big|\theta^I_{0}(x_j)-\theta^I_3(x_j)\big|=0\, ,\notag\\
&\big|\theta^I_s(x)-\theta^I_s(y)\big|\geq \gamma\, .
\end{align}
However, we have by theorem \ref{t:quantitative_0_symmetry} and the standard cone splitting of theorem \ref{t:con_splitting} we have that $I$ is $k$-symmetric with respect to the $k$-plane $V$ on $B_1$.  In particular, $I$ is invariant under translation by elements of $V$.  However, this is a contradiction to \eqref{e:covering:improved_cont:1}, and therefore we have proved the result.

\end{proof}
\vspace{.5 cm}

Now the first goal is to partition $U_+$ into a finite collection, each of which will have a few more manageable properties than $U_+$ itself.  More precisely:

\begin{lemma}\label{l:covering:break_up_U_+}
For each $R\geq 5$ there exists $N(n,R)>1$ such that we can break up $U_+$ as a union
\begin{align}\label{e:covering:U^a_decomposition}
U_+ = \bigcup_{a=1}^N U^a_+=\bigcup_{a=1}^N \bigcup_{i\in J^a} \B{r_i}{x_i}\, ,
\end{align}
such that each $U^a_+$ has the following property:  if $i\in J^a$, then for any other $j\in J^a$ we have that if $x_j\in B_{Rr_i}(x_i)$, then $r_j<R^{-2}r_i$.
\end{lemma}
\begin{proof}
Let us recall that the balls in the collection $\{B_{r_i/5}(x_i)\}$ are pairwise disjoint.  In particular, given $R\geq 5$ if we fix a ball $B_{r_i}(x_i)$ then by the usual covering arguments there can be at most $N(n,R)$ ball centers $\{x_j\}_1^N\subset U_+\cap B_{R^3 r_i}(x_i)$ with the property that $r_j\geq r_i$. 
Indeed, if $\{x_j\}_1^N$ is such a collection of balls then we get 
\begin{align}
\omega_n (2R^3 r_i)^n = \Vol(B_{2R^3 r_i}(x_i))&\geq \sum_1^N \Vol(B_{r_j/5}(x_j))\geq N\omega_n (r_i/5)^n\, ,
\end{align}
which by rearranging gives the estimate $N\leq N(n,R)$ as claimed.\\

Now we wish to build our decomposition $U_+ =\bigcup_1^N U^a_+$, where $N$ is from the first paragraph.  We shall do this inductively, with the property that at each step of the inductive construction we will have that the sets $U^a_{+}$ will satisfy the desired property. In particular, for every $a$ and $i\in J^a$, if $j\in J^a$ is such that $x_j\in B_{R r_i}(x_i)$, then $r_j<R^{-2}r_i$.\\


Begin by letting each $J^a$ be empty. We are going to sort the points $\cur{x_i}_{i\in J}$ into the sets $J^a$ one at a time.  At each step let $i\in J\setminus \bigcup_{a=1}^N J^a$ be an index such that $r_i=\max_{} r_j$, where the max is taken over all indexes in $J\setminus \bigcup_{a=1}^N J^a$, i.e., over all indexes which haven't been sorted out yet.  Now let us consider the collection of ball centers $\{y_j\}_{j\in J'}$ such that $J'\subset J$, $x_i\in B_{R r_j}(y_j)$ and $r_i\leq r_j\leq R^{2}r_i$.  
Note that, by construction, either $y_j$ has already been sorted out in some $J^a$, or $r_j=r_i$. Now evidently $y_j\in \B{R^3 r_i}{x_i}$ for all $j\in J'$ and so by the first paragraph the cardinality of $J'$ is at most $N(n,R)$.   
In particular, there must be some $J^a$ such that $J^a\cap J'=\emptyset$.  Let us assign the index $i$ to this piece of the decomposition, so that $i\in J^a$ and $x_i\in U_a^+$.  Clearly the decomposition $\bigcup U^a_+$ still satisfies the inductive hypothesis after the addition of this point, and so this finishes the inductive step of construction.  Since at each stage we have chosen $x_i$ to have the maximum radius, this process will continue indefinitely to give the desired decomposition of $U_+$.


\end{proof}
\vspace{.5 cm}

Now with a decomposition fixed, let us consider for each $1\leq a\leq N$ the measures
\begin{align}
\mu^a \equiv \sum_{i\in J^a} \omega_k r_i^k \delta_{x_i}\, .
\end{align}

The following is a crucial point in our construction.  It tells us that each ball $B_{10r_i}(x_i)$ either has small $\mu^a$-volume, or the point $x_i$ must have large mass at scale $r_i$.  Precisely:

\begin{lemma}\label{l:covering:energy_density}
Let $\eta',\beta,D>0$ be fixed. There exists $R=R(n,\Lambda,\eta',D,\beta,\epsilon)>0$ such that if we consider the decomposition \eqref{e:covering:U^a_decomposition}, and if
\begin{enumerate}
\item 
      $\eta\leq \eta(n,\Lambda,\eta',\beta,\epsilon)$, $r_i<10^{-2}$,
\item we have $\mu^a(B_{10r_i}(x_i))\geq 2\omega_k r_i^k$,
\item for all ball centers $y_j\in A_{r_i/5,10r_i}(x_i)\cap U^a_+$ and for $s = 10^{-2n}D^{-1}\omega_n r_i$, we have that $\mu^a(B_{s}(y_j))\leq D s^k$,
\end{enumerate}
then we have that $\theta_{\beta r_i/10}(x_i)\geq E-\eta'$.
\end{lemma}

\begin{proof}
Let us begin by choosing $\eta''<<\eta'< \eta'(n,\Lambda,\beta,\tau,\epsilon)$, which will be fixed later in the proof, and let us also define $\tau\equiv 10^{-2n}D^{-1}\omega_n$.  Let $\delta(n,\Lambda,\beta,\eta',\tau)$ be from lemma \ref{l:covering:improved_continuity} so that the conclusions hold with $10^{-1}\eta'$.
Now throughout we will assume $\eta''<\delta$. 
We will also choose $R=R(n,\Lambda,\eta'',D,\epsilon)>\max\{\tau^{-1},\delta^{-1},\delta'^{-1}\}$ so that lemma \ref{l:covering:energy_bound} is satisfied with $\eta''$. \\

Since it will be useful later, let us first observe that $r_i\geq R r$.  Indeed, if not then for each ball center $y_j\in A_{r_i/5,10r_i}(x_i)\cap U^a_+$ we would have $r\leq r_j<R^{-2}r_i<r$.  This tells us that there can be no ball centers in $A_{r_i/5,10r_i}(x_i)\cap U^a_+$.  However, by our volume assumption we have that
\begin{align}
\mu^a\big(A_{r_i/5,10r_i}(x_i)\big) = \mu^a\big(B_{10r_i}(x_i)\big)-\mu^a\big( B_{r_i/5}(x_i)\big)\geq 2\omega_k r_i^k-\omega_k r_i^k = \omega_k r_i^k\, ,
\end{align}
which contradicts this.  Therefore we must have that $r_i\geq R r$.\\

Now our first real claim is that under the assumptions of the lemma there exists ball centers $y_0,\ldots,y_{k}\in U^a_+\cap A_{r_i/5,10r_i}(x_i)$ which are $\tau r_i$-independent in the sense of Definition \ref{d:independent_points}.  Indeed, assume this is not the case, then we can find a $k-1$-plane $V^{k-1}$ such that
\begin{align}
\cur{y_i}_{i\in J^a}\cap A_{r_i/5,10r_i}(x_i) \subseteq B_{\tau r_i} \ton{V^{k-1}}\, .
\end{align}
In particular, by covering $ \cur{y_i}_{i\in J^a} \cap B_{\tau r_i} \ton{V}\cap B_{10r_i}(x_i)$ by $C(n)\tau^{1-k}< 10^{2n} \tau^{1-k}$ balls of radius $\tau r_i$ centered on $\cur{y_i}_{i\in J^a}$, and using our assumption that $\mu^a_+(B_{\tau r_i}(y_i))\leq D \tau^k r_i^k$, we are then able to conclude the estimate
\begin{align}
\mu^a\Big(A_{r_i/5,10r_i}(x_i)\Big) \leq \mu^a\Big(B_{\tau r_i} \ton{V}\cap B_{10r_i}(x_i)\Big) \leq 10^n D \tau r_i^k<\omega_n r_i^k\, .
\end{align}
On the other hand, our volume assumption guarantees that
\begin{align}\label{e:covering:energy_density:1}
\mu^a\big(A_{r_i/5,10r_i}(x_i)\big) = \mu^a\big(B_{10r_i}(x_i)\big)-\mu^a\big( B_{r_i/5}(x_i)\big)\geq 2\omega_k r_i^k-\omega_k r_i^k\geq \omega_k r_i^k\, ,
\end{align}
which leads to a contradiction.  Therefore there must exist $k+1$ ball centers $y_0,\ldots,y_{k}\in A_{r_i/5,10r_i}(x_i)$ which are $\tau r_i$-independent points, as claimed.\\

Let us now remark on the main consequences of the existence of these $k+1$ points.  Note first that for each $y_j$ we have that $\theta_{\eta R^{-1}r_i}(y_j)>E-\eta''$, since by the construction of $U^a_+$ we have that $r_{j}\leq R^{-2}r_i$, and therefore we can apply lemma \ref{l:covering:energy_bound}.  Thus we have $k+1$ points in $B_{10r_i}(x_i)$ which are $\tau r_i$-independent, and whose mass densities are $\eta''$-pinched.  To exploit this, let us first apply lemma \ref{l:covering:improved_continuity} in order to conclude that for each $x\in B_\delta(V)$ we have
\begin{align}
\theta_{\beta r_i/10}(x)\geq \theta_{\beta r_i/10}(y_j)-|\theta_{\beta r_i/10}(x)-\theta_{\beta r_i/10}(y_j)|\geq E-\eta''-10^{-1}\eta'>E-\eta'\, .
\end{align}
In particular, if we assume that $x_i$ is such that $\theta_{\beta r_i/10}(x_i)< E-\eta'$, then we must have that $r_i^{-1}d(x_i,V)\geq \delta =\delta(n,\Lambda,\eta')$.  \\

Therefore, let us now assume $\theta_{\beta r_i/10}(x_i)< E-\eta'$, and thus to prove the lemma we wish to find a contradiction.  To accomplish this notice that we have our $k+1$ points $y_0,\ldots,y_k\in B_{10r_i}$ which are $\tau r_i$-independent and for which $|\theta_{20r_i}(y_j)-\theta_{\eta R^{-1}r_i}(y_j)|<\eta''$.  Therefore by applying the cone splitting of theorem \ref{t:con_splitting} we have for each $\epsilon'>0$ that if $\eta''\leq \eta''(n,\Lambda,\epsilon')$ then $B_{10r_i}(x_i)$ is $(k,\epsilon')$-symmetric with respect to the $k$-plane $V^k$.  However, since $d(x_i,V)>\delta r_i$, we have by theorem \ref{t:quant_dim_red} that if $\epsilon'\leq \epsilon'(n,\Lambda,\epsilon)$ then there exists some $\tau'=\tau'(n,\Lambda,\epsilon)$ such that $B_{\tau' r_i}(x_i)$ is $(k+1,\epsilon)$-symmetric.  However, we can assume after a further increase that $R=R(n,\Lambda,D,\epsilon)>4\tau'^{-1}$, and thus we have that $\tau'r_i>4R^{-1} r_i>4r$.  This contradicts that $x_i\in S^k_{\epsilon,r/100}$, and thus we have contradicted that $\theta_{\beta r_i/10}(x_i)< E-\eta'$, which proves the lemma.
\end{proof}
\vspace{.5cm}

\subsection{\texorpdfstring{Estimating $U_+$ in lemma \ref{l:covering}}{Estimating U in lemma \ref{l:covering}}}\label{ss:covering_U_+}

Now we proceed to finish the proof of lemma \ref{l:covering} by estimating the set $U_+$.  First let us pick $D'=D'(n)\equiv 2^{16n}D(n)$, where $D(n)$ is from theorem \ref{t:reifenberg_W1p_discrete}.  For some $\eta',\beta$ fixed depending only on $n,\Lambda,\epsilon$, we can choose $R$ as in lemma \ref{l:covering:energy_density}. It is then enough to estimate each of the sets $U^a_+$, as there are at most $N=N(n,\Lambda,\epsilon,\eta',\beta)$ pieces to the decomposition.  Thus we will fix a set $U^a_+$ and focus on estimating the content of this set.  Let us begin by observing that if $r>0$ then we have the lower bound $r_i\geq r$.  Otherwise, let us fix any $r>0$ and restrict ourselves to the collection of balls in $U^a_+$ with $r_i\geq r$.  The estimates we will prove in the end will be independent of $r$, and thus by letting $r\to 0$ we will obtain estimates on all of $U^a_+$.\\

Now let us make the precise statement we will prove in this subsection.  Namely, consider any of ball centers $\cur{x_i}_{i\in J^a}$ and any radius $2^{-4}r_i\leq \er_\alpha\leq 2^{-6}$, where $\er_\alpha=2^{-\alpha}$.  Then we will show that
\begin{align}\label{e:covering:U_+:1}
\mu^a\big(B_{\er_\alpha}(x_i)\big)\leq 2^{8n}D(n) \er_\alpha^k\, .
\end{align}
Let us observe that once we have proved \eqref{e:covering:U_+:1} then we have finished the proof of the Covering lemma, as a simple covering argument gives us the the estimate
\begin{align}\label{e:covering:U_+:1'}
\sum r_i^k = \mu^a\big(B_2(x_i)\big)\leq C(n)\, . \\ \notag
\end{align}

We prove \eqref{e:covering:U_+:1} inductively on $\alpha$.  To begin notice that for each $x_i$ if $\alpha$ is the largest integer such that $2^{-4}r_i\leq \er_\alpha$, then the statement clearly holds, by the definition of the measure $\mu^a$.  In fact, we can go further than this.  For each $x_i$, let $r'_i\in [r_i/5,10r_i]$ be the largest radius such that for all $r_i/5\leq s< r'_i$ we have
\begin{align}
\mu^a\big( B_s(x_i)\big)\leq 2\omega_k (5s)^k\, .
\end{align}
In particular, we certainly have the much weaker estimate $\mu^a\big( B_{s}(x_i)\big)\leq 2^{8n}D(n) s^k$, and hence \eqref{e:covering:U_+:1} is also satisfied for all $2^{-4}r_i\leq \er_\alpha \leq r'_i$.  Notice that we then also have the estimate
\begin{align}
\omega_k r_i^k\leq \mu^a\big(B_{r_i/8}(x_i)\big)\leq \mu^a\big( B_{r'_i}(x_i)\big)\leq 2\omega_k\big(5 r'_i\big)^k\, .
\end{align}

Now let us focus on proving the inductive step of \eqref{e:covering:U_+:1}.  Namely, assume $\alpha$ is such that for all $x_i$ with $2^{-4}r_i\leq \er_{\alpha+1}<2$ we have that \eqref{e:covering:U_+:1} holds.  Then we want to prove that the same estimate holds for $\er_\alpha$.  Let us begin by seeing that a weak version of \eqref{e:covering:U_+:1} holds.  
Namely, for any index $i\in J^a$ and any radius $\er_{\alpha}\leq s\leq 8\er_\alpha$, by covering $B_s(x_i)$ by at most $2^{8n}$ balls $\{B_{\er_{\alpha+1}}(y_j)\}$ of radius $\er_{\alpha+1}$ we have the weak estimate
\begin{align}\label{e:covering:U_+:2}
\mu^a\big(B_s(x_i)\big) \leq \sum \mu^a\big(B_{\er_{\alpha+1}}(y_j)\big)\leq D'(n) s^k\, ,
\end{align}
where of course $D'(n)>>2^{8n}D(n)$.\\

To improve on this, let us fix an $i\in J^a$ and the relative ball center $x_i\in U^a_+$ with $2^{-4}r_i\leq \er_\alpha$.  Now let $\{x_j\}_{j\in J}= \cur{x_j}_{j\in J^a}\cap B_{\er_{\alpha}}(x_i)$ be the collection of ball centers in $B_{\er_{\alpha}}(x_i)$.  Notice first that if $r'_j>2\er_\alpha$ for any of the ball centers $\{x_j\}$, then we can estimate
\begin{align}
\mu^a\big(B_{\er_{\alpha}}(x_i)\big)\leq \mu^a\big(B_{2\er_{\alpha}}(x_j)\big) \leq 2\omega_k\big(10 \er_{\alpha}\big)^k\leq 2^{8n}D(n) \er_\alpha^k\, ,
\end{align}
so that we may fairly assume $r'_j\leq 2\er_\alpha$ for every $j\in J$.  Now for each ball $B_{r'_j}(x_j)$ let us define a new ball $B_{\bar r_j}(y_j)$ which is roughly equivalent, but will have some additional useful properties needed to apply the discrete Reifenberg. 
Namely, for a given ball $B_{r'_j}(x_j)$, let us consider the two options $r'_j<10r_j$ or $r'_j=10r_j$.  If $r'_j<10r_j$, then we let $y_j\equiv x_j$ with $\bar r_j\equiv r_j'$.  In this case we must have that $\mu^a(B_{s}(x_j))>2\omega_k \big(5s\big)^k$ for some $s$ arbitrarily close to $\bar r_j$, and thus we can apply lemma \ref{l:covering:energy_density} in order to conclude that $\theta_{\beta \bar r_j/2}(y_j)\geq E-\eta'$.  
In the case when $r'_j=10r_j$ is maximal, let $y_j\in B_{r_j}(x_j)$ be a point such that $\theta_{\eta r_j}(y_j)=E-\eta$, such a point exists by the definition of $r_j$, and let $\bar r_j\equiv 9r_j$.  In either case we then have the estimates
\begin{align}\label{e:covering:U_+:3}
&\theta_{\bar r_j/8}(y_j)\geq E-\eta'\, ,\notag\\
&\omega_k 10^{-k} \bar r_j^k\leq \mu^a\big(B_{\bar r_j/8}(y_j)\big)\leq \mu^a\big( B_{\bar r_j}(y_j)\big)\leq 2\omega_k (5r'_j)^k \leq 10^k \omega_k \bar r_i^k\, ,\notag\\
& x_j\in \B {\bar r_j/5}{y_j}\, .
\end{align}

Since $\operatorname{supp}(\mu_a)\cap \B {\er_\alpha}{x_i}\subseteq \bigcup B_{\bar r_j/5}(y_j)$, we can choose a Vitali subcovering of the support such that 
\begin{align}
\operatorname{supp}(\mu_a)\cap B_{\er_{\alpha}}(x_i)\subseteq \bigcup B_{\bar r_j}(y_j)\, ,
\end{align}
such that $\{B_{\bar r_j/5}(y_j)\}$ are disjoint, where we are now being loose on notation and referring to $\{y_j\}_{j\in \bar J}$ as the ball centers from this subcovering.  Let us now consider the measure
\begin{align}
\mu' \equiv \sum_{j\in \bar J} \omega_k \Big(\frac{\bar r_j}{10}\Big)^k \delta_{y_j}\, .
\end{align}
That is, we have associated to the disjoint collection $\{B_{\bar r_j/10}(y_j)\}$ the natural measure.  Our goal is to prove that
\begin{align}\label{e:covering:U_+:4}
\mu'\big( B_{\er_\alpha}(x_i)\big)\leq D(n) \er_\alpha^k\, .
\end{align}
Let us observe that if we prove \eqref{e:covering:U_+:4} then we are done.  Indeed, using \eqref{e:covering:U_+:3} we can estimate
\begin{align}
\mu^a\big(B_{\er_\alpha}(x_i)\big)\leq \sum \mu^a\big(B_{\bar r_j}(y_j)\big)\leq 10^k \omega_k\sum \bar r_j^k = 10^{2k}\mu'\big(B_{\er_\alpha}(x_i)\big)\leq 2^{8n} D(n) \er_\alpha^k\, ,
\end{align}
which would finish the proof of \eqref{e:covering:U_+:1} and therefore the lemma.\\

Thus let us concentrate on proving \eqref{e:covering:U_+:4}.  We will want to apply the discrete Reifenberg in this case to the measure $\mu'$.  Let us begin by proving a weak version of \eqref{e:covering:U_+:4}.  Namely, for any ball center $y_j$ from our subcovering and radius $\bar r_j<s\leq 4\er_\alpha$ let us consider the set $\{z_\ell\}=\cur{y_t}_{t\in \bar J}\cap B_s(y_j)$ of ball centers inside $B_s(y_j)$.  Since the balls $\{B_{\bar r_k/5}(z_\ell)\}$ are disjoint we have that $\bar r_k\leq 8s$.  Using this, \eqref{e:covering:U_+:2}, and \eqref{e:covering:U_+:3} we can estimate
\begin{align}\label{e:covering:U_+:5}
\mu'\big(B_s(y_j)\big) = \sum_{z_\ell\in B_s(y_j)} \omega_k 10^{-k}\bar r_k^k\leq C(n)\sum_{z_\ell\in B_s(y_j)}\mu^a(B_{\bar r_k/8}(z_\ell))\leq C(n)\mu^a(B_{2s}(y_j))\leq C(n) s^k\, ,
\end{align}
where of course $C(n)>>2^{8n}D(n)$. \\ 

Now let us finish the proof of \eqref{e:covering:U_+:4}.  Thus let us pick a ball center $y_j\in B_{\er_\alpha}(x_i)$ and a radius $s<4\er_\alpha$.  If $\mu'(B_s(y_j))\leq \epsilon_n  s^k$ then $D_{\mu'}(y_j,s)\equiv 0$ by definition, and if $s\leq \bar r_j/5$ then $D_{\mu'}(y_j,s)\equiv 0$, since the support of $\mu'$ in $B_{\bar r_i/5}(y_j)$ contains at most one point and thus is precisely contained in a $k$-dimensional subspace. In the case when $s>\bar r_i/5$ and $\mu(B_s(y_j))> \epsilon_n s^k$, we want to apply the estimates in theorem \ref{t:best_approximation}. In order to do so, we first remark that by picking $\beta\leq \beta_0(n,\Lambda,\beta')$ sufficiently small, since $\theta_{1}(y_j)-\theta_{\beta \bar r_j}(y_j)\leq \eta'$, we can apply theorem \ref{t:quantitative_0_symmetry} in order to prove that $\B {4s}{y}$ is $(0,\beta')$-symmetric, with $\beta'=\beta'(n,\Lambda,\epsilon)$, and in turn we obtain by theorem \ref{t:best_approximation} that
\begin{align}
D_{\mu'}(y_j,s) \leq C(n,\Lambda,\epsilon)s^{-k}\int_{B_s(y)} W_s(z)\,d\mu'(z)\, .
\end{align}
Note that $\B {s}{y_j}$ is not $(k+1,\epsilon)$-symmetric since $y_j\in S^k_{\epsilon,r/100}$ and $s\geq \bar r_j/5\geq r/100$.

By applying this to all $\bar r_j/5 <t\leq s$ we can estimate 
\begin{align}
s^{-k}\int_{B_s(x)} D_{\mu'}(y,t)\,d\mu'(y)&\leq Cs^{-k}\int_{B_s(x)}t^{-k}\int_{B_t(y)} W_t(z)\,d\mu'(z)\,d\mu'(y)\notag\\
&= Cs^{-k}t^{-k}\int_{B_{2s}(x)} \mu'(B_t(z)) W_t(z)\,d\mu'(z)\notag\\
&\leq Cs^{-k}\int_{B_{2s}(x)}W_t(z)\,d\mu'(z)\, ,
\end{align}
where we have used our estimate on $\mu'(B_t(y))$ from \eqref{e:covering:U_+:5} in the last line.  Let us now consider the case when $\bar r_j/5<t=\er_\beta\leq s\leq 2\er_{\alpha}$.  Then we can sum to obtain:
\begin{align}
\sum_{\er_\beta\leq s} s^{-k}\int_{B_s(x)} D_{\mu'}(y,\er_\beta)\,d\mu'(y) &\leq C\sum_{r'_y\leq \er_\beta\leq s} s^{-k}\int_{B_{2s}(x)}W_{\er_\beta}(y)\,d\mu'(y)\notag\\
&=C s^{-k}\int_{B_{2s}(x)}\sum_{\bar r_y\leq \er_\beta\leq s}W_{\er_\beta}(y)\,d\mu'(y)\notag\\
&\leq C\, s^{-k}\int_{B_{2s}(x)}\big|\theta_{8s}(y)-\theta_{\bar r_y}(y)\big| \,d\mu'(y)\leq C(n,\Lambda,\epsilon)\eta'\, ,
\end{align}
where we are using \eqref{e:covering:U_+:3} in the last line in order to see that $\big|\theta_s(y)-\theta_{\bar r_y}(y)\big|\leq \eta'$.  Now let us choose $\eta'\leq \eta'(n,\Lambda,\epsilon)$ such that
\begin{align}\label{e:covering:U_r:2c}
\sum_{\er_\beta\leq s} s^{-k}\int_{B_{2s}(x)} D_{\mu'}(y,\er_\beta)\,d\mu'(y) &\leq \delta^2\, ,
\end{align}
where $\delta$ is chosen from the discrete rectifiable-Reifenberg of theorem \ref{t:reifenberg_W1p_discrete}.  Since the estimate \eqref{e:covering:U_r:2c} holds for all $B_{2s}\subseteq B_{2\er_{\alpha}}(x)$, we can therefore apply theorem \ref{t:reifenberg_W1p_discrete} to conclude the estimate
\begin{align}
\mu(B_{\er_{\alpha}}(x'_i))\leq D(n) \er_{\alpha}^k\, .
\end{align}
This finishes the proof of \eqref{e:covering:U_+:3} , and hence the proof of lemma \ref{l:covering}. \hfill $\square$
\vspace{1cm}

\section{Proof of Main theorem's for Integral Varifolds with Bounded Mean Curvature}\label{s:main_theorem_stationary_proofs}

In this section we prove the main theorems of the paper concerning integral varifolds with bounded mean curvature.  With the tools of Sections \ref{s:biLipschitz_reifenberg}, \ref{s:best_approx}, and \ref{s:covering_main} developed, we will at this stage mainly be applying the covering of lemma \ref{l:covering} iteratively to arrive at the estimates.  When this is done carefully, we can combine the covering lemma with the cone splitting in order to check that for $k$-a.e. $x\in S^k_\epsilon$ there exists a unique $k$-dimensional subspace $V^k\subseteq T_xM$ such that every tangent cone of $I$ at $x$ is $k$-symmetric with respect to $V$.\\

\begin{remark}\label{rem_cover}
For the proofs of the theorems of this section let us make the following remark.  For any $\delta>0$ we can cover $B_1(p)$ by a collection of balls
\begin{align}\label{e:main_theorem_stationary_proofs:1}
B_1(p)\subseteq \bigcup_1^N B_{M^{-1}\delta}(p_i)\, ,
\end{align}
where $N\leq C(n)M^n\delta^{-n}$.  Thus if $\delta=\delta(n,\Lambda,\epsilon)$, $M=K+H$, and we can analyze each such ball, we can conclude from this estimates on all of $B_1(p)$.  In particular, by rescaling $B_{M^{-1}\delta}(p_i)\to B_1(p_i)$, we see that we can assume in our analysis that $K+H<\delta$ without any loss of generality.  We shall do this throughout this section.  \\ 
\end{remark}

\subsection{Proof of theorem \ref{t:main_quant_strat_stationary}}\label{ss:proof_thm_main_quant_strat_stationary}

Let $I^m\subseteq B_2$ be an integral varifold satisfying the bounds \eqref{e:manifold_bounds}, the mean curvature bound \eqref{e:mean_curvature_bound}, and the mass bound $\mu_I(B_2(p))\leq \Lambda$.  With $\epsilon,r>0$ fixed, let us choose $\eta(n,\Lambda, \epsilon)>0$ and $\delta(n,\Lambda,\epsilon)>0$ as in lemma \ref{l:covering}.  By Remark \ref{rem_cover}, we see that we can assume that $K+H<\delta$, which we will do for the remainder of the proof.\\ 

Now let us begin by first considering an arbitrary ball $B_s(x)$ with $x\in B_1(p)$ and $r<s\leq 1$, potentially quite small.  We will use lemma \ref{l:covering} in order to build a special covering of $B_s(x)$.  Let us define
\begin{align}
E_{x,s}\equiv \sup_{y\in B_s(x)}\theta_s(y)\, ,
\end{align}
and thus if we apply lemma \ref{l:covering} with $\eta(n,K_N,\Lambda,\epsilon)$ fixed to $B_s(x)$, then we can build a covering 
\begin{align}\label{e:proof_main_quant_strat:1}
S^k_{\epsilon,r}\cap B_s(x)\subseteq U_{r}\cup U_{+} = \bigcup B_{r}(x^r_{i})\cup \bigcup B_{r_{i}}(x_{i})\, ,
\end{align}
with $r_i>r$.  Let us recall that this covering satisfies the following:
\begin{enumerate}
\item[(a)] $r^{k-n}\Vol(B_r\, U_r)+\omega_k\sum r_i^k\leq C(n,\Lambda,\epsilon)\, s^k$.
\item[(b)] $\sup_{y\in B_{r_i}(x_i)}\theta_{\eta r_i}(y)\leq E_{x,s}-\eta$.
\end{enumerate}
As remarked above, since $\eta(n,K_N,\Lambda,\epsilon)$ is fixed, we can re-cover in a trivial way all balls $B_{r_i}(x_i)$ with smaller balls of radius $\eta r_i$ and obtain for this new covering that 
\begin{enumerate}
\item[(a')] $r^{k-n}\Vol(B_r\, U_r)+\omega_k\sum r_i^k\leq C(n,K_N,\Lambda,\epsilon)\, s^k$.
\item[(b')] $\sup_{y\in B_{r_i}(x_i)}\theta_{r_i}(y)\leq E_{x,s}-\eta$.
\end{enumerate}

Now that we have built our required covering on an arbitrary ball $B_s(x)$, let us use this iteratively to build our final covering of $S^k_{\epsilon,r}(I)$.  First, let us apply it to $B_1(p)$ is order to construct a covering
\begin{align}\label{eq_ref}
S^{k}_{\epsilon,r}(I)\subseteq U^1_{r}\cup U^1_+ = \bigcup B_{r}\ton{x^{r,1}_i}\cup\bigcup B_{r^1_i}\ton{x^1_i}\, ,
\end{align}
such that
\begin{align}
r^{k-n}\Vol\ton{B_r\ton{U^1_{r}}} + \omega_k\sum \ton{r^1_{i}}^{k} \leq C(n,K_N,\Lambda,\epsilon)\, ,
\end{align}
and with
\begin{align}
\sup_{y\in B_{r^1_i}(x_i^1)}\theta_{r^1_i}(y)\leq \Lambda-\eta\, .\\ \notag
\end{align}

Now let us tackle the following claim, which is our main iterative step in the proof:\\

{\bf Claim}:  For each $\ell$ there exists a constant $C_\ell(\ell,n,\Lambda,\epsilon)$ (recall that $K+H\leq \delta)$ and a covering 
\begin{align}
 S^{k}_{\epsilon,r}(I)\subseteq U^\ell_{r}\cup U^\ell_+ = \bigcup \B{r}{x^{r,\ell}_i} \cup\bigcup \B{r^\ell_i}{x^\ell_i}\, ,
\end{align}
with $r_i^\ell>r$, such that the following two properties hold:
\begin{align}\label{e:proof_main_quant_strat:2}
&r^{k-n}\Vol\ton{\B r {U^\ell_{r}}} + \omega_k\sum \big(r^\ell_{i}\big)^{k} \leq C_\ell(\ell,n,\Lambda,\epsilon)\, ,\notag\\
&\sup_{y\in B_{r^\ell_i}(x_i^\ell)}\theta_{r^\ell_i}(y)\leq \Lambda-\ell\cdot\eta\, .
\end{align}
\vspace{.5 cm}

To prove the claim let us first observe that we have shown this holds for $\ell=1$.  Thus let us assume we have proved the claim for some $\ell$, and determine from this how to build the covering for $\ell+1$ with some constant $C_{\ell+1}(\ell+1,n,\Lambda,\epsilon)$, which we will estimate explicitly.  \\

Thus with our covering determined at stage $\ell$, let us apply the covering of \eqref{e:proof_main_quant_strat:1} to each ball $\cur{\B{r^{\ell}_i}{x^\ell_i}}$ in order to obtain a covering
\begin{align}
S^k_{\epsilon,r}\cap B_{r^\ell_i}(x_i^\ell)\subseteq U_{i,r}\cup U_{i,+} = \bigcup_j \B{r}{x^r_{i,j}}\cup  \bigcup_j \B{r_{i,j}}{x_{i,j}}\, ,
\end{align}
such that
\begin{align}\label{e:proof_main_quant_strat:3}
&r^{k-n}\Vol\ton{\B r {U_{i,r}}} + \omega_k\sum_j (r_{i,j})^{k} \leq C(n,\Lambda,\epsilon)\big(r^\ell_i\big)^{k}\, ,\notag\\
&\sup_{y\in B_{r_{i,j}}(x_{i,j})}\theta_{r_{i,j}}(y)\leq \Lambda-(\ell+1)\eta\, .
\end{align}
Let us consider the sets
\begin{align}
&U^{\ell+1}_r \equiv U^\ell_r \cup\bigcup_i U_{i,r}\, ,\notag\\
&U^{\ell+1}_+ \equiv \bigcup_{i,j} B_{r_{i,j}}(x_{i,j})\, .
\end{align}
Notice that the second property of \eqref{e:proof_main_quant_strat:2} holds for $\ell+1$ by the construction, hence we are left analyzing the volume estimate of the first property.  Indeed, for this we combine our inductive hypothesis \eqref{e:proof_main_quant_strat:2} for $U^\ell$ and \eqref{e:proof_main_quant_strat:3} in order to estimate
\begin{align}
r^{k-n}\Vol\ton{\B r {U^{\ell+1}_{r}}} + \omega_k\sum_{i,j} (r_{i,j})^{k} &\leq r^{k-n}\Vol\ton{\B r {U^\ell_{r}}}+\sum_i\Big(r^{k-n}\Vol\ton{\B r {U_{i,r}}} + \omega_k\sum_j (r_{i,j})^{k} \Big)\notag\\
&\leq C_\ell + C\sum_i (r^\ell_i)^k\notag\\
&\leq C(n,\Lambda,\epsilon)\cdot C_\ell(\ell,n,\Lambda,\epsilon)\notag\\
&\equiv C_{\ell+1}\, .
\end{align}

Hence, we have proved that if the claim holds for some $\ell$ then the claim holds for $\ell+1$.  Since we have already shown the claim holds for $\ell=1$, we have therefore proved the claim for all $\ell$. \\

Now we can finish the proof.  Indeed, let us take $\ell = \lceil\eta^{-1}\Lambda\rceil = \ell(\eta,\Lambda)$.  Then if we apply the Claim to such an $\ell$, we must have by the second property of \eqref{e:proof_main_quant_strat:2} that
\begin{align}
U^\ell_+\equiv \emptyset\, ,
\end{align}
and therefore we have a covering
\begin{align}
S^k_{\epsilon,r}\subseteq U^\ell_r = \bigcup_i B_{r}(x_i)\, .
\end{align}
But in this case we have by \eqref{e:proof_main_quant_strat:2} that
\begin{align}
\Vol\ton{\B r {S^k_{\epsilon,r}(I)}}\leq \Vol\ton{ \B r { U^\ell_r}} \leq C(n,\Lambda,\epsilon) r^{n-k}\, ,
\end{align}
which proves the theorem.  \hfill $\square$
\vspace{1cm}

\subsection{Proof of theorem \ref{t:main_eps_stationary}}\label{ss:proof_thm_main_eps_stationary}

There are several pieces to theorem \ref{t:main_eps_stationary}.  To begin with, the volume estimate follows easily now that theorem \ref{t:main_quant_strat_stationary} has been proved.  That is, for each $r>0$ we have that
\begin{align}
S^k_\epsilon(I)\subseteq S^k_{\epsilon,r}(I)\, ,
\end{align}
and therefore we have the volume estimate
\begin{align}
\Vol\ton{\B r {S^k_{\epsilon}(I)}\cap \B 1 p}\leq \Vol\ton{ \B r {S^k_{\epsilon,r}(I)} \cap \B 1 p} \leq C(n,K,H,\Lambda,\epsilon) r^{n-k}\, .
\end{align}
In particular, this implies the much weaker Hausdorff measure estimate
\begin{align}
\lambda^k\ton{S^k_{\epsilon}(I)\cap \B 1 p} \leq C(n,K,H,\Lambda,\epsilon)\, ,
\end{align}
which proves the first part of the theorem.\\

Let us now focus on the rectifiability of $S^k_\epsilon$.  We consider the following claim, which is the $r=0$ version of the main Claim of theorem \ref{t:main_quant_strat_stationary}.  We will be applying lemma \ref{l:covering}, which requires $K+H<\delta$.  As in the proof of theorem \ref{t:main_quant_strat_stationary} we can just assume this without any loss, as we can cover $B_1(p)$ by a controlled number of balls of radius $M^{-1} \delta$, so that after rescaling we can analyze each of these balls with the desired curvature assumption. Thus let us consider the following:\\

{\bf Claim:}  If $K+H<\delta$, then for each $\ell$ there exists a covering $S^{k}_{\epsilon}(I)\subseteq U^\ell_{0}\cup U^\ell_+ =U^\ell_0 \bigcup B_{r^\ell_i}(x^\ell_i)$ such that
\begin{enumerate}
\item  $\lambda^k(U^\ell_0) + \omega_k\sum \big(r^\ell_{i}\big)^{k} \leq C_\ell(\ell,n,K,H,\Lambda,\epsilon)$.
\item $U^\ell_0$ is $k$-rectifiable.
\item $\sup_{y\in B_{r^\ell_i}(x_i^\ell)}\theta_{r^\ell_i}(y)\leq \Lambda-\ell\cdot\eta$
\end{enumerate}
\vspace{.5 cm}

The proof of the Claim follows essentially the same steps as those for the main Claim of theorem \ref{t:main_quant_strat_stationary}.  For base step $\ell=0$, we consider the decomposition $S^k_\epsilon \subseteq U^0_0\cup U^0_+$ where $U^0_0=\emptyset$ and $U^0_+=B_1(p)$.  

Now let us assume we have proved the claim for some $\ell$, then we wish to prove the claim for $\ell+1$.  Thus, let us consider the set $U^\ell_+$ from the previous covering step given by
\begin{align}
U^\ell_+ = \bigcup B_{r^\ell_i}(x^\ell_i)\, .
\end{align}
Now let us apply lemma \ref{l:covering} to each of the balls $B_{r^\ell_i}(x^\ell_i)$ in order to write
\begin{align}
S^k_\epsilon\cap B_{r^\ell_i}(x^\ell_i)\subseteq U_{i,0}\cup U_{i,+} = U_{i,0}\cup \bigcup_{j} B_{r_{i,j}}(x_{i,j})\, ,
\end{align}
with the following properties:
\begin{enumerate}
\item[(a)] $\lambda^k(U_{i,0})+\omega_k\sum_j r_{i,j}^k\leq C(n,K,H,\Lambda,\epsilon,p) (r^\ell_i)^k$,
\item[(b)] $\sup_{y\in B_{r_{i,j}}(x_{i,j})}\theta_{r_{i,j}}(y)\leq \Lambda-(\ell+1)\eta$,
\item[(c)] $U_{i,0}$ is $k$-rectifiable.
\end{enumerate}

Now let us define the sets
\begin{align}
&U^{\ell+1}_0 = \bigcup U_{i,0}\cup U^{\ell}_0\, ,\notag\\
&U^{\ell+1}_+ = \bigcup_{i,j} B_{r_{i,j}}(x_{i,j})\, .
\end{align}

Conditions $(2)$ and $(3)$ from the Claim are clearly satisfied.  We need only check condition $(1)$.  Using $(a)$ and the inductive hypothesis we can estimate that

\begin{align}
\lambda^k(U^{\ell+1}_0) + \omega_k\sum_{i,j} \big(r_{i,j}\big)^{k} &\leq \lambda^k(U^\ell_0)+ \sum_i\Big( \lambda^k(U_{i,0})+\omega_k\sum_j \big(r_{i,j}\big)^k\Big)\, ,\notag\\
&\leq C_\ell+C(n,K,H,\Lambda,\epsilon)\sum_i\big(r^\ell_i\big)^k\notag\\
&\leq C(n,K,H,\Lambda,\epsilon)\cdot C_\ell\notag\\
&\equiv C_{\ell+1}\, .
\end{align}
Thus, we have proved the inductive part of the claim, and thus the claim itself.
\vspace{1cm}

Let us now finish the proof that $S^k_\epsilon(I)$ is rectifiable.  So let us take $\ell = \lceil\eta^{-1}\Lambda\rceil = \ell(\eta,\Lambda)$.  Then if we apply the above Claim to $\ell$, then by the third property of the Claim we must have that
\begin{align}
U^\ell_+\equiv \emptyset\, ,
\end{align}
and therefore we have the covering
\begin{align}
S^k_{\epsilon}\subseteq U^\ell_0\, ,
\end{align}
where $U^\ell_0$ is $k$-rectifiable with the volume estimate $\lambda^k(U^\ell_0)\leq C$, which proves that $S^k_\epsilon$ is itself rectifiable.\\

Finally, we prove that for $k$ a.e. $x\in S^k_\epsilon$ there exists a $k$-dimensional subspace $V_x\subseteq T_xM$ such that {\it every} tangent cone at $x$ is $k$-symmetric with respect to $V_x$.  To see this we proceed as follows.  For each $\eta>0$ let us consider the finite decomposition
\begin{align}
S^k_\epsilon = \bigcup_{\alpha=0}^{\lceil \eta^{-1}\Lambda \rceil} W^{k,\alpha}_{\epsilon,\eta}\, ,
\end{align}
where by definition we have
\begin{align}
W^{k,\alpha}_{\epsilon,\eta}\equiv \big\{x\in S^k_\epsilon: \theta_0(x)\in \big[\alpha\eta,(\alpha+1)\eta\big)\big\}\, .
\end{align}
Note then that each $W^{k,\alpha}_{\epsilon,\eta}$ is $k$-rectifiable, and thus there exists a full measure subset $\tilde W^{k,\alpha}_{\epsilon,\eta}\subseteq W^{k,\alpha}_{\epsilon,\eta}$ such that for each $x\in\tilde W^{k,\alpha}_{\epsilon,\eta}$ the tangent cone of $W^{k,\alpha}_{\epsilon,\eta}$ exists and is a subspace $V_x\subseteq T_xM$.

Now let us consider such an $x\in \tilde W^{k,\alpha}_{\epsilon,\eta}$, and let $V^k_x$ be the tangent cone of $W^{k,\alpha}_{\epsilon,\eta}$ at $x$.  For all $r<<1$ sufficiently small we of course have $|\theta_r(x)-\theta_0(x)|<\eta$.  Thus, by the monotonicity and continuity of $\theta$ we have for all $r<<1$ sufficiently small and all $y\in W^{k,\alpha}_{\epsilon,\eta}\cap B_r(x)$ that $|\theta_r(y)-\theta_0(y)|<2\eta$.  In particular, by theorem \ref{t:quantitative_0_symmetry} we have for each $y\in W^{k,\alpha}_{\epsilon,\eta}\cap B_r(x)$ that $B_r(y)$ is $(0,\delta_\eta)$-symmetric, with $\delta_\eta\to 0$ as $\eta\to 0$.  Now let us recall the cone splitting of theorem \ref{t:con_splitting}.  Since the tangent cone at $x$ is $V^k_x$, for all $r$ sufficiently small we can find $k+1$ points $x_0,\ldots,x_k\in B_r(x)\cap W^{k,\alpha}_{\epsilon,\eta}$ which are $10^{-1}r$-independent, see Definition \ref{d:independent_points}, and for which $B_{2r}(x_j)$ are $(0,\delta_\eta)$-symmetric.  Thus, by the cone splitting of theorem \ref{t:con_splitting} we have that $B_r(x)$ is $(k,\delta_\eta)$-symmetric with respect to $V^k_x$ for all $r$ sufficiently small, where $\delta_\eta\to 0$ as $\eta\to 0$.  In particular, every tangent cone at $x$ is $(k,\delta_\eta)$-symmetric with respect to $V_x$, where $\delta_\eta\to 0$ as $\eta\to 0$.\\

Now let us consider the sets
\begin{align}
\tilde W^k_{\epsilon,\eta} \equiv \bigcup_\alpha \tilde W^{k,\alpha}_{\epsilon,\eta}\, .
\end{align}
So $\tilde W^k_{\epsilon,\eta}\subseteq S^k_\epsilon$ is a subset of full $k$-dimensional measure, and for every point $x\in \tilde W^k_{\epsilon,\eta}$ we have seen that {\it every} tangent cone of is $(k,\delta_\eta)$-symmetric with respect to some $V_x\subseteq T_xM$, where $\delta_\eta\to 0$ as $\eta\to 0$.  Finally let us define the set
\begin{align}
\tilde S^k_{\epsilon} \equiv \bigcap_j \tilde W^k_{\epsilon,j^{-1}}\, . 
\end{align}
This is a countable intersection of full measure sets, and thus $\tilde S^k_{\epsilon}\subseteq S^k_{\epsilon}$ is a full measure subset.  Further, we have for each $x\in \tilde S^k_{\epsilon}$ that {\it every} tangent cone must be $(k,\delta)$-symmetric with respect to some $V_x$, for all $\delta>0$.  In particular, every tangent cone at $x$ must be $(k,0)=k$-symmetric with respect to some $V_x$.  This finishes the proof of the theorem.  \hfill $\square$
\vspace{1cm}

\subsection{Proof of theorem \ref{t:main_stationary}}\label{ss:proof_thm_main_stationary}

Let us begin by observing the equality
\begin{align}\label{e:proof_thm_main_stationary:1}
S^k(I) = \bigcup_{\epsilon>0} S^k_{\epsilon}(I) = \bigcup_{\beta\in \dN} S^k_{2^{-\beta}}(I)\, .
\end{align}
Indeed, if $x\in S^k_\epsilon(I)$, then no tangent cone at $x$ can be $(k+1,\epsilon/2)$-symmetric, and in particular $k+1$-symmetric, and thus $x\in S^k(I)$.  This shows that $S^k_\epsilon(I)\subseteq S^k(I)$.  On the other hand, if $x\in S^k(I)$ then we claim there is some $\epsilon>0$ for which $x\in S^k_\epsilon(I)$.  Indeed, if this is not the case, then there exists $\epsilon_i\to 0$ and $r_i>0$ such that $B_{r_i}(x)$ is $(k+1,\epsilon_i)$-symmetric.  If $r_i\to 0$ then we can pass to a subsequence to find a tangent cone which is $k+1$-symmetric, which is a contradiction.  On the other hand, if $r_i>r>0$ then we see that $B_r(x)$ is itself $k+1$-symmetric, and in particular {\it every} tangent cone at $x$ is $k+1$-symmetric.  In either case we obtain a contradiction, and thus $x\in S^k_\epsilon(I)$ for some $\epsilon>0$.  Therefore we have proved \eqref{e:proof_thm_main_stationary:1}.\\

As a consequence, $S^k(I)$ is a countable union of $k$-rectifiable sets, and therefore is itself $k$-rectifiable.  On the other hand, theorem \ref{t:main_eps_stationary} tells us that for each $\beta\in\dN$ there exists a set $\tilde S^k_{2^{-\beta}}(I)\subseteq S^k_{2^{-\beta}}(I)$ of full measure such that
\begin{align}
\tilde S^k_{2^{-\beta}}\subseteq \big\{x: \, \exists\, V^k\subseteq T_xM\text{ s.t. every tangent cone at $x$ is $k$-symmetric wrt $V$}\big\}\, .
\end{align}
Hence, let us define 
\begin{align}
\tilde S^k(I)\equiv \bigcup \tilde S^k_{2^{-\beta}}(I)\, .
\end{align}
Then we still have that $\tilde S^k(I)$ has $k$-full measure in $S^k(I)$, and if $x\in \tilde S^k(I)$ then for some $\beta$ we have that $x\in \tilde S^k_{2^{-\beta}}$, which proves that there exists a subspace $V\subseteq T_xM$ such that every tangent cone at $x$ is $k$-symmetric with respect to $V$.  We have finished the proof of the theorem.  \hfill $\square$

\vspace{1cm}

\section{Proof of Main theorem's for Minimizing Hypersurfaces}\label{s:main_theorem_minimizing_proofs}

In this section we prove the main theorems of the paper concerning minimizing hypersurfaces.  That is, we finish the proofs of theorem \ref{t:main_min_finite_measure} and theorem \ref{t:main_min_weak_L7}. In fact, the proofs of these two results are almost identical, though the first relies on theorem \ref{t:main_eps_stationary} and the later on theorem \ref{t:main_quant_strat_stationary}.  However, for completeness sake we will include the details of both.\\

\subsection{Proof of theorem \ref{t:main_min_finite_measure}}

We wish to understand better the size of the singular set $\Sing\ton{I^{n-1}}$, where $I^{n-1}\subset M^n$, of a minimizing hypersurface.  Let us recall that the $\epsilon$-regularity of theorem \ref{t:eps_reg} tells us that if $I$ is minimizing, then there exists $\epsilon(n,K,H,\Lambda)>0$ 
with the property that if $x\in B_1(p)$ and $0<r<r(n,K,H,\Lambda)$ is such that $B_{2r}(x)$ is $(n-7,\epsilon)$-symmetric, then $r_I(x)\geq r$.  In particular, $x$ is a smooth point, and we have for $\epsilon(n,K,H,\Lambda)>0$ that
\begin{align}
\text{Sing}(I)\cap B_1(p)\subseteq S^{n-8}_\epsilon(I)\, .
\end{align}
Thus by theorem \ref{t:main_eps_stationary} there exists $C(n,K,H,\Lambda)>0$ such that for each $0<r<1$ we have
\begin{align}
\Vol\ton{\B r {\text{Sing}(I)}\cap B_1(p)}\leq \Vol\ton{\B r{S^{n-8}_\epsilon(I)} \cap B_1(p) }\leq C r^8\, .
\end{align}
This of course immediately implies, though of course is much stronger than, the Hausdorff measure estimate
\begin{align}
\lambda^{n-8}\big(\text{Sing}(I)\cap B_1(p)\big)\leq C\, ,
\end{align}
which finishes the proof of the first estimate in \eqref{e:main_min_finite_measure}. As a simple corollary of this and the uniform bound $\theta(x,r)\leq c\Lambda$ for all $x\in \B 1 p$ and $r<1$, we obtain also the second estimate in \eqref{e:main_min_finite_measure}. \hfill $\square$
\vspace{1cm}

\subsection{Proof of theorem \ref{t:main_min_weak_L7}}

We begin again by considering the $\epsilon$-regularity of theorem \ref{t:eps_reg}.  This tells us that if $I$ is a minimizing hypersurface, then there exists $\epsilon(n,K,H,\Lambda)>0$ 
with the property that if $x\in B_1(p)$ and $0<r<r(n,K,H,\Lambda)$ are such that $B_{2r}(x)$ is $(n-7,\epsilon)$-symmetric, then $r_I(x)\geq r$.  In particular, we have for such $\epsilon,r$ that
\begin{align}
\{x\in B_1(p): r_I(x)<r\}\subseteq S^{n-8}_{\epsilon,r}(I)\, .
\end{align}
Thus by theorem \ref{t:main_eps_stationary} there exists $C(n,K,H,\Lambda)>0$ such that for each $0<r<1$ we have
\begin{align}
\Vol\big(B_r\{x\in B_1(p): r_I(x)<r\}\big)\leq \Vol\ton{\B r{S^{n-8}_{\epsilon,r}(I)} \cap B_1(p) }\leq C r^8\, ,
\end{align}
which proves the second estimate of \eqref{e:main_min_weak_L7:1}.  To prove the first we observe that $|A|(x)\leq r_I(x)^{-1}$, while the third is a corollary of the bound $\theta(x,r)\leq c \Lambda$ for all $x\in \B 1 p$ and $r<1$. This concludes the proof of the theorem. \hfill $\square$

\section*{Acknowledgments}
We would like to thank the referees for their very precise comments on earlier versions of this article.

\bibliographystyle{aomalpha}
\bibliography{qstrat}

\end{document}